\documentclass[reqno]{amsart}

\usepackage{amsmath, amssymb, amsthm, amsfonts,bm, comment}
\usepackage{xcolor}

\input xy
\xyoption{all}

\usepackage{hyperref}

\usepackage{tikz} \usetikzlibrary{calc} \tikzset{>=latex} \usetikzlibrary{backgrounds}
\usepackage{tikz-cd}
\usepackage[all]{xy}
\usetikzlibrary{arrows}
\usetikzlibrary{decorations.pathreplacing,calligraphy}

\tikzset{
	wt/.style={circle, draw=black, fill=black!40!white, inner sep=2pt, outer sep=0pt, minimum size=5pt}, % a (grey) weight
	axiscolour/.style={blue}, % axis colour (obviously)
}

\numberwithin{equation}{section}

\theoremstyle{plain}
\newtheorem{theorem}{Theorem}[section]
\newtheorem{lemma}[theorem]{Lemma}
\newtheorem{prop}[theorem]{Proposition}
\newtheorem{cor}[theorem]{Corollary}

\newtheorem{assumption}[theorem]{Assumption}

\theoremstyle{definition}
\newtheorem{defn}[theorem]{Definition}

\newtheorem{remark}[theorem]{Remark}

\def\CC{\mathbb{C}}

\def\QQ{\mathbb{Q}}

\def\ZZ{\mathbb{Z}}

\newcommand\cC{\mathcal{C}}
\newcommand\cD{\mathcal{D}}
\newcommand\cE{\mathcal{E}}
\newcommand\cF{\mathcal{F}}
\newcommand\cG{\mathcal{G}}

\newcommand\cL{\mathcal{L}}

\newcommand\cP{\mathcal{P}}

\newcommand\cS{\mathcal{S}}

\newcommand\cW{\mathcal{W}}

\newcommand\cY{\mathcal{Y}}
\newcommand\cZ{\mathcal{Z}}

\newcommand\tilX{\widetilde{X}}

\newcommand\id{\textup{id}}

\newcommand\Hom{\textup{Hom}}
\newcommand\End{\textup{End}}

\renewcommand\sl{\mathfrak{sl}}
\newcommand\Vir{\textup{Vir}}
\newcommand\loc{\textup{loc}}

\newcommand\quash[1]{}
\newcommand\vac{\mathbf{1}}

\renewcommand\a\alpha
\renewcommand\b\beta
\newcommand\g\gamma
\renewcommand\d\delta
\newcommand\D\Delta

\title{Ribbon categories of weight modules for affine $\mathfrak{sl}_2$ at admissible levels}

\author{Thomas Creutzig}
\address{(Thomas Creutzig) Department Mathematik, Friedrich-Alexander Universit\"at Erlangen-N\"urnberg, 
91058 Erlangen, Germany}\email{thomas.creutzig@fau.de}
\author{Robert McRae}
\address{(Robert Mcrae) Yau Mathematical Sciences Center, Tsinghua University, Beijing 100084, China}\email{rhmcrae@tsinghua.edu.cn}
\author{Jinwei Yang}
\address{(Jinwei Yang) School of Mathematical Sciences, Shanghai Jiao Tong University, Shanghai 200240, China}
\email{jinwei2@sjtu.edu.cn}

\date{}


\begin{document}

\begin{abstract} 
We show that the braided tensor category of finitely-generated weight modules for the simple affine vertex operator algebra $L_k(\mathfrak{sl}_2)$ of $\mathfrak{sl}_2$ at any admissible level $k$ is rigid and hence a braided ribbon category. The proof uses a recent result of the first two authors with Shimizu and Yadav on embedding a braided Grothendieck-Verdier category $\mathcal{C}$ into the Drinfeld center of the category of modules for a suitable commutative algebra $A$ in $\mathcal{C}$, in situations where the braided tensor category of local $A$-modules is rigid. Here, the commutative algebra $A$ is Adamovi\'{c}'s inverse quantum Hamiltonian reduction of $L_k(\mathfrak{sl}_2)$, which is the simple rational Virasoro vertex operator algebra at central charge $1-\frac{6(k+1)^2}{k+2}$ tensored with a half-lattice conformal vertex algebra. As a corollary, we also show that the category of finitely-generated weight modules for the $N = 2$ super Virasoro vertex operator superalgebra at central charge $-6\ell-3$ is rigid for $\ell$ such that $(\ell+1)(k+2) = 1$. 
\end{abstract}

\maketitle

\tableofcontents

\section{Introduction}

Vertex operator algebras formalize in an algebraic way the notion of symmetry algebra of a two-dimensional conformal field theory. The axioms of a rational conformal field theory suggest that the representations of the associated vertex operator algebra should form a modular tensor category \cite{MS88}, and indeed Huang proved that the representation category of a strongly rational vertex operator algebra is a modular tensor category \cite{H3}. However, most vertex operator algebras are not rational. In particular, they usually have infinitely many inequivalent simple modules, as well as modules that are not completely reducible. Nonetheless, it is expected that under reasonably general conditions, representation categories of non-rational vertex operator algebras should still be braided ribbon categories, that is, rigid braided tensor categories with a ribbon twist.
Indeed, if a category of modules for a vertex operator algebra admits the monoidal category structure constructed by Huang, Lepowsky, and Zhang in \cite{HLZ1}-\cite{HLZ8}, then it automatically has a braiding and twist satisfying the balancing equation (see \cite[Section 12]{HLZ8} and \cite[Theorem 4.1]{H3}).
 But it is rather difficult in general to determine whether a given category of modules for a vertex operator algebra is a monoidal category, and if so, whether it is rigid.

 In the past, we have improved our understanding of vertex algebraic tensor categories, and their rigidity, by studying instructive examples.
 However, before the present work, together with the simultaneous independent papers \cite{Flor} and \cite{C2}, some of the prime examples of vertex operator algebras for logarithmic conformal field theory had not yet been fully understood. These are
 the simple affine vertex operator algebras $L_k(\sl_2)$ at Kac-Wakimoto admissible levels $k$ \cite{KW}, which correspond to fractional level $\mathfrak{sl}_2$ WZW models in physics. After extensive study in the mathematical physics literature, including the works \cite{Ga, LMRS, Ri1, Ri2, CR1, CR2}, it became clear that the correct category of $L_k(\sl_2)$-modules for conformal field theory is the category $\cC^{\mathrm{wt}}_k(\sl_2)$ of finitely-generated weight modules, since the characters of modules in this category are closed under modular transformations in a suitable sense. Based on these modular character transformations, the first author and Ridout conjectured a Verlinde formula for the Grothendieck fusion rules of $\cC^{\mathrm{wt}}_k(\sl_2)$ \cite{CR1, CR2}; this conjecture suggested that $\cC^{\mathrm{wt}}_k(\sl_2)$ should be a rigid braided tensor category.

 Objects of $\cC^{\mathrm{wt}}_k(\sl_2)$ consist of $L_k(\sl_2)$-modules on which the generator $h$ of the Cartan subalgebra of $\sl_2$ acts semisimply, such that each individual $h$-eigenspace decomposes into finite-dimensional generalized eigenspaces for the Virasoro $L_0$ operator, and such that the generalized $L_0$-eigenvalues on each $h$-eigenspace have a lower bound. This category is much larger than the Kazhdan-Lusztig category $KL_k(\sl_2)$ of ordinary $L_k(\sl_2)$-modules, which decompose into finite-dimensional $L_0$-eigenspaces. While $KL_k(\sl_2)$ is a braided fusion category \cite{CHY}, $\cC^{\mathrm{wt}}_k(\sl_2)$ is neither finite nor semisimple.
Other than some early works on classification of simple $L_k(\sl_2)$-modules and partial results on their fusion rules \cite{AM, DLM}, most rigorous mathematical results on $\cC^{\mathrm{wt}}_k(\sl_2)$ are rather recent \cite{A, RW, KR, ACK, C1}. Especially, the complete description of $\cC^{\mathrm{wt}}_k(\sl_2)$ as an abelian category \cite{ACK} and the existence of vertex algebraic tensor category structure on $\cC^{\mathrm{wt}}_k(\sl_2)$ \cite{C1} only appeared last year.

In the present work, we prove the following theorem using algebraic methods based on \cite{A, CMSY}; this result is also obtained in the simultaneous independent work \cite{Flor} using analytic methods: 
\begin{theorem}\label{thm:intro-main-thm}
The tensor category $\cC^{\mathrm{wt}}_k(\sl_2)$ of finitely-generated weight $L_k(\sl_2)$-modules is rigid and thus is a braided ribbon category for any admissible level $k=-2+\frac{u}{v}$, where $u,v\in\ZZ_{\geq 2}$ are relatively prime.
\end{theorem}

We do not need to consider non-negative integer admissible levels $k=-2+u$, since $\cC^{\mathrm{wt}}_{-2+u}(\sl_2)$ is just the modular tensor category of modules for the rational vertex operator algebra $L_{-2+u}(\sl_2)$. 
Besides rigidity, fusion rules for simple modules in $\cC^{\mathrm{wt}}_k(\sl_2)$ are also obtained in \cite{Flor}. More generally, the simultaneous work \cite{C2} proves (under some conditions) the Verlinde formula in logarithmic conformal field theory and applies it to $L_k(\sl_2)$ at admissible levels. 
As a consequence, one obtains the fusion rules for $\cC^{\mathrm{wt}}_k(\sl_2)$. 
With these works, it is fair to say that the category of finitely-generated weight $L_k(\sl_2)$-modules for $k$ admissible is now understood as both an abelian and a monoidal category. 

To further motivate Theorem \ref{thm:intro-main-thm} and discuss our proof, we recall some of the recent example-driven development of the theory of vertex algebraic tensor categories. The study of ordinary modules for simple affine vertex operator algebras $L_k(\mathfrak{g})$ at admissible levels $k$ \cite{CHY} and of so-called $C_1$-cofinite modules for the Virasoro algebra \cite{Vir} led to a general criterion for the category of $C_1$-cofinite modules for a vertex operator algebra to admit the braided tensor category structure of \cite{HLZ8}.
 See in particular \cite[Theorem~3.6]{CY}, and also \cite[Theorem~2.3]{Mc2}; the most general statement is as follows:
\begin{theorem}[\cite{CMOY}]\label{thm:vosa-btc}
Let $\cC_1(V)$ be the category of $C_1$-cofinite grading-restricted generalized modules for a vertex operator superalgebra $V$. If $\cC_1(V)$ is closed under contragredient modules, then $\cC_1(V)$ admits the braided tensor category structure of \cite{HLZ1}-\cite{HLZ8}. In particular, $\cC_1(V)$ is a braided tensor category if the following two conditions hold:
\begin{enumerate}
    \item[(a)] The contragredient $W'$ of any simple $C_1$-cofinite $V$-module $W$ is $C_1$-cofinite.
    \item[(b)] $\cC_1(V)$ is equal to the category of finite-length $C_1$-cofinite $V$-modules whose composition factors are $C_1$-cofinite.
\end{enumerate}
\end{theorem}

This theorem applies to the categories of $C_1$-cofinite modules for simple affine vertex operator algebras at all admissible and at least many non-admissible levels \cite{CHY, CY}, Virasoro vertex operator algebras at all central charges \cite{Vir}, the vertex operator superalgebra of $\mathfrak{gl}_{1\vert 1}$ \cite{CMY3}, the singlet vertex operator algebras \cite{CMY1, CMY2}, the universal affine vertex operator algebra of $\sl_2$ at admissible levels \cite{MY2}, and the $N=1$ super Virasoro vertex operator superalgebras at all central charges \cite{CMOY}.
But it does not apply to the category $\cC^{\mathrm{wt}}_k(\sl_2)$ of finitely-generated weight $L_k(\sl_2)$-modules at admissible levels, since the vast majority of modules in this category are not $C_1$-cofinite.
Nonetheless, Theorem \ref{thm:vosa-btc} applies to the weight module category of the Kazama-Suzuki dual \cite{KS} of $L_k(\sl_2)$, which is the simple $N=2$ super Virasoro algebra at central charge $\frac{3k}{k+2}$. The Kazama-Suzuki duality is good enough that vertex algebraic tensor category structure on $\cC^{\mathrm{wt}}_k(\sl_2)$ can then be inferred from that on the category of weight modules for the dual $N=2$ super Virasoro algebra \cite{C1}.

Now once one has tensor category structure on a category of modules for a non-rational vertex operator algebra, there are a few approaches for proving rigidity.
One method used in previous work is to exploit equivalences with known rigid categories. For example, the ordinary modules of simple admissible affine vertex operator algebras in types $ADE$ can be related to the modules of rational principal $W$-algebras \cite{C}. But one cannot expect such nice relations in general.
Another method uses explicit analysis of correlation functions to prove rigidity of one or more relatively easy modules in the category, which then generate all (or at least enough) of the category through fusion. This method works well in ``low-rank'' examples such as the $\beta\gamma$-ghost vertex operator algebra \cite{AW} (see also \cite{CMY5}), the singlet algebras \cite{CMY1,CMY2}, Virasoro vertex operator algebras \cite{MY-cp1-Vir}, and the affine vertex operator superalgebra of $\mathfrak{gl}_{1|1}$ \cite{CMY3}. But this method is also cumbersome since it requires analysis (such as deriving and solving ordinary differential equations) as well as good knowledge of fusion rules.
It is also doubtful whether this analytic method is feasible for more complicated vertex operator algebras. For example, the correlation functions for the modules needed to generate $\cC^{\mathrm{wt}}_k(\sl_2)$ through fusion seem to satisfy rather complicated third-order differential equations that might be difficult to solve explicitly.
Thus better methods for proving rigidity of vertex algebraic tensor categories are needed.

Here, we prove rigidity of $\cC^{\mathrm{wt}}_k(\sl_2)$ in Theorem \ref{thm:intro-main-thm} using a new method introduced in \cite{CMSY}. No analysis is required, and we only need one relatively easy fusion rule in $\cC^{\mathrm{wt}}_k(\sl_2)$.
The idea is to exploit the observation that
many vertex operator algebras $V$ admit a conformal embedding into a larger conformal vertex algebra $A$ whose representation theory is known, semisimple, and rigid. For example, the singlet algebras embed into a rank $1$ Heisenberg vertex operator algebra, and $L_k(\mathfrak{sl}_2)$ for non-integral admissible $k$ embeds into a rational Virasoro vertex operator algebra tensored with a half-lattice conformal vertex algebra \cite{A}. Formalizing this situation, we have a category $\cC$ of $V$-modules which we assume is a braided tensor category. Then $A$ has the structure of a commutative algebra in $\cC$ \cite{HKL}, and there is then a tensor category $\cC_A$ of modules for this commutative algebra and a braided tensor subcategory $\cC_A^{\loc}$ of local modules. 
The vertex algebraic tensor category of $A$-modules that lie in $\cC$ is isomorphic to $\cC_A^{\loc}$ \cite{CKM-ext}. It turns out that under suitable conditions, information about $\cC_A^{\loc}$ and $\cC_A$ gives important information about $\cC$, such as the Verlinde formula \cite{C2}, Kazhdan-Lusztig correspondences with quantum groups \cite{CLR}, and most importantly for the present work, rigidity \cite{CMSY}. 

Let $V$ be a vertex operator algebra and let $\cC$ be a braided tensor category of $V$-modules. Although $\cC$ is not automatically rigid, if it is closed under the contragredient modules of \cite{FHL}, then it has the weaker duality structure of a Grothendieck-Verdier category \cite{ALSW}. This structure is given by an anti-equivalence $D: \cC \rightarrow \cC$ such that $\Hom_\cC( - \boxtimes X, D\vac) \cong \Hom_\cC( - , DX)$, where $\vac$ is the tensor identity of $\cC$. In the vertex operator algebra setting, $\vac = V$ and $D$ is given by contragredient duals, $DX = X'$. A braided Grothendieck-Verdier category is called ribbon if it has a twist $\theta$ that satisfies the usual balancing equation and is compatible with the duality structure, in the sense that $\theta_{DX} = D\theta_X$ for any object $X$ in $\cC$. A vertex algebraic Grothendieck-Verdier category is always ribbon as long as the vertex operator algebra $V$ is $\ZZ$-graded. A Grothendieck-Verdier category is called an r-category if $D\vac \cong \vac$, which in the vertex operator algebra setting means that $V$ is isomorphic to its own contragredient. 
The first rigidity result of \cite{CMSY} that is important for our purposes here is the following; in this theorem, the condition $\theta_A^2 = \id_A$ in vertex operator algebra terms is that $A$ is $\frac{1}{2} \ZZ$-graded by conformal weights:
\begin{theorem}[\cite{CMSY}, Theorem 3.14]\label{thm:intro-CAloc-to-CA}
    Let $(\cC,D,\theta)$ be a locally finite abelian ribbon Grothendieck-Verdier category and let $A$ be a commutative algebra in $\cC$ such that $\theta_A^2=\id_A$. If the abelian braided monoidal category $\cC_A^\loc$ is rigid and every simple object of $\cC_A$ is an object of $\cC_A^\loc$, then $\cC_A$ is rigid.
\end{theorem}

We now want to use this result to infer rigidity of $\cC$. For this, the monoidal induction functor $\cF: \cC \rightarrow \cC_A$ given on objects by $X\mapsto A\boxtimes X$ is useful. 
Also, for an object $X$ in  an r-category $\cC$, let $e_X: DX\boxtimes X\rightarrow\vac$ be the image of $\id_{DX}$ under the isomorphism $\Hom_\cC(DX,DX)\rightarrow\Hom_\cC(DX\boxtimes X,\vac)$. In the setting of Theorem \ref{thm:intro-CAloc-to-CA}, let $\iota_A:\vac\rightarrow A$ be the unit morphism of the commutative algebra $A$, and let $c_{-,-}$ denote the braiding of $\cC$. Then the main result of \cite{CMSY} is the following:
\begin{theorem}[\cite{CMSY}, Theorem 3.21]\label{thm:intro-C-rigid}
Let $(\cC,D,\theta)$ be a locally finite abelian ribbon r-category, and let $A$ be a commutative algebra in $\cC$ such that $\theta_A^2=\id_A$ and the unit morphism $\iota_A:\vac\rightarrow A$ is injective. Also assume the following:
\begin{enumerate}
\item[(1)] $\cC_A^{\loc}$ is rigid and every simple object of $\cC_A$ is an object of $\cC_A^{\loc}$.

\item[(2)] For any simple object $X$ in $\cC$, the map $e_X: DX\boxtimes X\rightarrow\vac$ is surjective and any non-zero $\cC_A$-morphism from $\cF(DX)$ to $\cF(X)^*$ is an isomorphism.

\item[(3)] For any $\cC$-subobject $s: S\hookrightarrow A$ such that $S$ is not contained in $\mathrm{Im}\,\iota_A$, there exists an object $Z\in\cC$ such that $c_{Z,S}\neq c_{S,Z}^{-1}$ and the map $s\boxtimes\id_Z: S\boxtimes Z\rightarrow A\boxtimes Z$ is injective.

\end{enumerate}
Then $\cC$ is rigid.
\end{theorem}

In view of Theorem \ref{thm:intro-CAloc-to-CA}, condition (1) guarantees that $\cC_A$ is rigid, and therefore so is its Drinfeld center $\cZ(\cC_A)$. Conditions (2) and (3) are then sufficient to prove that the natural lift of $\cF$ to a braided monoidal functor $\cC\rightarrow\cZ(\cC_A)$ is fully faithful and commutes with the duality structures on $\cC$ and $\cZ(\cC_A)$. Then $\cC$ is rigid because it embeds into a rigid category in a way compatible with duality.

Of course now we need a way to verify the assumptions of Theorem \ref{thm:intro-C-rigid}, but we suggest that this verification is substantially simpler than previous rigidity arguments such as the study of correlation functions. 
As evidence for this claim, we prove Theorem \ref{thm:intro-main-thm} here
by checking the conditions of Theorem \ref{thm:intro-C-rigid} for $\cC=\cC^{\mathrm{wt}}_k(\sl_2)$, and we do this purely algebraically using fairly minimal knowledge of the fusion rules in $\cC^{\mathrm{wt}}_k(\sl_2)$.
In our proof, the commutative algebra $A$ is the previously mentioned tensor product of a rational Virasoro vertex operator algebra with a half-lattice conformal vertex algebra.

The proof of Theorem \ref{thm:intro-main-thm} goes in a few steps, and we believe that the strategy is quite general and will serve as a useful guide for future rigidity problems:

\begin{itemize}

\item {\bf Step 1:} \textit{Show that the induced module $\cF(A)$ in $\cC_A$ has a composition series with all composition factors in $\cC_A^{\loc}$.} For $\cC=\cC^{\mathrm{wt}}_k(\sl_2)$, this step requires one specific fusion rule that we compute in Section \ref{sec:fusion}.

\item {\bf Step 2:} \textit{Show that every simple object in $\cC_A$ is local.} Checking this step uses a relation between tensor products in $\cC$ and $\cC_A$ that involves $\cF(A)$ (see Lemma \ref{lemma:CLR}), as well as Frobenius reciprocity for the induction functor $\cF$. 

\item {\bf Step 3:} \textit{Show that for any simple object $X$ in $\cC$, any non-zero morphism from $\cF(X')$ to $\cF(X)^*$ in $\cC_A$ is an isomorphism.} For this step, we compute the induction of every simple object in $\cC^{\mathrm{wt}}_k(\sl_2)$, and the main tools are the same as in the previous step.

\item {\bf Step 4:} \textit{Verify the non-degeneracy condition (3) of Theorem \ref{rigidity-of-C-from-CAloc}.} For this step, we derive general statements in Theorem \ref{thm:centralizer} and Corollary \ref{cor:condition3} that hold for $\cC^{\mathrm{wt}}_k(\sl_2)$.
\end{itemize}
 All four steps are straightforward when the denominator $v$ of the shifted level in Theorem \ref{thm:intro-main-thm} is $2$, that is, $k = -2 + \frac{u}{2}$ for $u$ odd. So in Section \ref{subsec:v=2}, we quickly illustrate the strategy in these cases, before moving on to the technically more involved $v\geq 3$ case in Section \ref{subsec:v>2}. 

Finally, recall that as discussed above, the braided tensor category structure on $\cC^{\mathrm{wt}}_k(\sl_2)$ was obtained in \cite{C1} from that on the category of weight modules for the Kazama-Suzuki dual \cite{KS} of $L_k(\sl_2)$, namely the $N=2$ super Virasoro algebra. The $N=2$ super Virasoro vertex operator superalgebra at central charge $-6\ell-3$ is also the principle $W$-superalgebra $\cW_{\ell}(\sl_{2|1})$ of $\sl_{2|1}$ at level $\ell$, so once Theorem \ref{thm:intro-main-thm} is proved, we can use the Kazama-Suzuki duality to prove rigidity for the tensor category of weight $\cW_\ell(\sl_{2|1})$-modules at suitable levels.
In particular, assume that $k$ is non-integral admissible for $\mathfrak{sl}_2$ and define $\ell$ by the relation $(\ell+1)(k+2)=1$. Then $\mathcal W_\ell(\mathfrak{sl}_{2|1}) \otimes V_{\sqrt{-1}\mathbb Z}$ is a simple current extension of $L_k(\mathfrak{sl}_2) \otimes \pi$, where $V_{\sqrt{-1}\ZZ}$ is a lattice conformal vertex superalgebra and $\pi$ is a non-degenerate rank $1$ Heisenberg vertex operator algebra \cite{CGN}.  
The theory of simple current extensions, whose relevant part we discuss in Section \ref{sec:sc}, then allows us to transfer rigidity from weight modules for $L_k(\mathfrak{sl}_2)$ to weight modules for $\mathcal W_\ell(\mathfrak{sl}_{2|1})$ (see Theorem \ref{thm:N=2} below):
\begin{cor}
    If $k$ is a non-integral admissible level for $\mathfrak{sl}_2$ and $\ell$ is defined by $(\ell+1)(k+2)=1$, then the braided tensor category of finitely-generated weight modules for $\mathcal W_\ell(\mathfrak{sl}_{2|1})$ is rigid.
\end{cor}

\medskip

\noindent\textbf{Acknowledgments.}
We thank Shashank Kanade, Florencia Orosz Hunziker, and David Ridout for discussions on an earlier approach to rigidity of weight modules for $L_k(\mathfrak{sl}_2)$ at admissible levels.
RM is partially supported by a research fellowship from the Alexander von Humboldt Foundation and also thanks the Universit\"{a}t Hamburg for hospitality during the visit in which this work was finished. JY is partially supported by NSFC grant 12371030 and a startup grant from Shanghai Jiao Tong University. 
While this work was progressing, we became aware of other work \cite{Flor} also proving rigidity of weight modules for $L_k(\mathfrak{sl}_2)$ at admissible levels. Since the rigidity proof in \cite{Flor} uses completely different methods from those used here, it was decided to complete both manuscripts independently and submit them to arXiv simultaneously.

\section{The category of weight \texorpdfstring{$L_k(\mathfrak{sl}_2)$}{Lk(sl(2))}-modules}

In this section we recall basic properties and results on the category of weight modules for the affine vertex operator algebra of $\mathfrak{sl}_2$ at admissible levels.

\subsection{Weight modules for \texorpdfstring{$\mathfrak{sl}_2$}{sl(2)}}

As usual, let $\mathfrak{sl}_2$ be the simple Lie algebra spanned by $e,f,h$, with brackets
\[
[h, e] = 2e,\;\;\; [h,f]=-2f,\;\;\; [e,f] = h.
\]
Let $\langle\cdot,\cdot\rangle$ be the non-degenerate symmetric invariant bilinear form on $\mathfrak{sl}_2$ normalized so that $\langle h,h\rangle =2$. The Casimir element of $\mathfrak{sl}_2$ with respect to $\langle\cdot,\cdot\rangle$ is 
\[
C = \frac{1}{2}h^2 + ef + fe.
\]
A \textit{weight module} for $\mathfrak{sl}_2$ is a module on which $h$ acts semisimply, with finite-dimensional eigenspaces. The $h$-eigenvalues on a weight $\mathfrak{sl}_2$-module are called \textit{weights}. There are four classes of irreducible weight $\mathfrak{sl}_2$-modules:
\begin{itemize}
\item[(1)] Finite-dimensional highest-weight modules $L_r$, with highest weight $r-1$, $r \in \ZZ_{\geq 1}$.
\item[(2)] Infinite-dimensional highest-weight modules $D_\lambda^+$, with highest weight $\lambda \notin \ZZ_{\geq 0}$.
\item[(3)] Infinite-dimensional lowest-weight modules $D_\lambda^{-}$, with lowest weight $\lambda\notin \ZZ_{\leq 0}$.
\item[(4)] Relaxed highest-weight modules $E_{\lambda, \underline{C}}$, with weights $\lambda + 2n$ for some $\lambda \in \CC$ and all $n \in \ZZ$, and Casimir eigenvalue $\underline{C}$. 
\end{itemize}

We briefly recall the construction of $E_{\lambda, \underline{C}}$. First, for $\lambda, \underline{C} \in \CC$, there is an $\mathfrak{sl}_2$-module
\[
E_{\lambda, \underline{C}}^- = \bigoplus_{n \in \ZZ}\CC v_{\lambda+2n}
\]
with action given by
\begin{align*}
& h.v_{\lambda +2i} = (\lambda + 2i)v_{\lambda + 2i}, \qquad e.v_{\lambda + 2i} = v_{\lambda + 2i + 2},\qquad C.v_{\lambda + 2i} = \underline{C} v_{\lambda +2i},\\
& f.v_{\lambda + 2i} = \frac{1}{2}\bigg(\underline{C} - \frac{1}{2}(\lambda + 2i-2)^2 - (\lambda + 2i-2)\bigg)v_{\lambda + 2i -2}.
\end{align*}
It is easy to show that $E_{\lambda, \underline{C}}^-$ is reducible if and only if $\underline{C} = C_\mu := \frac{1}{2}\mu^2 + \mu$ for some $\mu \in \lambda + 2\ZZ$, in which case there is an exact sequence
\begin{equation}\label{eqn:sl2-E-}
0 \longrightarrow D_{\mu+2}^- \longrightarrow E_{\lambda, C_\mu}^- \longrightarrow D_{\mu}^+ \longrightarrow 0.
\end{equation}
Similarly, there is an $\mathfrak{sl}_2$-module
\begin{equation*}
    E_{\lambda,\underline{C}}^+ =\bigoplus_{n\in\ZZ} \CC u_{\lambda+2n}
\end{equation*}
with action given by
\begin{align*}
& h.u_{\lambda +2i} = (\lambda + 2i)u_{\lambda + 2i}, \qquad f.u_{\lambda + 2i} = u_{\lambda + 2i - 2},\qquad C.u_{\lambda + 2i} = \underline{C} u_{\lambda +2i},\\
& e.u_{\lambda + 2i} = \frac{1}{2}\bigg(\underline{C} - \frac{1}{2}(\lambda + 2i+2)^2 + (\lambda + 2i+2)\bigg)u_{\lambda + 2i +2}.
\end{align*}
The module $E_{\lambda,\underline{C}}^+$ is reducible if and only if $C=C_{-\mu}$ for some $\mu\in\lambda+2\ZZ$, in which case there is a short exact sequence
%By interchange the action of $e$ and $f$, we obtain another $\mathfrak{sl}_2$-module, denote by $E_{\lambda, \underline{C}}^+$. In this case, we have the exact sequence
\begin{equation}\label{eqn:sl2-E+}
0 \longrightarrow D_{\mu-2}^+ \longrightarrow E_{\lambda, C_{-\mu}}^+ \longrightarrow D_{\mu}^- \longrightarrow 0.
\end{equation}
If $\underline{C} \neq \frac{1}{2}\mu^2 \pm \mu$ for all $\mu \in \lambda +2\ZZ$, then
$E_{\lambda, \underline{C}}^-  \cong E_{\lambda, \underline{C}}^+$, and we denote this common module by $E_{\lambda, \underline{C}}$.

\subsection{Weight modules for affine \texorpdfstring{$\sl_2$}{sl(2)} at admissible levels}\label{subsec:affine-sl2-wt-mods}

Let $\widehat{\mathfrak{sl}}_2=\mathfrak{sl}_2[t,t^{-1}]\oplus\CC\mathbf{k}$ be the affine Lie algebra of $\mathfrak{sl}_2$, spanned by elements $a_n$ for $a\in\mathfrak{sl}_2$, $n\in\ZZ$, and a central element $\mathbf{k}$, with brackets
\begin{equation*}
    [a_m,b_n] =[a,b]_{m+n}+m\langle a,b\rangle\delta_{m+n,0}\mathbf{k}
\end{equation*}
for $a,b\in\mathfrak{sl}_2$ and $m,n\in\ZZ$. An $\widehat{\mathfrak{sl}}_2$-module has \textit{level} $k\in\CC$ if the central element $\mathbf{k}$ acts by the scalar $k$. For an $\mathfrak{sl}_2$-module $M$, the generalized Verma module of level $k$ associated to $M$ is the induced module
\[
\widehat{M}_k = U(\widehat{\mathfrak{sl}}_2) \otimes_{U(\mathfrak{sl}_2\otimes \CC[t]\oplus \CC{\bf k})}
M,
\]
where $\mathfrak{sl}_2\otimes t\CC[t]$ acts trivially on $M$ and $\mathbf{k}$ acts by $k$. If $M$ is simple and $k$ is not the critical level $-2$, then $\widehat{M}_k$ has a unique simple quotient. Let $L_k(\mathfrak{sl}_2)$ be the simple quotient of the generalized Verma module induced from the trivial $\mathfrak{sl}_2$-module $L_0$; it has the structure of a vertex operator algebra \cite{FZ}.

A weak $L_k(\mathfrak{sl}_2)$-module is a module for $L_k(\mathfrak{sl}_2)$ considered as a vertex algebra, with no grading assumptions. A \textit{generalized $L_k(\mathfrak{sl}_2)$-module} is a weak module with a $\CC$-grading given by generalized eigenspaces of the Virasoro $L_0$ operator; these generalized eigenspaces are called \textit{conformal weight spaces}, and the generalized $L_0$-eigenvalues are called \textit{conformal weights}.
\begin{defn}\label{def:wt-mod}
    A \textit{weight $L_k(\mathfrak{sl}_2)$-module} is a generalized $L_k(\mathfrak{sl}_2)$-module $W$ with a $\CC\times\CC$-grading $W=\bigoplus_{\Delta,\lambda\in\CC} W_{[\Delta]}^\lambda$ satisfying the following properties:
    \begin{enumerate}
        \item[(a)] $W_{[\Delta]}:=\bigoplus_{\lambda\in\CC} W_{[\Delta]}^\lambda$ is the conformal weight space of $W$ with conformal weight $\Delta$.
        \item[(b)] $W^\lambda:=\bigoplus_{\Delta\in\CC} W_{[\Delta]}^\lambda$ is the $h_0$-eigenspace of $W$ with eigenvalue $\lambda$.
        \item[(c)] For any $\Delta,\lambda\in\CC$, $\dim W_{[\Delta]}^\lambda <\infty$.
        \item[(d)] For any fixed $\lambda\in\CC$, $W^\lambda_{[\Delta]}=0$ for $\mathrm{Re}(\Delta)$ sufficiently negative.
    \end{enumerate}
    Let $\cC^{\mathrm{wt}}_k(\mathfrak{sl}_2)$ denote the category of finitely-generated weight $L_k(\mathfrak{sl}_2)$-modules. A generalized $L_k(\mathfrak{sl}_2)$-module is \textit{lower bounded} if $W_{[\Delta]}=0$ for $\mathrm{Re}(\Delta)$ sufficiently negative.
\end{defn}

\begin{remark}
    Weight $L_k(\mathfrak{sl}_2)$-modules are strongly $\CC$-graded generalized $L_k(\mathfrak{sl}_2)$-modules in the sense of \cite[Definition 2.25]{HLZ1}, where the $\CC$-grading here is the $h_0$-eigenvalue grading.
\end{remark}

We now fix $k$ to be an admissible level for $\mathfrak{sl}_2$, that is, $k=-2+t$ where $t=\frac{u}{v}\in\QQ_{>0}$ for coprime $u\in\ZZ_{\geq 2}$, $v\in\ZZ_{\geq 1}$. It is shown in \cite{ACK} that if $k$ is admissible, then the category $\cC_k^{\mathrm{wt}}(\mathfrak{sl}_2)$ of finitely-generated weight $L_k(\mathfrak{sl}_2)$-modules is a locally finite abelian category with enough projective objects. We first recall the lower-bounded simple modules in $\cC_k^{\mathrm{wt}}(\mathfrak{sl}_2)$, which were first classified in \cite[Theorem 4.2]{AM}; see also \cite{RW, KR}. To describe them, define
\begin{equation}
\lambda_{r,s} = r-1-ts,\qquad\Delta_{r,s}=\frac{(r-ts)^2-1}{4t}
\end{equation}
for integers $1\leq r\leq u-1$ and $0\leq s\leq v-1$. These have the symmetries
\begin{equation}\label{eqn:lambda-Delta-symmetries}
    \lambda_{u-r,v-s} = -\lambda_{r,s}-2,\qquad\Delta_{u-r,v-s}=\Delta_{r,s}
\end{equation}
for $1\leq r\leq u-1$ and $1\leq s\leq v-1$.

For $r\in\ZZ_{\geq 1}$, let $\cL_{r,0}$ denote the simple quotient of the generalized Verma module of level $k$ induced from the simple $r$-dimensional $\mathfrak{sl}_2$-module $L_r$. It is a module for $L_k(\mathfrak{sl}_2)$ if and only if $1\leq r\leq u-1$. Then for $\lambda\in\CC$, let $\cD_\lambda^+$ denote the simple quotient of the generalized Verma module of level $k$ induced from the simple highest-weight $\mathfrak{sl}_2$-module $D_\lambda^+$. 
It is a module for $L_k(\mathfrak{sl}_2)$ if and only if $\lambda=\lambda_{r,s}$ for some integers $1\leq r\leq u-1$ and $0\leq s\leq v-1$. In this case, we set $\cD_{r,s}^+:=\cD_{\lambda_{r,s}}^+$, so that in particular $\cD_{r,0}^+=\cL_{r,0}$. The module $\cD_{r,s}^+$ is the irreducible highest-weight $\widehat{\mathfrak{sl}}_2$-module whose highest-weight vector has $h_0$-weight $\lambda_{r,s}$ and conformal weight $\Delta_{r,s}$. 
 Let $\cD_{r,s}^-$ denote the contragredient $L_k(\mathfrak{sl}_2)$-module of $\cD_{r,s}^+$ in the sense of Definition 2.32 and Theorem 2.34 of \cite{HLZ1}. This is the irreducible lowest-weight $\widehat{\mathfrak{sl}}_2$-module whose lowest-weight vector has $h_0$-weight $-\lambda_{r,s}$ and conformal weight $\Delta_{r,s}$. Note that $\cD_{r,0}^- = \cL_{r,0}$.

Now for $\lambda+2\ZZ \in \CC/2\ZZ$ and $\Delta\in\CC$, let $\cE^\pm_{\lambda,\Delta}$ denote the quotient of the generalized Verma module induced from $E^\pm_{\lambda,2t\Delta}$ by the maximal submodule which intersects the lowest conformal weight subspace $E^\pm_{\lambda,2t\Delta}$ trivially. It is a module for $L_k(\mathfrak{sl}_2)$ if and only if $\Delta=\Delta_{r,s}$ for some integers $1\leq r\leq u-1$ and $1\leq s\leq v-1$. Note that
 $2t\Delta_{r,s}=\frac{1}{2}\mu^2\pm\mu$ for some $\mu\in\lambda+2\ZZ$ if and only if  $\lambda\in\pm\lambda_{r,s}+2\ZZ$. Thus if $\lambda \notin \pm\lambda_{r,s} + 2\ZZ$, then $\cE^+_{\lambda,\Delta_{r,s}}$ and $\cE^-_{\lambda,\Delta_{r,s}}$ are isomorphic simple $L_k(\mathfrak{sl}_2)$-modules, and we denote this common module by $\cE_{\lambda,\Delta_{r,s}}$.
If $\lambda\in\pm\lambda_{r,s}+2\ZZ$, then we set $\cE_{r,s}^+=\cE^+_{\lambda_{r,s},\Delta_{r,s}}$ and $\cE^-_{r,s}=\cE^-_{-\lambda_{r,s},\Delta_{r,s}}$. Note that $\cE^\pm_{\mp\lambda_{r,s},\Delta_{r,s}} =\cE^\pm_{u-r,v-s}$ by \eqref{eqn:lambda-Delta-symmetries}. The $L_k(\sl_2)$-modules $\cE_{r,s}^\pm$ are indecomposable, and using \eqref{eqn:sl2-E-}, \eqref{eqn:sl2-E+}, and \eqref{eqn:lambda-Delta-symmetries}, there are non-split short exact sequences
\begin{equation}\label{eqn:exact-seq-E+E-}
0 \longrightarrow \cD_{r,s}^+ \longrightarrow \cE_{r,s}^+ \longrightarrow \cD^-_{u-r,v-s} \longrightarrow 0, \qquad\quad 0 \longrightarrow \cD_{r,s}^- \longrightarrow \cE_{r,s}^- \longrightarrow \cD^+_{u-r,v-s} \longrightarrow 0.
\end{equation}
for $1\leq r\leq u-1$ and $1\leq s\leq v-1$.

By \cite[Theorem 4.4]{ACK}, the above modules comprise the complete list of indecomposable lower-bounded weight $L_k(\mathfrak{sl}_2)$-modules. To summarize, these are the following modules:
\begin{enumerate}
    \item The irreducible highest-weight modules $\cL_{r,0}$ and $\cD_{r,s}^+$ for $1\leq r\leq u-1$ and $1\leq s\leq v-1$.
    \item The irreducible lowest-weight modules $\cD_{r,s}^-=(\cD_{r,s}^+)'$ for $1\leq r\leq u-1$ and $1\leq s\leq v-1$.
    \item The irreducible relaxed highest-weight modules $\cE_{\lambda,\Delta_{r,s}}$ for $\lambda+2\ZZ\in\CC/2\ZZ$, $1\leq r\leq u-1$, $1\leq s\leq v-1$, and $\lambda\notin\pm\lambda_{r,s}+ 2\ZZ$.
    \item The reducible relaxed highest-weight modules $\cE_{r,s}^\pm =\cE_{\pm\lambda_{r,s}, \Delta_{r,s}}^{\pm}$ for $1\leq r\leq u-1$, $1\leq s\leq v-1$.
\end{enumerate}

For the remaining simple objects of $\cC_k^{\mathrm{wt}}(\mathfrak{sl}_2)$, we recall the spectral flow twist of an $\widehat{\mathfrak{sl}}_2$ module. For $\ell\in\ZZ$, the spectral flow automorphism $\sigma^\ell$ of $\widehat{\mathfrak{sl}}_2$ fixes $\mathbf{k}$ and acts on the remaining basis elements of $\widehat{\mathfrak{sl}}_2$ by
\[
\sigma^\ell(e_n) = e_{n-\ell},\qquad \sigma^\ell(f_n) = f_{n+\ell},\qquad \sigma^\ell(h_n) = h_n + \delta_{n,0}\ell {\bf k} 
\]
for $n \in \ZZ$. Then the spectral flow twist $\sigma^\ell(M)$ of an $\widehat{\mathfrak{sl}}_2$-module $M$ has the same underlying vector space as $M$ with the action of $x\in\widehat{\mathfrak{sl}}_2$ given by
\[
x.m = \sigma^{-\ell}(x).m
\]
for $m\in\sigma^\ell(M)$. If $M$ is additionally an $L_k(\mathfrak{sl}_2)$-module, then so is $\sigma^\ell(M)$, with vertex operator 
\[
Y_{\sigma^\ell(M)}(-,x): L_k(\mathfrak{sl}_2)\otimes\sigma^\ell(M)\rightarrow\sigma^\ell(M)((x))
\]
given by Li's $\Delta$-operators \cite{Li}. Specifically, $Y_{\sigma^\ell(M)}(v,x) =Y_M(\Delta(\frac{\ell}{2} h,x)v,x)$, where
\begin{equation*}
    \Delta\left(\frac{\ell}{2}h,x\right) =x^{\ell h_0/2}\exp\left(\sum_{n=1}^\infty \frac{(-1)^{n+1}\ell}{2n} h_n\, x^{-n}\right).
\end{equation*}
Note that $\Delta(\frac{\ell_1}{2}h,x)\Delta(\frac{\ell_2}{2}h,x)=\Delta(\frac{\ell_1+\ell_2}{2}h,x)$ for any $\ell_1,\ell_2\in\ZZ$, and thus
\begin{equation}\label{eqn:comp-of-spec-flow}
    \sigma^{\ell_1}(\sigma^{\ell_2}(M)) = \sigma^{\ell_1+\ell_2}(M)
\end{equation}
for any $L_k(\mathfrak{sl}_2)$-module $M$.
By \cite[Corollary 3.4]{ACK} (see also Remark 3.2 and Theorem 3.11 in \cite{AKR}), every simple object of $\cC_k^{\mathrm{wt}}(\mathfrak{sl}_2)$ is a spectral flow twist of some lower-bounded simple weight $L_k(\mathfrak{sl}_2)$-module. Thus any simple object of $\cC_k^{\mathrm{wt}}(\mathfrak{sl}_2)$ is isomorphic to at least one of the following:
\begin{equation*}
\sigma^\ell(\cL_{r,0}),\qquad\sigma^\ell(\cD_{r,s}^+),\qquad\sigma^\ell(\cD_{r,s}^-),\qquad\sigma^\ell(\cE_{\lambda,\Delta_{r,s}})
\end{equation*} 
There is some repetition in this list of simple modules in $\cC_k^{\mathrm{wt}}(\mathfrak{sl}_2)$, since the simple highest-weight $L_k(\mathfrak{sl}_2)$-modules can be realized as spectral flows of lowest-weight modules (see \cite[Figure 2]{CR2}):
\begin{align}\label{eqn:HW-as-SF-of-LW}
\cD_{r,s}^+\cong\begin{cases}
    \sigma(\cD^-_{u-r,v-s-1}) & \text{if}\,\,\,1\leq s\leq v-2\\
    \sigma^2(\cD^-_{r,v-1}) & \text{if}\,\,\, s=v-1\\
\end{cases},\qquad \cL_{r,0}\cong\sigma(\cD_{u-r,v-1}^-)
\end{align}
for $1\leq r\leq u-1$. As a result, any simple object of $\cC^{\mathrm{wt}}_k(\sl_2)$ is isomorphic to exactly one of the following:
\begin{equation*}
    \sigma^\ell(\cD_{r,s}^+),\qquad\sigma^\ell(\cE_{\lambda,\Delta_{r,s}})
\end{equation*}
for integers $1\leq r\leq u-1$ and $1\leq s\leq v-1$, $\lambda\notin\pm\lambda_{r,s}+2\ZZ$, and $\ell\in\ZZ$.

It is also shown in \cite{ACK} that every simple object of $\cC^{\mathrm{wt}}_k(\sl_2)$ has a projective cover. The simple modules $\sigma^\ell(\cE_{\lambda,\Delta_{r,s}})$ are already projective in $\cC^{\mathrm{wt}}_k(\sl_2)$, while
logarithmic projective covers of the highest-weight modules were constructed \cite{A,ACK}. Let $\cP_{r,s}$ denote the projective cover of $\cD_{r,s}^+$ for $1\leq r\leq u-1$ and $1\leq s\leq v-1$, so that 
$\sigma^\ell(\cP_{r, s})$ is the projective cover (and also injective hull) of $\sigma^\ell(\cD^+_{r, s})$. Then there are non-split short exact sequences
\begin{equation*}
\begin{split}
         0 \longrightarrow \sigma^{\ell+1}(\cE^-_{u-r, v-s-1}) \longrightarrow 
        &\, \sigma^{\ell}(\cP_{r, s}) \longrightarrow  \sigma^{\ell}(\cE^-_{u-r, v-s}) \longrightarrow 0 \qquad (\text{if} \ 1\leq s\leq v-2) \\
         0 \longrightarrow \sigma^{\ell+2}(\cE^-_{r, v-1}) \longrightarrow 
        &\,\sigma^{\ell}(\cP_{r, v-1}) \longrightarrow  \sigma^{\ell}(\cE^-_{u-r, 1}) \longrightarrow 0 
    \end{split} 
\end{equation*}
for $1\leq r\leq u-1$, $1\leq s\leq v-1$, and $\ell\in\ZZ$.

It is shown in \cite{C1} that $\cC_k^{\mathrm{wt}}(\mathfrak{sl}_2)$ admits the braided tensor category structure with ribbon twist of \cite{HLZ1}-\cite{HLZ8} for categories of strongly $\CC$-graded generalized modules for a vertex operator algebra. The tensor product operation on $\cC_k^{\mathrm{wt}}(\mathfrak{sl}_2)$ is the conformal-field-theoretic fusion tensor product $\boxtimes$, which is mathematically defined in terms of intertwining operators of type $\binom{M_3}{M_1\,M_2}$, where $M_1$, $M_2$, $M_3$ are objects of $\cC_k^{\mathrm{wt}}(\mathfrak{sl}_2)$. These are linear maps
\begin{equation*}
    \cY: M_1\otimes M_2\longrightarrow M_3[\log x]\lbrace x\rbrace
\end{equation*}
satisfying several properties specified in \cite[Definition 3.10]{HLZ2}. We will in particular need two properties of intertwining operators in $\cC_k^{\mathrm{wt}}(\mathfrak{sl}_2)$: the associator formula
\begin{equation}\label{eqn:intw_op_it}
\cY(a_{n}m_1, x)m_2 = \sum_{i\geq 0}\binom{n}{i}(-x)^{i}a_{n-i}\cY(m_1,x)m_2-\sum_{i\geq 0} \binom{n}{i}(-x)^{n-i}\cY(m_1,x)a_i m_2
\end{equation}
and the commutator formula
\begin{equation}\label{eqn:intw_op_comm}
\cY(m_1,x)a_{n} m_2=a_{n}\cY(m_1,x)m_2-\sum_{i\geq 0}\binom{n}{i}x^{n-i}\cY(a_i m_1,x)m_2
\end{equation}
for $a\in\mathfrak{sl}_2$, $n\in\ZZ$, $m_1\in M_1$, and $m_2\in M_2$. We will also need the following lemma on intertwining operators from \cite{DL}:
\begin{lemma}[\cite{DL}, Proposition 11.9]\label{lem:nonzerofusion}
Let $V$ be a conformal vertex algebra, let $M_1, M_2,M_3$ be $V$-modules such that $M_1, M_2$ are simple, and let $\cY$ be a non-zero intertwining operator of type $\binom{M_3}{M_1\,M_2}$. Then $\cY(m_1,x)m_2\neq 0$ for all non-zero $m_1\in M_1$, $m_2\in M_2$.
\end{lemma}

If $M_1$ and $M_2$ are objects of $\cC_k^{\mathrm{wt}}(\mathfrak{sl}_2)$, then their fusion tensor product is an object $M_1\boxtimes M_2$ of $\cC_k^{\mathrm{wt}}(\mathfrak{sl}_2)$ equipped with an intertwining operator $\cY_{M_1,M_2}$ of type $\binom{M_1\boxtimes M_2}{M_1\,M_2}$ satisfying a universal property:
for any intertwining operator $\cY$ of type $\binom{M_3}{M_1\,M_2}$, where $M_3$ is an object of $\cC^{\mathrm{wt}}_k(\mathfrak{sl}_2)$, there is a unique $L_k(\mathfrak{sl}_2)$-module map $f: M_1\boxtimes M_2\rightarrow M_3$ such that $f\circ\cY_{M_1,M_2}=\cY$. An important property of fusion tensor products is their compatibility with spectral flow:
\begin{prop}\label{prop:tens-prod-and-spec-flow}
   For any objects $M_1$, $M_2$ of $\cC_k^{\mathrm{wt}}(\mathfrak{sl}_2)$ and $\ell_1,\ell_2\in\ZZ$, 
   \[
   \sigma^{\ell_1}(M_1)\boxtimes\sigma^{\ell_2}(M_2)\cong\sigma^{\ell_1+\ell_2}(M_1\boxtimes M_2).
   \]
\end{prop}
\begin{proof}
    We first consider the case $\ell_1=0$. By \cite[Lemma 2.1]{Li}, $\widetilde{\cY}=\cY_{M_1,M_2}\circ(\Delta(\frac{\ell_2}{2}h,x)\otimes\id_{M_2})$ is an intertwining operator of type $\binom{\sigma^{\ell_2}(M_1\boxtimes M_2)}{M_1\,\,\sigma^{\ell_2}(M_2)}$, so there is a unique $L_k(\mathfrak{sl}_2)$-module map 
    \[
    f: M_1\boxtimes\sigma^{\ell_2}(M_2)\longrightarrow\sigma^{\ell_2}(M_1\boxtimes M_2).
    \]
    such that $f\circ\cY_{M_1,\sigma^{\ell_2}(M_2)}=\widetilde{\cY}$.
    Similarly, using \eqref{eqn:comp-of-spec-flow}, there is an $L_k(\mathfrak{sl}_2)$-module map
    \begin{equation*}
         \sigma^{\ell_2}(M_1\boxtimes M_2) =\sigma^{\ell_2}(M_1\boxtimes\sigma^{-\ell_2}(\sigma^{\ell_2}(M_2))) \xrightarrow{g} \sigma^{\ell_2}(\sigma^{-\ell_2}(M_1\boxtimes\sigma^{\ell_2}(M_2))) = M_1\boxtimes\sigma^{\ell_2}(M_2),
    \end{equation*}
    where $g\circ\cY_{M_1,M_2} =\cY_{M_1,\sigma^{\ell_2}(M_2)}\circ(\Delta(-\frac{\ell_2}{2}h,x)\otimes\id_{M_2})$. These definitions imply
    \begin{equation*}
        f\circ g\circ\cY_{M_1,M_2}=\cY_{M_1,M_2},\qquad g\circ f\circ\cY_{M_1,\sigma^{\ell_2}(M_2)} =\cY_{M_1,\sigma^{\ell_2}(M_2)},
    \end{equation*}
    and thus $f\circ g=\id_{\sigma^{\ell_2}(M_1\boxtimes M_2)}$ and $g\circ f=\id_{M_1\boxtimes\sigma^{\ell_2}(M_2)}$, because $\cY_{M_1,M_2}$ and $\cY_{M_1,\sigma^{\ell_2}(M_2)}$ are surjective intertwining operators (see \cite[Proposition 4.23]{HLZ3}). This proves the $\ell_1=0$ case of the proposition, and the general case follows because
    \begin{align*}
        \sigma^{\ell_1}(M_1)\boxtimes\sigma^{\ell_2}(M_2) & \cong\sigma^{\ell_2}(\sigma^{\ell_1}(M_1)\boxtimes M_2)\cong\sigma^{\ell_2}(M_2\boxtimes\sigma^{\ell_1}(M_1))\nonumber\\
        & \cong\sigma^{\ell_2}(\sigma^{\ell_1}(M_2\boxtimes M_1)) =\sigma^{\ell_1+\ell_2}(M_2\boxtimes M_1)\cong \sigma^{\ell_1+\ell_2}(M_1\boxtimes M_2),
    \end{align*}
    using \eqref{eqn:comp-of-spec-flow} and commutativity of the fusion tensor product.
\end{proof}

\subsection{Free field realization}\label{sec:Ada}

In this subsection, we recall the free field realization of $L_k(\mathfrak{sl}_2)$ proved by Adamovi\'c \cite{A} (see also \cite{S}). Let $L = \ZZ c + \ZZ d$ be the lattice such that
\[
\langle c, c \rangle = \langle d, d \rangle = 0,\qquad \langle c, d \rangle = 2,
\]
and let $V_L=\pi^{c,d}\otimes \CC[L]$ be the corresponding lattice conformal vertex algebra \cite{FLM}, where $\pi^{c,d}$ is the rank $2$ Heisenberg vertex operator algebra generated by $c$ and $d$ and $\CC[L]$ is the group algebra of the lattice $L$. Define the conformal vertex subalgebra $\Pi(0) = \pi^{c,d} \otimes \CC[\ZZ c] \subseteq V_L$; its representation theory was studied in \cite{BDT}. Also let $\Vir_{c_k}$ be the simple Virasoro vertex operator algebra of central charge $c_k=1 - \frac{6(k+1)^2}{k+2}$.
\begin{theorem}[\cite{A} Theorem 5.5]\label{thm: conformal embedding}
    If $k$ is a non-integral admissible level of $\mathfrak{sl}_2$, then there is an injective vertex algebra homomorphism $L_k(\mathfrak{sl}_2) \hookrightarrow \Vir_{c_k} \otimes \Pi(0)$ sending $h_{-1}\vac$ to $2\mu_{-1}\vac$, where $\mu=\frac{k}{4}c+\frac{1}{2}d$. This embedding is conformal if the conformal vector of $\Pi(0)$ is taken to be $\frac{1}{2}(c_{-1}d_{-1}-d_{-2}+\frac{k}{2}c_{-2})\vac$.
\end{theorem}

In particular, any $\Vir_{c_k} \otimes \Pi(0)$-module restricts to an $L_k(\mathfrak{sl}_2)$-module. The conformal vertex algebra $\Pi(0)$ has simple strongly $\CC$-graded modules 
\[
\Pi_\ell(\lambda) = \pi^{c,d}\otimes e^{\ell\mu+\lambda c}\CC[\ZZ c]
\]
for $\ell \in \ZZ$ and $\lambda \in \CC$, where the $\CC$-grading is given by $\mu_0$-eigenvalues.
The simple Virasoro vertex operator algebra $\mathrm{Vir}_{c_k}$ is strongly rational \cite{W}, and any simple $\mathrm{Vir}_{c_k}$-module is isomorphic to the simple highest-weight $\mathrm{Vir}$-module $M_{r,s}^k$ of lowest conformal weight $h_{r,s}=\frac{1}{4uv}((su-rv)^2-(u-v)^2)$ for $1\leq r \leq u-1$ and $1\leq s \leq v-1$; note that $M_{r,s}^k = M_{u-r,v-s}^k$. The fusion rules of the simple $\mathrm{Vir}_{c_k}$-modules are given by
\begin{equation}\label{eqn:Vir-fus-rules}
    M_{r,s}^k \boxtimes M_{r',s'}^k \cong  \bigoplus_{\stackrel{r'' = \vert r-r'\vert+1}{r+r'+r''\,\text{odd}}}^{\min(r+r'-1,2u-r-r'-1)}\bigoplus_{\stackrel{s'' = \vert s-s'\vert+1}{s+s'+s''\,\text{odd}}}^{\min(s+s'-1,2v-s-s'-1)} M_{r'',s''}^k.
\end{equation}
Thus $\Vir_{c_k} \otimes \Pi(0)$ has simple modules $M_{r, s}^k \otimes \Pi_\ell(\lambda)$
for $1\leq r \leq u-1$, $1\leq s \leq v-1$, $\ell \in \ZZ$, and $\lambda \in \mathbb C$. Note that $\Pi_\ell(\lambda) \cong \Pi_\ell(\lambda+1)$, so these weight labels are taken in $\mathbb C/\mathbb Z$.

Let $\cG$ be the restriction functor from the category of generalized $\Vir_{c_k}\otimes\Pi(0)$-modules to the category of generalized $L_k(\mathfrak{sl}_2)$-modules.
We call a simple $\Vir_{c_k} \otimes \Pi(0)$-module $X$ {\em typical} if $\cG(X)$ is a simple $L_k(\mathfrak{sl}_2)$-module, and we call $X$ \textit{atypical} otherwise. Restrictions of simple $\Vir_{c_k}\otimes\Pi(0)$-modules are as follows (see \cite[Section 7]{A}):
\begin{equation}\label{eqn:FF-affine-sl2-correspondence}
\cG(M_{r, s}^k \otimes \Pi_{\ell -1}(\lambda)) \cong  \begin{cases}
    \sigma^\ell(\cE^-_{u-r, v-s}) & \text{if}\,\, \lambda = \nu_{r, s} \ \text{(atypical)} \\
    \sigma^\ell(\cE^-_{r, s}) & \text{if}\,\,\lambda = \nu_{u-r, v-s} \ \text{(atypical)}\\
    \sigma^\ell(\cE_{2\lambda-k, \Delta_{r, s}}) & \text{otherwise \ (typical)} \\
\end{cases} ,
\quad
\text{for} 
\ \ 
\nu_{r, s} = \frac{1}{2}(r-1-t(s-1))).
\end{equation}
The modules $\sigma^\ell(\cD^+_{r, s})$ for $1\leq r\leq u-1$ and $1\leq v\leq s-1$ appear as both submodules and quotients of simple $\Vir_{c_k} \otimes \Pi(0)$-modules. Namely, 
\begin{equation*}
       \cG(M_{r, s}^k \otimes \Pi_{\ell -1}(\nu_{r, s})) \cong \sigma^\ell(\cE^-_{u-r, v-s}) \twoheadrightarrow \sigma^\ell(\cD^+_{r, s}),
\end{equation*} 
and using \eqref{eqn:HW-as-SF-of-LW},
\begin{equation*}
\sigma^\ell(\cD^+_{r, s}) \hookrightarrow \begin{cases}
    \sigma^{\ell+1}(\cE^-_{u-r, v-s -1}) \cong \cG(M_{r, s+1}^k \otimes \Pi_{\ell}(\nu_{r, s+1})) & \text{if}\,\,\,1\leq s \leq v-2 \\
    \sigma^{\ell+2}(\cE^-_{r, v-1}) \cong \cG(M_{u-r, 1}^k \otimes \Pi_{\ell+1}(\nu_{u-r, 1})) & \text{if}\,\,\,s = v-1 \\
\end{cases} .
\end{equation*}

\begin{remark}\label{rem:onetone}
For later use, we formalize the above observation.
Let $\cC=\cC^{\mathrm{wt}}_k(\mathfrak{sl}_2)$, and let $\cD$ be the category of generalized $\Vir_{c_k} \otimes \Pi(0)$-modules $X$ such that $\cG(X)$ is an object of $\cC^{\mathrm{wt}}_k(\mathfrak{sl}_2)$. Then there are injective maps $\tau, \tilde\tau: \mathrm{Irr}(\cC) \rightarrow \mathrm{Irr}(\cD)$ such that $X \hookrightarrow \cG(\tau(X))$ and $\cG(\tilde\tau(X)) \twoheadrightarrow X$ for all $X\in\mathrm{Irr}(\cC)$.
Explicitly, 
\[\tau(\sigma^\ell(\cE_{2\lambda-k, \Delta_{r, s}})) =
\tilde\tau(\sigma^\ell(\cE_{2\lambda-k, \Delta_{r, s}})) = M_{r, s}^k \otimes \Pi_{\ell -1}(\lambda)
\]
if $\lambda \neq \nu_{r, s}, \nu_{u-r, v-s}$ mod $\ZZ$,
and 
\[
\tau(\sigma^\ell(\cD^+_{r, s})) = \begin{cases}
    M_{r, s+1}^k \otimes \Pi_{\ell}(\nu_{r, s+1}) & \text{if}\,\,\, 1\leq s \leq v-2 \\
    M_{u-r, 1}^k \otimes \Pi_{\ell+1}(\nu_{u-r, 1}) & \text{if}\,\,\, s = v-1 \\
\end{cases}, \qquad\quad
\tilde\tau(\sigma^\ell(\cD^+_{r, s})) = M_{r, s}^k \otimes \Pi_{\ell -1}(\nu_{r, s}).
\]
for integers $1\leq r\leq u-1$ and $1\leq s\leq v-1$. After proving Theorem \ref{thm:CAloc-rigid} below, it will follow that $\tau$ and $\tilde{\tau}$ are in fact bijections.
\end{remark}

The contragredient dual of the typical module $\sigma^\ell(\cE_{\lambda, \Delta_{r,s}})$ for $\lambda \notin \pm \lambda_{r,s}+2\ZZ$ is $\sigma^{-\ell}(\cE_{-\lambda, \Delta_{r,s}})$, so
\begin{equation} \label{eq:contragedientdual typical}
    \cG(M_{r, s}^k \otimes \Pi_{\ell}(\lambda))' = (\sigma^{\ell+1}(\cE_{2\lambda-k, \Delta_{r,s}}))' = \sigma^{-\ell-1}(\cE_{-2\lambda + k, \Delta_{r,s}}) = \cG(M_{r, s}^k \otimes \Pi_{-\ell-2}(t-\lambda)).
\end{equation}
if $\lambda\neq\nu_{r,s},\nu_{u-r,v-s}$ mod $\ZZ$. For atypical modules, the contragredient dual of $\sigma^\ell(\cE^-_{u-r, v-s})$ is $\sigma^{-\ell}(\cE^-_{r, s})$, 
so
\begin{equation} \label{eq:contragedientdual atypical}
    \cG(M_{r, s}^k \otimes \Pi_{\ell}(\nu_{r, s}))' = \cG(M_{r, s}^k \otimes \Pi_{-\ell -2}(\nu_{u-r, v-s}))
\end{equation}
for integers $1\leq r\leq u-1$ and $1\leq s\leq v-1$. Finally, we discuss the structure of $\Vir_{c_k}\otimes\Pi(0)$ as an $L_k(\mathfrak{sl}_2)$-module. By the $\ell=r=s=1$ case of \eqref{eqn:FF-affine-sl2-correspondence},
\begin{align*}
    \cG(\Vir_{c_k}\otimes\Pi(0)) = \sigma(\cE_{u-1,v-1}^-),
\end{align*}
and thus by \eqref{eqn:exact-seq-E+E-} and \eqref{eqn:HW-as-SF-of-LW}, there is a non-split short exact sequence
\begin{equation}\label{eqn:A-as-Lk(sl2)-mod}
    0\longrightarrow \cL_{1,0}\longrightarrow \cG(\Vir_{c_k}\otimes\Pi(0))\longrightarrow \sigma(\cD_{1,1}^+)\longrightarrow 0.
\end{equation}
Here $\cL_{1,0}$ is just $L_k(\mathfrak{sl}_2)$ considered as a module for itself. Thus $\Vir_{c_k}\otimes\Pi(0)$ is indecomposable as an $L_k(\mathfrak{sl}_2)$-module, with unique simple submodule $L_k(\mathfrak{sl}_2)$ and unique simple quotient $\sigma(\cD_{1,1}^+)$.

\section{Commutative algebras in braided tensor categories}

In this section, we recall and prove basic results on commutative algebras in braided tensor categories and their relation to conformal vertex algebra extensions.
Then we apply these results particularly to the extension $L_k(\mathfrak{sl}_2)\hookrightarrow \Vir_{c_k}\otimes\Pi(0)$, obtaining fundamental results on the representation theory of $\Vir_{c_k}\otimes\Pi(0)$ that will be needed to prove rigidity of $\cC^{\mathrm{wt}}_k(\sl_2)$ in Section \ref{sec:rigidity}.

\subsection{Tensor categories for conformal vertex algebra extensions}
First, we recall some results on tensor categories for conformal vertex algebra extensions developed in \cite{HKL, CKM-ext, CMSY}. Let $V$ be a vertex operator algebra and $A$ a conformal vertex algebra extension of $V$ in a braided tensor category $\cC$ of $V$-modules. In this paper, we mainly take $V = L_k(\mathfrak{sl}_2)$, $A = \Vir_{c_k} \otimes \Pi(0)$, and $\cC=\cC_k^{\mathrm{wt}}(\mathfrak{sl}_2)$. The extension $A$ has the structure of a commutative algebra in the braided tensor category $\cC$ \cite{HKL},  
and there is a tensor category $\cC_A$ of modules for the commutative algebra $A$ in $\cC$. The subcategory $\cC_A^{\loc}$ of local $A$-modules is a braided tensor category, and by \cite{CKM-ext}, the braided tensor category structure on $\cC_A^{\loc}$ is the same as that given in \cite{HLZ8} for the category of modules for $A$ (considered as a conformal vertex algebra) that lie in $\cC$.

There is a monoidal induction functor $\cF: \cC \rightarrow \cC_A$ given on objects by $\cF(X)=A\boxtimes X$, and the forgetful functor $\cG: \cC_A \rightarrow \cC$ is the right adjoint of $\cF$, that is, Frobenius reciprocity holds:
\begin{lemma}\label{reciprocity}
For objects $X$ in $\cC$ and $Y$ in $\cC_A$,  $\Hom_{\cC_A}(\cF(X), Y) \cong \Hom_{\cC}(X, \cG(Y))$.
\end{lemma}

To state the next result, we recall the definition of Grothendieck-Verdier category from \cite{BD}: 
   \begin{itemize}
\item A \textit{(braided)  Grothendieck-Verdier category} is a (braided) monoidal category $\cC$ equipped with an anti-equivalence $D:\cC\rightarrow\cC$
such that for each object $X$ in $\cC$, there is a natural isomorphism
    \[
    \Hom_{\cC}(- \boxtimes X, D\vac) \cong \Hom_{\cC}(-, DX).
    \]
    The functor $D$ is called the {\em dualizing functor} of $\cC$.
    \item A {\em ribbon Grothendieck-Verdier category} is a braided Grothendieck-Verdier category $(\cC, D)$ with a twist $\theta$ satisfying the balancing equation and such that $D\theta_X = \theta_{DX}$ for all objects $X$ in $\cC$.
    \item An {\em r-category} is a Grothendieck-Verdier category $\cC$ such that $D\vac\cong\vac$.
\end{itemize}
If $\cC$ is a Grothendieck-Verdier category with dualizing functor $D$ and $A$ is a commutative algebra in $\cC$, then the category $\cC_A$ of $A$-modules in $\cC$ is also a Grothendieck-Verdier
category with dualizing functor $D_A$ such that $\cG\circ D_A =D$ \cite[Theorem 3.9]{CMSY}.

If $V$ is a $\ZZ$-graded vertex operator algebra and $\cC$ is a braided tensor category of $V$-modules that is closed under contragredient duals, then $\cC$ is a ribbon Grothendieck-Verdier category with dualizing functor $(-)'$ of taking contragredient duals and ribbon twist $e^{2\pi i L_0}$ \cite{ALSW}. If $V\cong V'$ as a $V$-module, then $\cC$ is also an r-category. Thus if $A$ is a conformal vertex algebra extension of $V$ in $\cC$, then $\cC_A$ is a Grothendieck-Verdier category, and if additionally $A$ is $\ZZ$- or $\frac{1}{2}\ZZ$-graded by $L_0$-eigenvalues, then by \cite[Theorem 3.11]{CMSY}, $\cC_A^{\loc}$ is a braided Grothendieck-Verdier category.

In this paper, we will prove that $\cC_k^{\mathrm{wt}}(\mathfrak{sl}_2)$ is rigid using the following main result of \cite{CMSY}:
\begin{theorem}[\cite{CMSY} Theorem 3.21]\label{rigidity-of-C-from-CAloc}
Let $(\cC, D, \theta)$ be a locally finite $\CC$-linear abelian r-category with braiding $c$, and let $A$ be a commutative algebra in $\cC$ such that $\theta_A^2=\id_A$ and the unit morphism $\iota_A: {\bf 1} \rightarrow A$ is injective. Suppose the following conditions are satisfied:
\begin{enumerate}
\item $\cC_A^{\loc}$ is rigid and every simple object of $\cC_A$ is an object of $\cC_A^{\loc}$.
\item For any simple object $X$ in $\cC$, the evaluation morphism $e_X: DX \boxtimes X \rightarrow {\bf 1}$ is surjective and any nonzero $\cC_A$-morphism from $\cF(DX)$ to $\cF(X)^*$ is an isomorphism, where $(-)^*$ denotes the left dual of an object in $\cC_A$.
\item For any $\cC$-subobject $s: S \hookrightarrow A$ such that $S$ is not contained in $\mathrm{Im}\,\iota_A$, there exists an object $Z$ in $\cC$ such that $c_{Z,S} \neq c_{S,Z}^{-1}$ and the map $s \boxtimes \id_Z: S\boxtimes Z \rightarrow A \boxtimes Z$  is injective. 
\end{enumerate}
Then $\cC$ is rigid.
\end{theorem}

In the proof of this theorem given in \cite{CMSY}, condition (1) is used to show that $\cC_A$ is rigid, and therefore so is its Drinfeld center $\cZ(\cC_A)$. Condition (2) is used to show that induction commutes with duality, that is, $\cF(DX)\cong\cF(X)^*$ for all objects $X$ in $\cC$. This together with condition (3) is then used to prove that the natural lift of $\cF$ to a braided tensor functor from $\cC$ to $\cZ(\cC_A)$ is fully faithful. Then $\cC$ is rigid because it embeds into a rigid category via a tensor functor that commutes with duality.

We now show that condition (3) of Theorem \ref{rigidity-of-C-from-CAloc} holds in a setting that we will show later includes the case $\cC=\cC^{\mathrm{wt}}_k(\mathfrak{sl}_2)$, $A=\Vir_{c_k}\otimes\Pi(0)$. 
\begin{assumption}\label{ass1}
    Let $V\hookrightarrow A$ be an embedding of simple conformal vertex algebras and let $\cC$ be a braided tensor category of $V$-modules such that $A$ is an object of $\cC$. Assume that for any simple object $X$ in $\cC$, there is a simple object $\tau(X)$ in $\cC_A^\loc$ together with an embedding $i_X: X\hookrightarrow\cG(\tau(X))$, such that $\tau(X)$ is M\"{u}ger central in $\cC_A^\loc$ only if $X\cong V$.
\end{assumption}

In the setting of Assumption \ref{ass1}, let $\cS$ be the image of the restriction functor $\cG:\cC_A^\loc\rightarrow\cC$. That is, an object $X$ in $\cC$ is an object of $\cS$ if and only if $X\cong\cG(Z)$ for some object $Z$ in $\cC_A^\loc$.
We say that an object $X$ in $\cC$ centralizes $\cS$ if $X$ centralizes every object in $\cS$, that is, $c_{X,S}=c_{S,X}^{-1}$ for all objects $S$ in $\cC$. The subcategory of $\cC$ consisting of objects that centralize $\cS$ is called the M\"{u}ger centralizer of $\cS$ in $\cC$. 

\begin{theorem}\label{thm:centralizer} 
Under Assumption \ref{ass1}, the only simple object in the M\"{u}ger centralizer of $\cS$ is $V$.
\end{theorem}
\begin{proof}
Let $X$ be a simple object of $\cC$ which is not isomorphic to $V$. Then by hypothesis, there is an object $Z$ of $\cC_A^\loc$ such that the monodromy isomorphism $M^A_{\tau(X),Z}=c^A_{Z,\tau(X)}\circ c^A_{\tau(X),Z}$ is not the identity. Here $c^A$ is the braiding on $\cC_A^\loc$. We will show that the monodromy isomorphism $M_{X,\cG(Z)}=c_{\cG(Z),X}\circ c_{X,\cG(Z)}$ in $\cC$ is also not the identity, and therefore $X$ does not centralize $\cS$.

Let $\cY^A_{\tau(X),Z}:\tau(X)\otimes Z\rightarrow(\tau(X)\boxtimes_A Z)[\log x]\lbrace x\rbrace$ be the tensor product $A$-module intertwining operator of type $\binom{\tau(X)\boxtimes_A Z}{\tau(X)\,\,Z}$. Then from the definition of the braiding isomorphisms in vertex algebraic tensor categories, $M^A_{\tau(X),Z}\neq\id_{\tau(X)\boxtimes_A Z}$ if and only if
\begin{equation*}
    \cY^A_{\tau(X),Z}(-,e^{2\pi i}x)\neq\cY^A_{\tau(X),Z}(-,x)
\end{equation*}
(see for example \cite[Equation 3.15]{CKM-ext}). Consequently, if we write
\begin{equation*}
\cY^A_{\tau(X),Z}(-,x) =\sum_{\lambda+\ZZ\in\CC/\ZZ} \sum_{k=0}^K  x^\lambda (\log x)^k\cY^A_{\lambda,k}(-,x),
\end{equation*}
where $K$ is the highest power of $\log x$ that occurs in $\cY^A_{\tau(X),Z}$ and each $\cY^A_{\lambda,k}: \tau(X)\otimes Z\rightarrow(\tau(X)\boxtimes_A Z)((x))$ is a formal Laurent series, then $\cY^A_{\lambda,K}\neq 0$ for some $\lambda\in\CC$ such that either $\lambda\notin\ZZ$ or $K\neq 0$ (see for example the proof of \cite[Theorem 3.14]{CKL}). Moreover, $x^\lambda\cY^A_{\lambda,K}(-,x)$ is a non-zero intertwining operator of type $\binom{\tau(X)\boxtimes_A Z}{\tau(X)\,\,Z}$ (see \cite[Remark 3.26]{HLZ2}).

Now $x^\lambda\cY^A_{\lambda,K}(i_X(-),x)$ is a non-zero $V$-module intertwining operator of type $\binom{\cG(\tau(X)\boxtimes_A Z)}{X\,\,\cG(Z)}$ by Lemma \ref{lem:nonzerofusion}. Thus $\cY^A_{\tau(X),Z}\circ(i_X\otimes\id_Z)$ is a $V$-module intertwining operator which involves either non-integral powers of $x$ or positive powers of $\log x$. Since there is a unique $V$-module homomorphism $f:X\boxtimes\cG(Z)\rightarrow\cG(\tau(X)\boxtimes_A Z)$ such that $f\circ\cY_{X,\cG(Z)}=\cY^A_{\tau(X),Z}\circ(i_X\otimes\id_Z)$, it follows that the tensor product intertwining operator $\cY_{X,\cG(X)}$ of type $\binom{X\boxtimes\cG(Z)}{X\,\,\cG(Z)}$ also involves non-integral powers of $x$ or positive powers of $\log x$. Thus $M_{X,\cG(Z)}\neq\id_{X\boxtimes\cG(Z)}$ (again by \cite[Equation 3.15]{CKM-ext} and the proof of \cite[Theorem 3.14]{CKL}).
\end{proof}

As a consequence, we have:
 \begin{cor}\label{cor:condition3}
     Under Assumption \ref{ass1}, condition (3) of Theorem \ref{rigidity-of-C-from-CAloc} holds if there is a non-split exact sequence
    \[
    0 \longrightarrow V \xrightarrow{\,\iota_A\,} A \xrightarrow{\,\pi\,} Q \longrightarrow 0
    \]
    where $Q$ is simple in $\cC$ and not isomorphic to $V$.
 \end{cor}
\begin{proof}
The non-split exact sequence for $A$ implies that any
subobject $s: S \hookrightarrow A$ such that $S$ is not contained in $\mathrm{Im}\,\iota_A = V$ is $A$ itself, with $s$ an isomorphism. 
Since $Q$ is not isomorphic to $V$, by Theorem \ref{thm:centralizer} there is an object $Z$ in $\cS$ such that $c_{Z,Q}\neq c_{Q,Z}^{-1}$. Then since $\pi: A\rightarrow Q$ is surjective,
\begin{equation*}
    (\pi\boxtimes\id_Z)\circ c_{Z,A} = c_{Z,Q}\circ(\id_Z\boxtimes\pi) \neq c_{Q,Z}^{-1}\circ(\id_Z\boxtimes\pi) =(\pi\boxtimes\id_Z)\circ c_{A,Z}^{-1},
\end{equation*}
and thus $c_{Z,A}\neq c_{A,Z}^{-1}$ also. Moreover, $s\boxtimes\id_Z: A\boxtimes Z\rightarrow A\boxtimes Z$ is an isomorphism and thus injective, so condition (3) of Theorem \ref{rigidity-of-C-from-CAloc} is satisfied.
\end{proof}

We will show that Assumption \ref{ass1} holds for $V=L_k(\sl_2)$, $A=\Vir_{c_k}\otimes\Pi(0)$, and $\cC=\cC^{\mathrm{wt}}_k(\sl_2)$ in Corollary \ref{cor:condition-3-for-sl2} below.

\subsection{Remarks on simple current extensions}\label{sec:sc}

We now study conformal vertex (super)algebra extensions of rank $2$ Heisenberg vertex operator algebras. The main goal is to prove that $\cC_A^\loc$ is rigid in the case $\cC=\cC^{\mathrm{wt}}_k(\sl_2)$, $A=\Vir_{c_k}\otimes\Pi(0)$, and that condition (3) of Theorem \ref{rigidity-of-C-from-CAloc} holds as well. We also explain how rigidity of weight modules for the $N=2$ super Virasoro algebra follows from rigidity of $\cC^{\mathrm{wt}}_k(\sl_2)$.

Let $A$ be a simple conformal vertex (super)algebra extension of $V \otimes \pi$, where $V$ is some vertex operator algebra and $\pi$ is a rank $1$ Heisenberg vertex operator algebra at non-zero level, of the form
\[
A = \bigoplus_{\ell \in \mathbb Z} V_\ell \otimes \pi_\ell.
\]
We assume that the $V_\ell$ are objects of some braided tensor category $\cC$ of finitely-generated generalized $V$-modules. We also assume that $V$ itself contains a Heisenberg vertex operator subalgebra and that all objects of $\cC$ are weight modules as in Definition \ref{def:wt-mod}. That is, the Heisenberg subalgebra of $V$ acts semisimply and every module in $\cC$ is bi-graded by Heisenberg weight and conformal weight, such that each bi-graded weight space is finite dimensional and there is a lower bound on the conformal weights of each Heisenberg weight space.

Our main example is $A = \Pi(0)$, which is an extension of the rank $2$ Heisenberg vertex operator algebra generated by the fields corresponding to the basis elements $\lbrace c, d\rbrace$ of the lattice $L$. If we change the basis of $\CC\otimes_\ZZ L$ to $\lbrace\mu,\nu\rbrace$ where $c =\frac{2}{k}(\mu - \nu)$ and $d = \mu + \nu$, then $\langle\mu,\mu\rangle=\frac{k}{2}$, $\langle\nu,\nu\rangle=-\frac{k}{2}$, and $\langle\mu,\nu\rangle =0$.
Thus $\mu$ and $\nu$ generate commuting Heisenberg subalgebras of non-zero level for all non-integral admissible $k$, and hence we may take $V$ and $\pi$ to be the Heisenberg subalgebras of $\Pi(0)$ generated by $\mu$ and $\nu$, respectively.

Another example is $A = \mathcal W_\ell(\mathfrak{sl}_{2|1}) \otimes V_{\sqrt{-1}\mathbb Z}$, where $\mathcal W_\ell(\mathfrak{sl}_{2|1})$ is the simple principal $W$-superalgebra of $\mathfrak{sl}_{2|1}$ at level $\ell$, that is, the $N=2$ super Virasoro algebra. The level $\ell$ is related to $k$ via $(\ell+1)(k+2) =1$, and we assume that $k$ is an admissible level for $\mathfrak{sl}_2$. 
In this case, $V = L_k(\mathfrak{sl}_2)$ and the generator of $\pi$ has squared length $- \frac{2}{k+2}$ \cite[Equation 3.7]{CGNS}. This is the Kazama-Suzuki duality \cite{KS}, and it is a special case of the Feigin-Semikhatov duality \cite{FS} that has been proved and studied in \cite{CGN, CGNS}. For the Kazama-Suzuki duality at admissible level, see \cite{CLRW}.

Let $\mathcal C^\pi$ denote the category of $C_1$-cofinite weight $\pi$-modules. This is a rigid semisimple braided tensor category whose simple objects are given by the Fock modules $\pi_\lambda$, $\lambda\in\CC$, on which the zero-mode of the Heisenberg generator of $\pi$ acts by the scalar $\lambda$. The fusion tensor products of simple Fock modules are given by $\pi_\lambda\boxtimes\pi_\mu\cong\pi_{\lambda+\mu}$ for $\lambda,\mu\in\CC$. 
 Let $\mathcal D :=\mathcal C \boxtimes \mathcal C^\pi$ be the Deligne tensor product of $\cC$ and $\cC^\pi$; by \cite{CKM2, Mc}, it is the category of $V\otimes\pi$-modules which decompose as finite direct sums of the form $\bigoplus_\lambda M_\lambda\otimes\pi_\lambda$ where the $M_\lambda$ are objects of $\cC$. Thus $A$ is a conformal vertex (super)algebra in the direct limit completion $\mathrm{Ind}(\cD)$ \cite{CMY4}. 
 
 We are interested in weight $A$-modules $M$, which are tri-graded by conformal weights $\Delta$, weights $\eta$ of the Heisenberg subalgebra of $V$, and weights $\lambda$ of $\pi$. Each tri-graded weight space of $M$ is finite dimensional:
 \[
M = \bigoplus_{\lambda, \eta, \Delta\in\CC} M_{\lambda, \eta, \Delta}, \qquad \dim M_{\lambda, \eta, \Delta} < \infty,
 \]
 and for each $\lambda,\eta\in\CC$, $M_{\lambda,\eta,\Delta}=0$ for $\mathrm{Re}(\Delta)$ sufficiently negative.
 In particular, $M$ has a decomposition
 \begin{equation}\label{eqn:decom-for-E}
 M = \bigoplus_{\lambda\in\CC} M_\lambda \otimes \pi_\lambda
 \end{equation}
such that each $M_\lambda$ is a weight $V$-module. Let $\cE$ denote the subcategory of  $\mathrm{Ind}(\cD)_A^{\loc}$ consisting of all finitely-generated weight $A$-modules $M$ such that in the decomposition \eqref{eqn:decom-for-E}, each $M_\lambda$ is an object of $\cC$.

The  first statement in the following result is essentially the same as one of the statements in \cite[Proposition 3.7]{CMY5} and has essentially the same proof:
\begin{theorem}\label{thm:properties-of-E}
    Any weight $A$-module $M$ in $\cE$ is isomorphic to an induced module $\cF(X)$ for some $X$ in $\cD$. Consequently, $\cE$ is a tensor subcategory of $\mathrm{Ind}(\cD)_A^{\loc}$, and if the category $\cC$ of $V$-modules is locally finite, rigid, or semisimple, then so is $\cE$.
\end{theorem}
\begin{proof}
    We may assume that the zero-mode of the Heisenberg generator of $\pi$ acts by the scalar $\ell\in\ZZ$ on the Fock modules $\pi_\ell$ appearing in the decomposition of $A$ as a $V\otimes\pi$-module. Then any module $M$ in $\cE$ has a decomposition $M=\bigoplus_{\lambda+\ZZ\in\CC/\ZZ} N_{\lambda+\ZZ}$ as an $A$-module, where $N_{\lambda+\ZZ}=\bigoplus_{\ell\in\ZZ} M_{\lambda+\ell}\otimes\pi_{\lambda+\ell}$. Since $M$ is finitely generated, only finitely many $N_{\lambda+\ZZ}$ are non-zero. Thus because $\cF$ is an additive functor, it is enough to show that each non-zero $N_{\lambda+\ZZ}$ is an induced module.

    By the proof of \cite[Proposition 3.7]{CKLR}, if $N_{\lambda+\ZZ}$ is non-zero, then $M_\lambda\otimes\pi_\lambda$ is non-zero and generates $N_{\lambda+\ZZ}$ as an $A$-module. So because $M_\lambda$ is an object of $\cC$ by assumption and thus $M_\lambda\otimes\pi_\lambda$ is an object of $\cD$, the inclusion $M_\lambda\otimes\pi_\lambda\hookrightarrow N_{\lambda+\ZZ}$ induces an $A$-module surjection $\cF(M_\lambda\otimes\pi_\lambda)\twoheadrightarrow N_{\lambda+\ZZ}$ by Frobenius reciprocity (Lemma \ref{reciprocity}). Moreover, there are $V\otimes\pi$-module isomorphisms
    \begin{equation*}
        \cF(M_\lambda\otimes\pi_\lambda)\cong\bigoplus_{\ell\in\ZZ} (V_\ell\otimes\pi_\ell)\boxtimes(M_\lambda\otimes\pi_\lambda)\cong\bigoplus_{\ell\in\ZZ} M_{\lambda+\ell}\otimes\pi_{\lambda+\ell}\cong N_{\lambda+\ZZ},
    \end{equation*}
    where the middle isomorphism comes from \cite[Proposition 3.7]{CKLR}. Thus the simultaneous conformal and Heisenberg weight spaces of $\cF(M_\lambda\otimes\pi_\lambda)$ and $N_{\lambda+\ZZ}$ are isomorphic finite-dimensional vector spaces. Since the $A$-module surjection $\cF(M_\lambda\otimes\pi_\lambda)\twoheadrightarrow N_{\lambda+\ZZ}$ preserves these simultaneous weight spaces, it is also injective and therefore an isomorphism. This proves the first statement of the theorem.

    Because induction is a tensor functor, the fusion tensor product of two objects of $\cE$ is still an induced module in $\mathrm{Ind}(\cD)_A^{\loc}$, of the form $\bigoplus_{\lambda\in\CC}\cF(M_\lambda\otimes\pi_\lambda)$ where each $M_\lambda$ is an object of $\cC$ and only finitely many $M_\lambda$ are non-zero. Each $\cF(M_\lambda\otimes\pi_\lambda)$ is an object of $\cE$ because it is finitely generated as an $A$-module by the finitely many $V\otimes\pi$-module generators of $M_\lambda\otimes\pi_\lambda$, and because
\begin{equation*}
    \cF(M_\lambda\otimes\pi_\lambda)\cong\bigoplus_{\ell\in\ZZ} (V_\ell\boxtimes M_{\lambda})\otimes\pi_{\lambda+\ell}
\end{equation*}
as a $V\otimes\pi$-module, where each $V_\ell\boxtimes M_\lambda$ is an object of $\cC$. Thus $\cE$ is closed under tensor products. As the unit object $A$ of $\mathrm{Ind}(\cD)_A^{\loc}$ is also an object of $\cE$, it follows that $\cE$ is a tensor subcategory of $\mathrm{Ind}(\cD)_A^{\loc}$.

    Now the modules $V_\ell\otimes\pi_\ell$ appearing in the decomposition of $A$ as a $V\otimes\pi$-module are simple currents, and if $W$ is a simple $V\otimes\pi$-module in $\cD$, then $(V_\ell\otimes\pi_\ell)\boxtimes W\cong (V_{\ell'}\otimes\pi_{\ell'})\boxtimes W$ only if $\ell =\ell'$ (because the Heisenberg Fock modules have this property). Thus by \cite[Proposition 4.4]{CKM-ext}, $\cF(W)$ is simple for any simple object $W$ in $\cD$. Moreover, the simple currents $V_\ell\otimes\pi_\ell$ are rigid objects of $\cD$, which implies that $\cF$ is exact. Thus induction preserves the length of finite-length objects of $\cD$, and consequently every object of $\cE$ has finite length if every object of $\cC$ has finite length. Morphism spaces in $\cE$ are also finite dimensional since the finitely many generators of any object in $\cE$ are contained in the direct sum of finitely many finite-dimensional tri-graded weight spaces. Thus $\cE$ is locally finite if $\cC$ is.

If $\cC$ is rigid, then so is $\cD$ because $\cC^\pi$ is rigid, and then so is $\cE$ because induction is a tensor functor and thus inductions of rigid objects are rigid (see for example \cite[Exercise 2.10.6]{EGNO}). Finally, if $\cC$ is semisimple, then so is $\cD$ because $\cC^\pi$ is semisimple, and then so is   $\cE$ because $\cF$ is additive.
\end{proof}

We will apply Theorem \ref{thm:properties-of-E} to the example $A=\Pi(0)$ in the next subsection. For now, we consider the example $A=\cW_\ell(\mathfrak{sl}_{2\vert 1})\otimes V_{\sqrt{-1}\ZZ}$, $V=L_k(\mathfrak{sl}_2)$, $\cC=\cC^{\mathrm{wt}}_k(\mathfrak{sl}_2)$ and show that rigidity for the category of finitely-generated weight modules for the $N=2$ super Virasoro algebra at suitable level follows from rigidity of $\cC^{\mathrm{wt}}_k(\sl_2)$:
\begin{theorem}\label{thm:N=2}
    If $k$ is an admissible level for $\mathfrak{sl}_2$ and $\ell\in\QQ$ is defined by $(\ell+1)(k+2)=1$, then the braided tensor category of finitely-generated weight $\cW_\ell(\mathfrak{sl}_{2\vert 1})$-modules is rigid.
\end{theorem}
\begin{proof}
    Here, the tensor category $\cE$ of Theorem \ref{thm:properties-of-E} is the category of finitely-generated weight $\cW_\ell(\mathfrak{sl}_{2\vert 1})\otimes V_{\sqrt{-1}\ZZ}$-modules $M$ which have $L_k(\mathfrak{sl}_2)\otimes\pi$-module decompositions of the form $M=\bigoplus_{\lambda\in\CC} M_\lambda\otimes\pi_\lambda$ such that each $M_\lambda$ an object of $\cC^{\mathrm{wt}}_k(\mathfrak{sl}_2)$. By Theorems \ref{thm:intro-main-thm} and \ref{thm:properties-of-E}, this category $\cE$ is rigid. Moreover, there is a tensor embedding of the category of finitely-generated weight $\cW_\ell(\mathfrak{sl}_{2\vert 1})$-modules into $\cE$ defined on objects by $M\mapsto M\otimes V_{\sqrt{-1}\ZZ}$. As this functor commutes with contragredient duals, the category of finitely-generated weight $\cW_\ell(\mathfrak{sl}_{2\vert 1})$-modules inherits rigidity from $\cE$ (see \cite[Lemma 3.15]{CMSY}).
\end{proof}

\subsection{Local modules for the free field realization}

In this subsection, we first apply Theorem \ref{thm:properties-of-E} to the example $A=\Pi(0)$, where $V$ and $\pi$ are the Heisenberg subalgebras generated by $\mu$ and $\nu$, respectively, and $\cC=\cC^\pi$ is the category of $C_1$-cofinite weight $V$-modules (so $\cD=\cC^\pi\boxtimes\cC^\pi$).
We then give a comprehensive description of the local module category $\cC^{\mathrm{wt}}_k(\mathfrak{sl}_2)_{\Vir_{c_k}\otimes\Pi(0)}^\loc$ for non-integral admissible levels $k$. Let $\cC^{\mathrm{wt}}_{\Pi(0)}$ denote the category of finitely-generated weight $\Pi(0)$-modules.

\begin{theorem}\label{thm:Pi(0)-tens-cat-properties}
    The category $\cC^{\mathrm{wt}}_{\Pi(0)}$ is a rigid semisimple non-degenerate braided tensor category and every simple object of $\cC^{\mathrm{wt}}_{\Pi(0)}$ is isomorphic to $\Pi_\ell(\lambda)$ for some $\ell\in\ZZ$ and $\lambda+\ZZ\in\CC/\ZZ$. Moreover, fusion tensor products of simple modules are given by
    \begin{equation}\label{eqn:P(0)-fusion-rules}
\Pi_\ell(\lambda) \boxtimes \Pi_{\ell'}(\lambda') \cong \Pi_{\ell+ \ell'}(\lambda + \lambda')
\end{equation}
for $\ell,\ell'\in\ZZ$ and $\lambda+\ZZ,\lambda'+\ZZ\in\CC/\ZZ$.
\end{theorem}
\begin{proof}
    We need to show that $\cC^{\mathrm{wt}}_{\Pi(0)}$ is the same as the category $\cE$ in Theorem \ref{thm:properties-of-E}. The inclusion of $\cE$ into $\cC^{\mathrm{wt}}_{\Pi(0)}$ follows from the definition of $\cE$. Conversely, if $M$ is a finitely-generated weight $\Pi(0)$-module, then $M=\bigoplus_{\lambda\in\CC} M_\lambda\otimes\pi_\lambda$ where each $M_\lambda$ is a weight module for the Heisenberg subalgebra of $\Pi(0)$ generated by $\mu$, that is, $M_\lambda=\bigoplus_{\eta\in\CC} V_{\lambda,\eta}\otimes\pi_\eta$ where each $V_{\lambda,\eta}$ is a finite-dimensional vector space. The finitely many generators of $M$ are contained in the direct sum of finitely many $V_{\lambda,\eta}\otimes\pi_\eta\otimes\pi_\lambda$, and this implies that $M$ is the direct sum of $V_{\lambda+\ell,\eta+\ell}\otimes\pi_{\eta+\ell}\otimes\pi_{\lambda+\ell}$ for finitely many $\lambda,\eta\in\CC$ and all $\ell\in\ZZ$. This implies $M_\lambda$ has finite length for each $\lambda\in\CC$ and is thus an object of the category $\cC^\pi$ of $C_1$-cofinite weight $\pi$-modules. Hence $M$ is an object of $\cE$, as required.

    It now follows from Theorem \ref{thm:properties-of-E} that $\cC^{\mathrm{wt}}_{\Pi(0)}$ is a rigid semisimple braided tensor category, and its simple objects are given by $\cF(\pi_\eta\otimes\pi_\lambda)$ for certain $\eta,\lambda\in\CC$. Changing generators from $\lbrace\mu,\nu\rbrace$ back to $\lbrace c,d\rbrace$ again, let $\pi_{a, b}$ be the Fock module for the rank $2$ Heisenberg subalgebra of $\Pi(0)$ on which the zero-modes of $c$ and $d$ act by the scalars $a$ and $b$, respectively. Then $\Pi_\ell(\lambda)$ contains the Fock submodule $\pi^{c,d}\otimes e^{\ell\mu+\lambda c}\cong\pi_{\ell,2\lambda+\frac{k}{2}\ell}$, so by Lemma \ref{reciprocity}, there is a non-zero $\Pi(0)$-module map $\cF(\pi_{\ell,2\lambda+\frac{k}{2}\ell})\rightarrow\Pi_\ell(\lambda)$ which must be an isomorphism because both modules are simple. Then since $\cF$ is a tensor functor,
    \begin{align*}
        \Pi_\ell(\lambda)\boxtimes\Pi_{\ell'}(\lambda') & \cong\cF(\pi_{\ell,2\lambda+\frac{k}{2}\ell})\boxtimes\cF(\pi_{\ell',2\lambda'+\frac{k}{2}\ell'}) \cong \cF(\pi_{\ell,2\lambda+\frac{k}{2}\ell}\boxtimes\pi_{\ell',2\lambda'+\frac{k}{2}\ell'})\nonumber\\
        &\cong\cF(\pi_{\ell+\ell',2(\lambda+\lambda')+\frac{k}{2}(\ell+\ell')})\cong\Pi_{\ell+\ell'}(\lambda+\lambda')
    \end{align*}
    for $\ell,\ell'\in\ZZ$ and $\lambda+\ZZ,\lambda'+\ZZ\in\CC/\ZZ$.

    To show that every simple object of $\cC^{\mathrm{wt}}_{\Pi(0)}$ is isomorphic to $\Pi_\ell(\lambda)$ for some $\ell\in\ZZ$ and $\lambda+\ZZ\in\CC/\ZZ$, and to show that $\cC^{\mathrm{wt}}_{\Pi(0)}$ is non-degenerate, we need to compute monodromy isomorphisms in $\cC^\pi\boxtimes\cC^\pi$ and $\cC^{\mathrm{wt}}_{\Pi(0)}$. The balancing equation for the ribbon twist $e^{2\pi i L_0}$ implies that the double braidings on both $\pi_{\ell,2\lambda+\frac{k}{2}\ell}\boxtimes\pi_{\ell',2\lambda'+\frac{k}{2}\ell'}\cong\pi_{\ell+\ell',2(\lambda+\lambda')+\frac{k}{2}(\ell+\ell')}$ in $\cC^\pi\boxtimes\cC^\pi$ and on $\Pi_\ell(\lambda)\boxtimes\Pi_{\ell'}(\lambda')\cong\Pi_{\ell+\ell'}(\lambda+\lambda')$ in $\cC^{\mathrm{wt}}_{\Pi(0)}$ are given by the scalar
    \begin{equation*}
        e^{2\pi i(h_{\ell+\ell',2(\lambda+\lambda')+k(\ell+\ell')/2}-h_{\ell,2\lambda+k\ell/2}-h_{\ell',2\lambda'+k\ell'/2})},
    \end{equation*}
    where the conformal weights are given by
     \begin{equation}\label{eqn:conf-wt-Pi(0)}
        h_{\ell,2\lambda+\frac{k}{2}\ell} = \frac{1}{2}\left(\langle \ell\mu+\lambda c,\ell\mu+\lambda c\rangle +\langle d,\ell\mu+\lambda c\rangle -\frac{k}{2}\langle c,\ell\mu+\lambda c\rangle\right) =
        \frac{k}{4}\ell^2+\lambda(\ell+1),
    \end{equation}
    using the conformal vector for $\pi^{c,d}\subseteq\Pi(0)$ in Theorem \ref{thm: conformal embedding}. So $M_{\pi_{\ell,2\lambda+k\ell/2}, \pi_{\ell',2\lambda'+k\ell'/2}}$ and $M_{\Pi_\ell(\lambda),\Pi_{\ell'}(\lambda')}$ are both scalar multiplications by
    \begin{equation}\label{eqn:monodromy}
 e^{2\pi i(k\ell\ell'/2+\lambda\ell'+\lambda'\ell)}
    \end{equation}
    for $\lambda,\lambda'\in\CC$ and $\ell,\ell'\in\CC$ (for the Heisenberg case) or $\ell,\ell'\in\ZZ$ (for the $\Pi(0)$ case).
    
    Now to classify the simple objects of $\cC^{\mathrm{wt}}_{\Pi(0)}$, we need to determine which Fock modules $\pi_{a,b}$ induce to local $\Pi(0)$-modules, that is, objects of $\mathrm{Ind}(\cC^\pi\boxtimes\cC^\pi)_{\Pi(0)}^{\loc}$. To this end, note that for any $a,b\in\CC$, $\pi_{a,b}=\pi_{\ell,2\lambda+\frac{k}{2}\ell}$ for unique $\ell,\lambda\in\CC$. By \cite[Proposition 2.65]{CKM-ext}, the induction of this module is local if and only if the double braiding of $\Pi(0)$ with $\pi_{a,b}$ in $\mathrm{Ind}(\cC^\pi\boxtimes\cC^\pi)$ is the identity. As $\Pi(0)\cong\bigoplus_{n\in\ZZ} \pi_{0,2n}$ as $\pi^{c,d}$-module, $\cF(\pi_{\ell,2\lambda+\frac{k}{2}\ell})$ is an object of $\cC^{\mathrm{wt}}_{\Pi(0)}$ if and only if
    \begin{equation*}
        M_{\pi_{\ell,2\lambda+k\ell/2}, \pi_{0,2n}} = e^{2\pi i n\ell} = 1
    \end{equation*}
    for all $n\in\ZZ$, that is, if and only if $\ell\in\ZZ$. Thus every simple object of $\cC^{\mathrm{wt}}_{\Pi(0)}$ is isomorphic to $\cF(\pi_{\ell,2\lambda+\frac{k}{2}\ell})\cong\Pi_\ell(\lambda)$ for some $\ell\in\ZZ$ and $\lambda+\ZZ\in\CC/\ZZ$.
   
    Finally, to show that $\cC^{\mathrm{wt}}_{\Pi(0)}$ is non-degenerate, it is enough to show that for any $\ell\in\ZZ$ and $\lambda+\ZZ\in\CC/\ZZ$ such that $(\ell,\lambda+\ZZ)\neq(0,\ZZ)$, there exist $\ell'\in\ZZ$ and $\lambda'+\ZZ\in\CC/\ZZ$ such that $M_{\Pi_\ell(\lambda),\Pi_{\ell'}(\lambda')}$ is not the identity. Indeed, if $\ell\neq 0$, then by \eqref{eqn:monodromy} we can take $\ell'=0$ and any $\lambda'\notin\frac{1}{\ell}\ZZ$. Then if $\ell=0$ and $\lambda\notin\ZZ$, we can take $\ell'=1$ and $\lambda'$ arbitrary.
\end{proof}

To simplify notation, we now fix $\cC=\cC_k^{\mathrm{wt}}(\mathfrak{sl}_2)$ and $A=\Vir_{c_k}\otimes\Pi(0)$. Also let $\cC^{\Vir}_k$ denote the modular tensor category of $\Vir_{c_k}$-modules. In the next theorem, we show that $\cC_A^\loc$ is equivalent to the Deligne tensor product of $\cC^\Vir_k$ and $\cC^{\mathrm{wt}}_{\Pi(0)}$:
\begin{theorem}\label{thm:CAloc-rigid}
    If $k=-2+\frac{u}{v}$ is a non-integral admissible level for $\mathfrak{sl}_2$, then $\cC_A^\loc$ is braided tensor equivalent to $\cC^\Vir_{c_k}\boxtimes\cC^{\mathrm{wt}}_{\Pi(0)}$. In particular, $\cC_A^\loc$ is rigid, semisimple, and non-degenerate; any simple object of $\cC_A^\loc$ is isomorphic to $M^k_{r,s}\otimes\Pi_\ell(\lambda)$ for some $1\leq r\leq u-1$, $1\leq s\leq v-1$, $\ell\in\ZZ$, and $\lambda+\ZZ\in\CC/\ZZ$;
   and the fusion rules for simple modules in $\cC_A^\loc$ are given by
   \begin{equation}\label{eqn:A-fusion-rules}
       (M^k_{r,s}\otimes\Pi_\ell(\lambda))\boxtimes_A(M^k_{r',s'}\otimes\Pi_{\ell'}(\lambda'))\cong  \bigoplus_{\stackrel{r'' = \vert r-r'\vert+1}{r+r'+r''\,\text{odd}}}^{\min(r+r'-1,2u-r-r'-1)}\bigoplus_{\stackrel{s'' = \vert s-s'\vert+1}{s+s'+s''\,\text{odd}}}^{\min(s+s'-1,2v-s-s'-1)} M_{r'',s''}^k\otimes\Pi_{\ell+\ell'}(\lambda+\lambda')
   \end{equation}
   for $1\leq r,r'\leq u-1$, $1\leq s,s'\leq v-1$, $\ell,\ell'\in\ZZ$, and $\lambda+\ZZ,\lambda'+\ZZ\in\CC/\ZZ$.
\end{theorem}
\begin{proof}
   Using Theorem \ref{thm:Pi(0)-tens-cat-properties} and \cite[Theorem 5.5]{CKM2}, the Deligne tensor product $\cC^\Vir_k\boxtimes\cC^{\mathrm{wt}}_{\Pi(0)}$ is the rigid semisimple non-degenerate braided tensor category of $\Vir_{c_k}\otimes\Pi(0)$-modules with simple objects $M^k_{r,s}\otimes\Pi_{\ell}(\lambda)$ for $1\leq r\leq u-1$, $1\leq s\leq v-1$ and $\ell\in\ZZ$, $\lambda+\ZZ\in\CC/\ZZ$. The fusion rules of these simple modules follow from \eqref{eqn:Vir-fus-rules} and \eqref{eqn:P(0)-fusion-rules}.
   Moreover, each $M^k_{r,s}\otimes\Pi_\ell(\lambda)$ is an object of $\cC_A^\loc$, so $\cC^\Vir_k\boxtimes\cC^{\mathrm{wt}}_{\Pi(0)}$ is a subcategory of $\cC_A^\loc$. It is also a tensor subcategory since the braided tensor category structures on $\cC^\Vir_k\boxtimes\cC^{\mathrm{wt}}_{\Pi(0)}$ and $\cC_A^\loc$ are both the vertex algebraic one of \cite{HLZ8} (by \cite[Theorem 5.2]{CKM2} and \cite[Theorem 3.65]{CKM-ext}, respectively). 
   Thus it is enough to show that every object of $\cC_A^\loc$ is also in $\cC^\Vir_k\boxtimes\cC^{\mathrm{wt}}_{\Pi(0)}$, and for this we will show that $\cC_A^\loc$ is semisimple and has the same simple objects as $\cC^\Vir_k\boxtimes\cC^{\mathrm{wt}}_{\Pi(0)}$.

    If $M$ is any object of $\cC_A^\loc$, then because $\Vir_{c_k}$ is strongly rational, $M$ decomposes as a direct sum of $\Vir_{c_k}\otimes\Pi(0)$-submodules of the form $M^k_{r,s}\otimes W$, where $W=\Hom_{\Vir_{c_k}}(M^k_{r,s}, M)$ is a $\Pi(0)$-module. So if $M$ is simple, then only one direct summand occurs and $M\cong M^k_{r,s}\otimes W$ for some $1\leq r\leq u-1$, $1\leq s\leq v-1$. 
    We claim that $W$ is a simple $\Pi(0)$-module. Indeed, if $w\in W$ and $v\in M^k_{r,s}$ are any non-zero vectors, then the $\Vir_{c_k}\otimes\Pi(0)$-submodule of $M$ generated by $v\otimes w$ is $M^k_{r,s}\otimes\langle w\rangle$, where $\langle w\rangle$ is the $\Pi(0)$-submodule generated by $w$. Since $M$ is simple, we get $\langle w\rangle =W$, and thus $W$ is a simple $\Pi(0)$-module.

    To show that $W\cong\Pi_\ell(\lambda)$ for some $\ell,\lambda$, we need to show that $W$ is a weight $\Pi(0)$-module. Indeed, since $M\cong M^k_{r,s}\otimes W$ restricts to a weight $L_k(\mathfrak{sl}_2)$-module, and since the Heisenberg zero-mode $\mu_0$ is the same as $\frac{1}{2}h_0$ by Theorem \ref{thm: conformal embedding}, $\mu_0$ acts semisimply on $W$. Also, each $\mu_0$-eigenspace of $W$ has finite-dimensional conformal weight spaces and the conformal weights of each $\mu_0$-eigenspace of $W$ have a lower bound.
    Thus it is enough to show that the linearly independent Heisenberg zero-mode $c_0$ also acts semisimply on $W$.
     In fact, $c_0$ is a $\Pi(0)$-module endomorphism of $W$ since
    \begin{equation*}
        [c_0, Y(v,x)] = Y(c_0 v,x) = 0
    \end{equation*}
    for all $v\in\Pi(0)$ (because $c_0\Pi(0)=0$ due to $c_0 e^{nc} =\langle c,nc\rangle e^{nc}=0$ for all $n\in\ZZ$). Moreover, there exists a $c_0$-eigenvector in $W$ since $c_0$ preserves the finite-dimensional simultaneous $L_0$- and $\mu_0$-eigenspaces of $W$. Thus $c_0$ acts on $W$ by a scalar since $W$ is a simple $\Pi(0)$-module. This completes the proof that $W$ is an object of $\cC^{\mathrm{wt}}_{\Pi(0)}$ and thus $\cC_A^\loc$ has the same simple objects as $\cC^\Vir_k\boxtimes\cC^{\mathrm{wt}}_{\Pi(0)}$.

    To show that $\cC_A^\loc$ is semisimple, every object of $\cC_A^\loc$ has finite length because $\cC$ is locally finite (see for example the proof of \cite[Theorem 3.14]{CMSY}). Thus it is enough to show that any short exact sequence
    \begin{equation*}
        0\longrightarrow M^k_{r,s}\otimes\Pi_{\ell}(\lambda)\longrightarrow M\longrightarrow M^k_{r',s'}\otimes\Pi_{\ell'}(\lambda')\longrightarrow 0
    \end{equation*}
    in $\cC_A^\loc$ splits. Since $\Vir_{c_k}$ is strongly rational, it is enough to consider the case $(r,s)=(r',s')$. We may also assume $\ell=\ell'$ since otherwise the zero-mode $c_0$ acts semisimply with distinct eigenvalues $\ell$ and $\ell'$, and the corresponding $c_0$-eigenspaces are $A$-module direct summands isomorphic to $M^k_{r,s}\otimes\Pi_{\ell}(\lambda)$ and $M^k_{r,s}\otimes\Pi_{\ell'}(\lambda')$, respectively. Similarly, we may further assume $\lambda-\lambda'\in\ZZ$, since otherwise the sum of the $\mu_0$-eigenspaces of $M$ with eigenvalues contained in $\frac{k}{2}\ell+\lambda+\ZZ$ and the sum of the $\mu_0$-eigenspaces with eigenvalues contained in $\frac{k}{2}\ell+\lambda'+\ZZ$ would be direct summands  isomorphic to $M^k_{r,s}\otimes\Pi_{\ell}(\lambda)$ and $M^k_{r,s}\otimes\Pi_{\ell'}(\lambda')$, respectively. Thus we are reduced to showing that any self-extension
    \begin{equation}\label{eqn:self-ext}
        0\longrightarrow M^k_{r,s}\otimes\Pi_{\ell}(\lambda)\longrightarrow M\longrightarrow M^k_{r,s}\otimes\Pi_{\ell}(\lambda)\longrightarrow 0
    \end{equation}
    in $\cC_A$ splits.

    If \eqref{eqn:self-ext} fails to split, then tensoring with the simple current $M^k_{1,1}\otimes\Pi_{-\ell-1}(-\lambda)$ yields a non-split exact sequence
    \begin{equation}\label{eqn:self-ext-2}
        0\longrightarrow M^k_{r,s}\otimes\Pi_{-1}(0)\longrightarrow \widetilde{M}\longrightarrow M^k_{r,s}\otimes\Pi_{-1}(0)\longrightarrow 0
    \end{equation}
    where $\widetilde{M}=M\boxtimes_A(M^k_{1,1}\otimes\Pi_{-\ell-1}(-\lambda))$. Thus it is enough to show that any self-extension $\widetilde{M}$ of $M^k_{r,s}\otimes\Pi_{-1}(0)$ in $\cC_A^\loc$ splits. We first note that $\widetilde{M}\cong M^k_{r,s}\otimes\widetilde{W}$ as an $A$-module, where $\widetilde{W}=\Hom_{\Vir_{c_k}}(M^k_{r,s},\widetilde{M})$. Then because $M^k_{r,s}$ is simple, \eqref{eqn:self-ext-2} is the tensor product of $M^k_{r,s}$ with an exact sequence
    \begin{equation*}
        0\longrightarrow \Pi_{-1}(0)\longrightarrow\widetilde{W}\longrightarrow\Pi_{-1}(0)\longrightarrow 0
    \end{equation*}
    of $\Pi(0)$-modules. Thus since $\cC^{\mathrm{wt}}_{\Pi(0)}$ is semisimple by Theorem \ref{thm:Pi(0)-tens-cat-properties}, it is enough to show that $\widetilde{W}$ is a finitely-generated weight $\Pi(0)$-module. Certainly $\widetilde{W}$ is finitely generated since it has length $2$, so we just need to show that the Heisenberg zero-modes $\mu_0$ and $\nu_0$ act semisimply, where as before $\lbrace\mu,\nu\rbrace$ is the basis of $\CC\otimes_\ZZ L$ such that $c=\frac{2}{k}(\mu-\nu)$ and $d=\mu+\nu$. 

    Since $\cG(\widetilde{M})$ is a weight $L_k(\sl_2)$-module, $\mu_0=\frac{1}{2}h_0$ acts semisimply on $\widetilde{W}$. Moreover, by either \eqref{eqn:FF-affine-sl2-correspondence} or \eqref{eqn:conf-wt-Pi(0)}, $\cG(\widetilde{M})$ is lower bounded with lowest conformal weight $h_{r,s}+\frac{k}{4}=\Delta_{r,s}$. By the classification of finitely-generated lower-bounded weight $L_k(\sl_2)$-modules derived in \cite{ACK} and reviewed in Section \ref{subsec:affine-sl2-wt-mods}, the Virasoro $L_0$ operator acts semisimply on $\widetilde{M}$. Writing the conformal vector of $\Pi(0)$ in Theorem \ref{thm: conformal embedding} in terms of $\mu$ and $\nu$, $L_0$ acts on the lowest conformal weight space of $\widetilde{M}$ by 
    \begin{equation*}
        h_{r,s} +\frac{1}{k}(\mu_0^2-\nu_0^2)+\nu_0 =\Delta_{r,s},
    \end{equation*}
    equivalently
    \begin{equation*}
        \nu_0^2-k\nu_0 =\mu_0^2-\frac{k^2}{4}.
    \end{equation*}
    Since $\mu_0^2$ acts on the lowest conformal weight space of $\widetilde{M}$ semisimply with eigenvalues $n-\frac{k}{2}$ for $n\in\ZZ$, the action of $\nu_0$ on a $\mu_0$-eigenvector for such an eigenvalue satisfies
    \begin{equation*}
        (\nu_0-n)(\nu_0+n-k) =0.
    \end{equation*}
    Since $k\notin\ZZ$, $\nu_0$ has distinct eigenvalues on each $\mu_0$-eigenspace of the lowest conformal weight space of $\widetilde{M}$. Thus $\nu_0$ is diagonalizable on the lowest conformal weight space of $\widetilde{M}$, and since $\widetilde{M}$ is generated as an $A$-module by its lowest conformal weight space, $\nu_0$ acts semisimply on $\widetilde{M}$. In particular, $\widetilde{W}$ is a weight $\Pi(0)$-module and is therefore semisimple. This completes the proof that $\cC_A^\loc$ is semisimple.
\end{proof}

We can now use Theorem \ref{thm:CAloc-rigid} and Corollary \ref{cor:condition3}, to check the third condition of Theorem \ref{rigidity-of-C-from-CAloc} for the case of $\cC=\cC^{\mathrm{wt}}_k(\sl_2)$ and $A=\Vir_{c_k}\otimes\Pi(0)$:
\begin{cor}\label{cor:condition-3-for-sl2}
 If $k$ is a non-integral admissible level, then condition (3) of Theorem \ref{rigidity-of-C-from-CAloc} holds for $\cC=\cC^{\mathrm{wt}}_k(\sl_2)$ and $A=\Vir_{c_k}\otimes\Pi(0)$.
\end{cor}
\begin{proof}
     By Remark \ref{rem:onetone}, each simple object $X$ of $\cC$ embeds into a simple object $\tau(X)$ of $\cC_A^\loc$, such that $\tau(X)\cong A$ if and only if $X\cong L_k(\sl_2)$. 
    Then since $\cC_A^\loc$ is non-degenerate by Theorem \ref{thm:CAloc-rigid}, $\tau(X)$ is M\"{u}ger central in $\cC_A^\loc$ only if $X\cong L_k(\sl_2)$. Thus Assumption \ref{ass1} holds  in this setting. Now because $A$ also satisfies the non-split exact sequence \eqref{eqn:A-as-Lk(sl2)-mod}, Corollary \ref{cor:condition3} implies that condition (3) of Theorem \ref{rigidity-of-C-from-CAloc} holds .
\end{proof}

\section{Rigidity of the tensor category of affine \texorpdfstring{$\mathfrak{sl}_2$}{sl2} weight modules}\label{sec:rigidity}

Fix a non-integral admissible level $k=-2+t=-2+\frac{u}{v}$ for $\widehat{\mathfrak{sl}}_2$. In this section, we show that the category $\cC^{\mathrm{wt}}_k(\sl_2)$ of finitely-generated weight $L_k(\sl_2)$-modules is rigid, proving Theorem \ref{thm:intro-main-thm}.
 First, we compute the fusion tensor product of $\cD_{1,1}^+$ with any lowest-weight module $\cD_{r,s}^-$, which will help with computing the induced module $\cF(A)$ in the following subsections. Then in Section \ref{subsec:v=2}, we prove $\cC^{\mathrm{wt}}_k(\sl_2)$ is rigid in the case $v=2$, at the same time illustrating the general strategy that we will then apply to the technically more involved $v\geq 3$ case in Section \ref{subsec:v>2}.

\subsection{Fusion tensor product of \texorpdfstring{$\cD_{1,1}^+$}{D11+} and \texorpdfstring{$\cD_{r,s}^-$}{Drs-}}\label{sec:fusion}

To show that $\cC^{\mathrm{wt}}_k(\mathfrak{sl}_2)$ is rigid using the extension $L_k(\mathfrak{sl}_2)\hookrightarrow \Vir_{c_k}\otimes\Pi(0)$ and Theorem \ref{rigidity-of-C-from-CAloc}, we will need to compute the fusion tensor product of $\cG(\Vir_{c_k}\otimes\Pi(0))$ with itself in $\cC^{\mathrm{wt}}_k(\mathfrak{sl}_2)$. By \eqref{eqn:A-as-Lk(sl2)-mod}, this largely amounts to computing $\sigma(\cD_{1,1}^+)\boxtimes\sigma(\cD_{1,1}^+)$, which by \eqref{eqn:HW-as-SF-of-LW} and Proposition \ref{prop:tens-prod-and-spec-flow} amounts to computing the fusion product of $\cD_{1,1}^+$ with some $\cD_{r,s}^-$. In fact, it is little extra effort to compute $\cD_{1,1}^+\boxtimes\cD_{r,s}^-$ for all integers $1\leq r\leq u-1$ and $1\leq s\leq v-1$.

First we show that the fusion tensor product of a highest-weight module with a lowest-weight module in $\cC^{\mathrm{wt}}_k(\sl_2)$ is lower bounded, or more generally:
\begin{prop}\label{prop:Drs+_times_Drs-_lower_bounded}
For $0\leq r,r'\leq u-1$ and $1\leq s,s'\leq v-1$, if $\cY$ is a surjective intertwining operator of type $\binom{X}{\cD_{r,s}^+\,\cD_{r',s'}^-}$ for some finitely-generated generalized $L_k(\mathfrak{sl}_2)$-module $X$, then $X$ is lower bounded. 
\end{prop}
\begin{proof}
Since $\cY$ is surjective, $X$ is spanned by coefficients of $x^h(\log x)^0$ in $\cY(w,x)w'$ for $h\in\CC$, $w\in\cD_{r,s}^+$, and $w'\in\cD_{r',s'}^-$ (see for example \cite[Lemma 2.2]{MY-cp1-Vir}). Then using first the associator formula \eqref{eqn:intw_op_it} and then the commutator formula \eqref{eqn:intw_op_comm},
 one can show that $X$ is generated as a module for $(\widehat{\mathfrak{sl}}_2)_-:=\sl_2\otimes t^{-1}\CC[t^{-1}]$ by coefficients of $x^h(\log x)^0$ in $\cY(u,x)u'$, where $h\in\CC$ and $u$, $u'$ range over the lowest conformal weight spaces of $\cD_{r,s}^+$, $\cD_{r',s'}^-$, respectively. That is, if $v_{r,s}\in\cD_{r,s}^+$ is a highest-weight vector and $v_{r',s'}'\in\cD_{r',s'}^-$ is a lowest-weight vector, then $X$ is generated as an $(\widehat{\mathfrak{sl}}_2)_-$-module by powers of $x^h(\log x)^0$ in 
\begin{equation}
\cY(f^m. v_{r,s}, x)e^n. v_{r',s'}'
\end{equation}
for $m,n\in\ZZ_{\geq 0}$.

Now since $e^n f^m. v_{r,s}$ is a multiple of $f^{m-n}. v_{r,s}$ if $n\leq m$ and is $0$ if $n>m$, the $n=0$ case of \eqref{eqn:intw_op_comm} implies that $X$ is generated as a module for $(\widehat{\mathfrak{sl}}_2)_{\leq 0}:=\sl_2\otimes\CC[t^{-1}]\oplus\CC\mathbf{k}$ by powers of $x^h(\log x)^0$ in 
\begin{equation}
\cY(f^m . v_{r,s},x)v_{r',s'}'
\end{equation}
for $m\in\ZZ_{\geq 0}$. Finally, we have 
\begin{equation*}
\cY(f^m. v_{r,s},x)v_{r',s'}' = f_0\cY(f^{m-1}. v_{r,s},x)v_{r',s'}' - \cY(f^{m-1}. v_{r,s},x)f. v_{r',s'}' = f_0\cY(f^{m-1}. v_{r,s},x)v_{r',s'}',
\end{equation*}
so by induction on $m$, $X$ is generated as an $(\widehat{\mathfrak{sl}}_2)_{\leq 0}$-module by powers of $x^h(\log x)^0$ in $\cY(v_{r,s}, x)v_{r',s'}$.

As affine Lie algebra operators in $(\widehat{\mathfrak{sl}}_2)_{\leq 0}$ weakly raise conformal weights, the conformal weights of the coefficients of minimal powers of $x$ in $\cY(v_{r,s}, x)v_{r',s'}$ are minimal conformal weights of $X$, where a conformal weight $h$ of $X$ is minimal if $h-n$ is not a conformal weight of $x$ for any $n\in\ZZ_{\geq 1}$. As $X$ is finitely generated, only finitely many such minimal conformal weights are possible, so $X$ is lower bounded.
\end{proof}

We now consider $\cD_{1,1}^+=\cD_{-t}^+=\cD_{-k-2}^+$, and we let $v_{1,1}$ denote a highest-weight vector in $\cD_{1,1}^+$. A straightforward calculation shows that
\begin{equation}\label{eqn:singular_vector}
s_{1,1}= e_{-1}f_0^2 v_{1,1}-(t+1)h_{-1}f_0 v_{1,1} -t(t+1)f_{-1} v_{1,1}
\end{equation}
is a singular vector that vanishes in $\cD_{1,1}^+$. This singular vector places strong constraints on the fusion product of $\cD_{1,1}^+$ with any $\cD_{r,s}^-$. In fact, we have:
\begin{prop}\label{prop:D11+_Drs-_lowest_weights}
For $1\leq r\leq u-1$ and $1\leq s\leq v-1$, let $\cY$ be a surjective intertwining operator of type $\binom{X}{\cD_{1,1}^+\,\cD_{r,s}^-}$ for some indecomposable (necessarily lower-bounded) generalized $L_k(\mathfrak{sl}_2)$-module $X$. Then the lowest conformal weight of $X$ is either $\Delta_{r,s+1}$ or $\Delta_{r,s-1}$.
\end{prop}
\begin{proof}
Since $X$ is indecomposable and lower bounded, its conformal weights are contained in $\Delta+\ZZ_{\geq 0}$ for some $\Delta\in\CC$. Thus $X=\bigoplus_{n=0}^\infty X(n)$ where $X(n)=X_{[\Delta+n]}$. We use $\pi_0$ to denote the projection from $X$ to its lowest weight space. Thus for homogeneous $w\in\cD_{1,1}^+$ and $w'\in\cD_{r,s}^-$, $\pi_0\cY(w,1)w'$ denotes the coefficient of $x^{h}(\log x)^0$ in $\cY(w,x)w'$ for $h$ such that the conformal weight of the coefficient of $x^{h}$ is $\Delta$. Note that $x^{h}$ will in fact be the minimal power of $x$ in $\cY(w,x)w'$. 

Let $v_{r,s}'\in\cD_{r,s}^-$ be a lowest-weight vector. The proof of Proposition \ref{prop:Drs+_times_Drs-_lower_bounded} shows that the conformal weight $\Delta$ subspace of $X$ is generated as an $\sl_2$-module by the coefficient of $x^h(\log x)^0$ in $\cY(v_{1,1},x)v_{r,s}'$, where $x^h$ is the lowest power of $x$ appearing in $\cY(v_{1,1},x)v_{r,s}'$.
In particular $\pi_0\cY(v_{1,1},1)v_{r,s}'$ is non-zero if $X$ is non-zero, so we just need to find the conformal weight of this vector. Using the associator formula \eqref{eqn:intw_op_it}, we get
\begin{align}\label{eqn:sing_vector_calc1}
0 & = \cY(s_{1,1},x)v_{r,s}'\nonumber\\
& = \sum_{i\geq 0} x^i\left( e_{-i-1}\cY(f_0^2 v_{1,1}, x)v_{r,s}'-(t+1)h_{-i-1} \cY(f_0 v_{1,1}, x)v_{r,s}'-t(t+1) f_{-i-1}\cY(v_{1,1},x)v_{r,s}'\right)\nonumber\\
& \qquad +x^{-1}\left( \cY(f_0^2 v_{1,1}, x)e_0 v_{r,s}'-(t+1) \cY(f_0 v_{1,1}, x)h_0 v_{r,s}'-t(t+1)\cY(v_{1,1},x)f_0 v_{r,s}'\right).
\end{align}
Setting $x=1$ and projecting to $X(0)$, this implies
\begin{equation}\label{eqn:pi0_sing_reln}
\pi_0\cY(f_0^2 v_{1,1},1)e_0 v_{r,s}' = -\lambda_{r,s}(t+1)\pi_0\cY(f_0 v_{1,1},1)v_{r,s}'.
\end{equation}
We will use this relation to constrain the action of the Casimir operator
\begin{equation}
C=e_0f_0+\frac{1}{2}h_0^2+f_0e_0 = 2e_0f_0+\frac{1}{2}h_0(h_0-2)
\end{equation}
on $\pi_0\cY(v_{1,1},1)v_{r,s}'$, which is a generalized eigenvector for $C$ with generalized eigenvalue $\underline{C}=2t\Delta$.

Now, $h_0$ acts on $\pi_0(\cY(v_{1,1},1)v_{r,s}'$ by the scalar $\lambda_{1,1}-\lambda_{r,s}$. Thus $\pi_0\cY(v_{1,1},1)v_{r,s}'$ is a generalized eigenvector for $e_0f_0$. Using the $n=0$ case of the commutator formula \eqref{eqn:intw_op_comm}, we have
\begin{align*}
e_0f_0\cdot\pi_0\cY(v_{1,1},1)v_{r,s}' & = e_0\cdot\pi_0\cY(f_0 v_{1,1},1)v_{r,s}'\\
& =\pi_0\cY(e_0f_0 v_{1,1},1)v_{r,s}' +\pi_0\cY(f_0 v_{1,1},1)e_0 v_{r,s}'\\
& =\lambda_{1,1}\pi_0\cY(v_{1,1},1)v_{r,s}' +\pi_0\cY(f_0 v_{1,1},1)e_0 v_{r,s}'.
\end{align*}
Thus applying $e_0f_0$ again and using \eqref{eqn:pi0_sing_reln} yields
\begin{align*}
(e_0f_0)^2\cdot\pi_0\cY(v_{1,1},1)v_{r,s}' & =\lambda_{1,1}e_0f_0\cdot\pi_0\cY(v_{1,1},1)v_{r,s}'+ e_0\cdot\pi_0\cY(f_0^2 v_{1,1},1)e_0 v_{r,s}'\\
&\qquad+e_0\cdot\pi_0\cY(f_0 v_{1,1},1)f_0e_0 v_{r,s}'\\
& =\lambda_{1,1}e_0f_0\cdot\pi_0\cY(v_{1,1},1)v_{r,s}'-\lambda_{r,s}(t+1)e_0\cdot\pi_0\cY(f_0 v_{1,1},1)v_{r,s}'\\
&\qquad +\lambda_{r,s}e_0\cdot\pi_0\cY(f_0 v_{1,1},1)v_{r,s}'\\
& =\lambda_{1,1}e_0f_0\cdot\pi_0\cY(v_{1,1},1)v_{r,s}' -t\lambda_{r,s}e_0\cdot\pi_0\cY(f_0 v_{1,1},1)v_{r,s}'\\
& =(\lambda_{1,1}-t\lambda_{r,s})e_0f_0\cdot\pi_0\cY(v_{1,1},1)v_{r,s}'.
\end{align*}
Thus the generalized eigenvalue of $e_0f_0$ on $\pi_0\cY(v_{1,1},1)v_{r,s}'$ is either $0$ or $\lambda_{1,1}-t\lambda_{r,s}$. Also, since $\lambda_{1,1}-t\lambda_{r,s}=-t(r-ts)\neq 0$ as $1\leq s\leq v-1$, actually $\pi_0\cY(v_{1,1},1)v_{r,s}'$ is an $e_0f_0$-eigenvector.

It now follows that the lowest conformal weight $\Delta$ of $X$ is given by either
\begin{equation*}
\Delta =\frac{1}{4t}(\lambda_{1,1}-\lambda_{r,s})(\lambda_{1,1}-\lambda_{r,s}-2)
\end{equation*}
or
\begin{equation*}
\Delta =\frac{1}{4t}(4(\lambda_{1,1}-t\lambda_{r,s})+(\lambda_{1,1}-\lambda_{r,s})(\lambda_{1,1}-\lambda_{r,s}-2)).
\end{equation*}
A calculation shows that the first of these conformal weights is $\Delta_{r,s-1}$, and the second is $\Delta_{r,s+1}$.
\end{proof}

Note from the above argument that
\begin{equation}\label{eqn:conf_wt_diff}
\Delta_{r,s+1}-\Delta_{r,s-1} =\frac{1}{t}(\lambda_{1,1}-t\lambda_{r,s}) = -r+ts,
\end{equation}
which is never an integer if $1\leq s\leq v-1$. We use this to prove:
\begin{prop}\label{prop:X+X-}
If $1\leq r\leq u-1$ and $1\leq s\leq v-1$, then $\cD_{1,1}^+\boxtimes\cD_{r,s}^-\cong X_+\oplus X_-$, where $X_{\pm}$ is either $0$ or a lower-bounded generalized module of lowest conformal weight $\Delta_{r,s\pm 1}$. For either sign choice, if $X_\pm\neq 0$, then it is generated by a single $L_0$-eigenvector of conformal weight $\Delta_{r,s\pm1}$ and $h_0$-weight $-\lambda_{r,s-1}$.
\end{prop}
\begin{proof}
Since $\cD_{1,1}^+\boxtimes\cD_{r,s}^-$ has finite length, it is a finite direct sum of indecomposable submodules. By the preceding propositions, each indecomposable submodule has lowest conformal weight $\Delta_{r,s+1}$ or $\Delta_{r,s-1}$, so we can set $X_{\pm}$ to be the sum of all indecomposable summands with conformal weight $\Delta_{r,s\pm1}$. 

Now let $\tilX_+$ be the submodule of $X_+$ generated by the subspace of conformal weight $\Delta_{r,s+1}$. Then $X_+/\tilX_+$ is a lower-bounded generalized $L_k(\mathfrak{sl}_2)$-module with conformal weights contained in $\Delta_{r,s+1}+\ZZ_{\geq 1}$. However, if $X_+/\tilX_+$ is non-zero, then for any indecomposable summand $W$ of $X_+/\tilX_+$, there is a surjective intertwining operator of type $\binom{W}{\cD_{1,1}^+\,\cD_{r,s}^-}$. So the lowest conformal weight of $W$ is $\Delta_{r,s+1}$ or $\Delta_{r,s-1}$ by Proposition \ref{prop:D11+_Drs-_lowest_weights}, but $\Delta_{r,s\pm1}\notin\Delta_{r,s + 1}+\ZZ_{\geq 1}$ by \eqref{eqn:conf_wt_diff}. Thus $X_+/\tilX_+$ must be $0$, showing that $X_+$ is generated by its lowest conformal weight space, and so is $X_-$ by the same argument.

Now to find generators for $X_\pm$, it is enough to find generators for their lowest conformal weight spaces.
In fact, by the proof of Proposition \ref{prop:Drs+_times_Drs-_lower_bounded}, the lowest conformal weight space of $X_\pm$ is generated as an $\sl_2$-module by $\pi_0\cY_\pm(v_{1,1},1)v_{r,s}'$, where $\cY_\pm$ is any surjective intertwining operator of type $\binom{X_\pm}{\cD_{1,1}^+\,\cD_{r,s}^-}$.
By the proof of Proposition \ref{prop:D11+_Drs-_lowest_weights}, this generator is an $L_0$-eigenvector of conformal weight $\Delta_{r,s\pm1}$ and $h_0$-weight $\lambda_{1,1}-\lambda_{r,s}=-\lambda_{r,s-1}$, as required.
\end{proof}

To obtain more information about the summands $X_+$ and $X_-$ of the preceding proposition, we need to know which indecomposable lower-bounded weight $L_k(\mathfrak{sl}_2)$-modules have lowest conformal weight $\Delta_{r,s+1}$ or $\Delta_{r,s-1}$. To this end, note that for $1\leq r\leq u-1$ and $0\leq s\leq v$,
\begin{equation*}
\Delta_{r,s}=\Delta_{r',s'}\longleftrightarrow (r',s')=(r,s)\,\,\text{or}\,\,(u-r,v-s).
\end{equation*}
Using this relation and recalling the list of indecomposable lower-bounded modules in $\cC_k^{\mathrm{wt}}(\mathfrak{sl}_2)$ from Section \ref{subsec:affine-sl2-wt-mods}, we now place a strong upper bound on $X_+$:
\begin{prop}\label{prop:esummand}
    In the setting of Proposition \ref{prop:X+X-}, if the direct summand $X_+$ of $\cD_{1,1}^+\boxtimes\cD_{r,s}^-$ is non-zero, then $1\leq s\leq v-2$ and $X_+=\cE_{-\lambda_{r,s-1},\Delta_{r,s+1}}$. In particular, $X_+=0$ if $s=v-1$.
\end{prop}
\begin{proof}
    First suppose $s=v-1$. Then the lowest conformal weight of $X_+$ is $\Delta_{r,v}=\Delta_{u-r,0}$, and the only indecomposable lower-bounded weight $L_k(\mathfrak{sl}_2)$-module with this lowest conformal weight is $\cL_{u-r,0}$. But $\cL_{u-r,0}$ does not contain any vector of $h_0$-weight $-\lambda_{r,v-2}$ because
    \begin{equation*}
        \lambda_{u-r,0}-(-\lambda_{r,v-2})= (u-r-1)+(r-1-t(v-2)) = 2t-2\notin 2\ZZ.
    \end{equation*}
    Thus by Proposition \ref{prop:X+X-}, $X_+$ has to be $0$ in this case.

    Now suppose $1\leq s\leq v-2$ and assume $X_+\neq 0$. Then by Proposition \ref{prop:X+X-}, $X_+$ is generated by a single lowest conformal weight vector of conformal weight $\Delta_{r,s+1}=\Delta_{u-r,v-s-1}$ and $h_0$-weight $-\lambda_{r,s-1}$. The indecomposable lower-bounded weight $L_k(\mathfrak{sl}_2)$-modules with lowest conformal weight $\Delta_{r,s+1}$ are:
    \begin{equation*}
        \cD_{r,s+1}^\pm, \,\cD_{u-r,v-s-1}^\pm,\, \cE_{r,s+1}^\pm, \cE_{u-r,v-s-1}^\pm,\, \cE_{\lambda,\Delta_{r,s+1}} \,\,(\lambda\neq\pm\lambda_{r,s+1}\,\,\mathrm{mod}\,\,2\ZZ).
    \end{equation*}
    In the first four cases, these modules have $h_0$-weights congruent to $\pm\lambda_{r,s+1}$ modulo $2\ZZ$, but
    \begin{align*}
        \lambda_{r,s+1}-(-\lambda_{r,s-1}) & =r-1-t(s+1)+r-1-t(s-1) = 2(r-1-st)\notin 2\ZZ,\\
        -\lambda_{r,s+1}-(-\lambda_{r,s-1}) & = -r+1+t(s+1)+r-1-t(s-1)= 2t\notin 2\ZZ.
    \end{align*}
    Thus any indecomposable direct summand of $X_+$ must be isomorphic to $\cE_{-\lambda_{r,s-1},\Delta_{r,s+1}}$. Then because $X_+$ is singly-generated and non-zero, it follows that $X_+=\cE_{-\lambda_{r,s-1},\Delta_{r,s+1}}$, as required.
\end{proof}

We have to work harder to place an upper bound on $X_-$:
\begin{prop}\label{prop:dsummand}
    In the setting of Proposition \ref{prop:X+X-}, assume that the direct summand $X_-$ of $\cD_{1,1}^+\boxtimes\cD_{r,s}^-$ is non-zero. Then $X_-=\cD_{r,s-1}^-$; in particular, $X_-=\cL_{r,0}$ if $s=1$.
\end{prop}
\begin{proof}
 Assume $X_-\neq 0$, and first suppose $s=1$. Then by Proposition \ref{prop:X+X-}, $X_-$ is generated by a single lowest conformal weight vector of conformal weight $\Delta_{r,0}$ and $h_0$-weight $-\lambda_{r,0}$. Since $\cL_{r,0}$ is the only indecomposable lower-bounded weight $L_k(\mathfrak{sl}_2)$-module containing such a lowest conformal weight vector, it follows that $X_-=\cL_{r,0}$ in this case.

 Now suppose $2\leq s\leq v-1$. Then by the proof of Proposition \ref{prop:X+X-}, $X_-$ is generated by the lowest conformal weight vector $u:=\pi_0\cY(v_{1,1},1)v_{r,s}'$, where $\cY$ is a surjective intertwining operator of type $\binom{X_-}{\cD_{1,1}^+\,\cD_{r,s}^-}$, $\pi_0$ denotes projection to the lowest conformal weight space of $X_-$, $v_{1,1}\in\cD_{1,1}^+$ is a highest-weight vector, and $v_{r,s}'\in\cD_{r,s}^-$ is a lowest-weight vector. We want to show that $u\in X_-$ is a lowest-weight vector; then $X_-=\cD_{r,s-1}^-$ since this is the only lowest-weight $L_k(\mathfrak{sl}_2)$-module generated by a lowest-weight vector of the correct lowest weight. Thus it is enough to show that $f.u=0$.

 To show that $f.u=0$, recall the singular vector $s_{1,1}\in\cD_{1,1}^+$ from \eqref{eqn:singular_vector}, which leads to the relation \eqref{eqn:sing_vector_calc1}. In \eqref{eqn:sing_vector_calc1}, we first set $x=1$, then apply the projection $\pi_1$ from $X_-$ to the conformal weight space of conformal weight $\Delta_{r,s-1}+1$, and then apply the affine Lie algebra mode $h_1$. We get, using commutator formulas for $\widehat{\mathfrak{sl}}_2$ and intertwining operators:
 \begin{align*}
0 & = h_1e_{-1}\cdot\pi_0\cY(f_0^2 v_{1,1},1)v_{r,s}'-(t+1)h_1h_{-1}\cdot\pi_0\cY(f_0 v_{1,1},1)v_{r,s}'-t(t+1)h_1f_{-1}\cdot\pi_0\cY(v_{1,1},1)v_{r,s}'\nonumber\\
& \qquad +h_1\cdot\pi_1\cY(f_0^2 v_{1,1},1)e_0 v_{r,s}'+(t+1)\lambda_{r,s} h_1\cdot\pi_1\cY(f_0 v_{1,1},1)v_{r,s}'\nonumber\\
& = 2e_0f_0\cdot\pi_0\cY(f_0 v_{1,1},1)v_{r,s}'-2k(t+1)\pi_0\cY(f_0 v_{1,1},1)v_{r,s}'+2t(t+1)f_0\cdot\pi_0\cY(v_{1,1},1)v_{r,s}'\nonumber\\
&\qquad +\pi_0\cY(h_0f_0^2 v_{1,1},1)e_0 v_{r,s}'+(t+1)\lambda_{r,s}\pi_0\cY(h_0f_0 v_{1,1},1)v_{r,s}'\nonumber\\
& = \left(2e_0f_0-2(t-2)(t+1)+2t(t+1)+(t+1)\lambda_{r,s}(\lambda_{1,1}-2)\right)f.u +(\lambda_{1,1}-4)\pi_0\cY(f_0^2 v_{1,1},1)e_0 v_{r,s}'.
 \end{align*}
In this calculation, $\lambda_{1,1}=-t$, and 
\begin{align*}
    \pi_0\cY(f_0^2 v_{1,1},1) e_0 v_{r,s}' & = e_0\cdot\pi_0\cY(f_0^2 v_{1,1},1)v_{r,s}'-\pi_0\cY(e_0f_0^2 v_{1,1},1)v_{r,s}'\nonumber\\
    & = e_0f_0f_0 u-(\lambda_{1,1}+\lambda_{1,1}-2)\pi_0\cY(f_0 v_{1,1},1)v_{r,s}' =(e_0f_0+2(t+1))f. u.
\end{align*}
Thus we get
\begin{align*}
    \left(-(t+2)e_0f_0+4(t+1)-(t+1)(t+2)\lambda_{r,s}-2(t+1)(t+4)\right)f.u =0,
\end{align*}
and dividing by $-(t+2)$ yields
\begin{equation*}
    \left(e_0f_0+(t+1)(\lambda_{r,s}+2)\right)f. u = 0.
\end{equation*}
From the calculations at the end of the proof of Proposition \ref{prop:D11+_Drs-_lowest_weights}, we see that $e_0f_0u =0$, and thus
\begin{equation*}
    e_0f_0f_0 u =f_0e_0f_0u+h_0f_0 u =0+(\lambda_{1,1}-\lambda_{r,s}-2)f. u =-(\lambda_{r,s}+t+2)f. u.
\end{equation*}
Then
\begin{align*}
    0 & = (-\lambda_{r,s}-t-2+(t+1)(\lambda_{r,s}+2))f. u = t(\lambda_{r,s}+1)f. u =t(r-st)f. u.
\end{align*}
 Since $t(r-st)\neq 0$, it follows that $f. u=0$ as required, and therefore $X_-=\cD_{r,s-1}^-$ if $X_-\neq 0$.
\end{proof}

To prove that $X_+$ and $X_-$ are non-zero in most cases, we obtain non-zero intertwining operators via the free field realization $L_k(\mathfrak{sl}_2)\hookrightarrow\Vir_{c_k}\otimes\Pi(0)$, thus completing the calculation of $\cD_{1,1}^+\boxtimes\cD_{r,s}^-$: 

\begin{theorem}\label{thm:nonzero}
For $1\leq r\leq u-1$ and $1\leq s\leq v-1$, we have
\begin{equation*}
    \cD_{1,1}^+\boxtimes\cD_{r,s}^- \cong\begin{cases}
        \cL_{r,0}\,[\oplus\,\cE_{-\lambda_{r,0},\Delta_{r,2}}] & \text{if}\,\,\,s=1\\
        \cD_{r,s-1}^-\oplus\cE_{-\lambda_{r,s-1},\Delta_{r,s+1}} & \text{if}\,\,\, 2\leq s\leq v-2\\
        \cD_{r,v-2}^- & \text{if} \,\,\,s=v-1
    \end{cases}
\end{equation*}
in $\cC_k^{\mathrm{wt}}(\mathfrak{sl}_2)$, where the term in brackets does not occur when $v=2$.
\end{theorem}
\begin{proof}
By Remark \ref{rem:onetone}, $\cD_{1,1}^+$ embeds into the simple $\Vir_{c_k}\otimes\Pi(0)$-module $M_1$, where 
\begin{equation*}
M_1 =\begin{cases}
M^k_{1,1}\otimes\Pi_1(\nu_{u-1,1}) & \text{if}\,\,\, v=2\\
M^k_{1,2}\otimes\Pi_0(\nu_{1,2}) & \text{if}\,\,\, v\geq 3\\
\end{cases} .
\end{equation*}
Similarly, since $\cD_{r,s}^- =\sigma^{-1}(\cD_{u-r,v-s-1}^+)$ if $1\leq s\leq v-2$ and $\cD_{r,s}^- =\sigma^{-2}(\cD_{r,s}^+)$ if $s=v-1$ by \eqref{eqn:HW-as-SF-of-LW}, $\cD_{r,s}^-$ embeds into the simple $\Vir_{c_k}\otimes\Pi(0)$-module $M_2$, where
\begin{equation*}
M_2 = M^k_{r,s}\otimes\Pi_{-1}(\nu_{u-r,v-s}).
\end{equation*}
The $\Vir_{c_k}$-modules contained in $M_1$ and $M_2$ have fusion rules
\begin{equation*}
M^k_{1,1}\boxtimes M^k_{r,s} \cong M^k_{r,s},\qquad M^k_{1,2}\boxtimes M^k_{r,s}\cong\begin{cases}
    M^k_{r,2} & \text{if}\,\,\,s=1\\
    M^k_{r,s-1}\oplus M^k_{r,s+1} & \text{if}\,\,\,2\leq s\leq v-2\\
    M^k_{r,v-2} & \text{if}\,\,\,s=v-1
\end{cases} .
\end{equation*}
by \eqref{eqn:Vir-fus-rules}, while for any $\ell,\ell'\in\ZZ$ and $\lambda,\lambda'\in\CC$, there is a non-zero $\Pi(0)$-module intertwining operator of type $\binom{\Pi_{\ell+\ell'}(\lambda+\lambda')}{\Pi_\ell(\lambda)\,\,\Pi_{\ell'}(\lambda')}$ by \eqref{eqn:P(0)-fusion-rules}.

In the case $v=2$, and thus $s=1$, we get a non-zero $\Vir_{c_k}\otimes\Pi(0)$-module intertwining operator $\cY$ of type $\binom{M_3}{M_1\,M_2}$, where
\begin{equation*}
    M_3 = M^k_{r,1}\otimes \Pi_0(\nu_{u-1,1}+\nu_{u-r,1}) = M^k_{r,1}\otimes \Pi_0(\nu_{r,1}).
\end{equation*}
Thus by Lemma \ref{lem:nonzerofusion} and \eqref{eqn:FF-affine-sl2-correspondence}, there is a non-zero $L_k(\mathfrak{sl}_2)$-module intertwining operator of type $\binom{\sigma(\cE^-_{u-r,1})}{\cD_{1,1}^+\,\,\cD_{r,1}^-}$, and there is an $L_k(\mathfrak{sl}_2)$-module surjection from $\cD_{1,1}^+\boxtimes\cD_{r,s}^-$ onto the image of this intertwining operator. By \eqref{eqn:exact-seq-E+E-}, this image is either $\sigma(\cD_{u-r,1}^-)\cong\cL_{r,0}$ (recall \eqref{eqn:HW-as-SF-of-LW}) or $\sigma(\cE_{u-r,1}^-)$ itself. By Proposition \ref{prop:dsummand}, the first case must occur and the direct summand $X_-$ of $\cD_{1,1}^+\boxtimes\cD_{r,1}^-$ is $\cL_{r,0}$. On the other hand, the direct summand $X_+$ of $\cD_{1,1}^+\boxtimes\cD_{r,1}^-$ is $0$ by Proposition \ref{prop:esummand}.

When $v\geq 3$ and $1\leq s\leq v-2$, we get a non-zero $\Vir_{c_k}\otimes\Pi(0)$-module intertwining operator of type $\binom{M_3^+}{M_1\,M_2}$, where
\begin{equation*}
M_3^+ = M^k_{r,s+1}\otimes\Pi_{-1}(\nu_{1,2}+\nu_{u-r,v-s}) = M^k_{r,s+1}\otimes\Pi_{-1}(\nu_{u-r,v-s+1}).
\end{equation*}
Since $\nu_{u-r,v-s+1}\neq\nu_{r,s+1},\nu_{u-r,u-s-1}$ mod $\ZZ$ when $1\leq s\leq v-2$, and also $2\nu_{u-r,v-s+1}-k=-\lambda_{r,s-1}$, Lemma \ref{lem:nonzerofusion} and \eqref{eqn:FF-affine-sl2-correspondence} imply that there is a non-zero $L_k(\mathfrak{sl}_2)$-module intertwining operator of type $\binom{\cE_{-\lambda_{r,s-1},\Delta_{r,s+1}}}{\cD_{1,1}^+\,\,\cD_{r,s}^-}$. Thus Proposition \ref{prop:esummand} shows that the direct summand $X_+$ of $\cD_{1,1}^+\boxtimes\cD_{r,s}^-$ is $\cE_{-\lambda_{r,s-1},\Delta_{r,s+1}}$ for $1\leq s\leq v-2$ and $0$ for $s=v-1$.

Similarly, when $v\geq 3$ and $2\leq s\leq v-1$, we get a non-zero $\Vir_{c_k}\otimes\Pi(0)$-module intertwining operator of type $\binom{M_3^-}{M_1\,M_2}$, where
\begin{equation*}
    M_3^- = M^k_{r,s-1}\otimes\Pi_{-1}(\nu_{u-r,v-s+1}).
\end{equation*}
Thus by Lemma \ref{lem:nonzerofusion} and \eqref{eqn:FF-affine-sl2-correspondence}, there is a non-zero $L_k(\mathfrak{sl}_2)$-module intertwining operator of type $\binom{\cE_{r,s}^-}{\cD_{1,1}^+\,\,\cD_{r,s}^-}$. So by \eqref{eqn:exact-seq-E+E-} and Proposition \ref{prop:dsummand}, the direct summand $X_-$ of $\cD_{1,1}^+\boxtimes\cD_{r,s}^-$ is $\cD_{r,s-1}^-$ when $2\leq s\leq v-1$. For the case $s=1$, there are isomorphisms of spaces of intertwining operators of the following types:
\begin{equation*}
    \binom{\cL_{r,0}}{\cD_{1,1}^+\,\cD_{r,1}^-}\longleftrightarrow\binom{\cD_{r,1}^+}{\cD_{1,1}^+\,\cL_{r,0}}\longleftrightarrow\binom{\sigma(\cD_{u-r,v-2}^-)}{\cD_{1,1}^+\,\sigma(\cD_{u-r,v-1}^-)}\longleftrightarrow\binom{\cD_{u-r,v-2}^-}{\cD_{1,1}^+\,\cD_{u-r,v-1}^-}.
\end{equation*}
The first isomorphism comes from \cite[Proposition 3.46]{HLZ2}, the second from \eqref{eqn:HW-as-SF-of-LW}, and the third from Proposition \ref{prop:tens-prod-and-spec-flow}. Since the last space of intertwining operators is non-zero, so is the first, and we get $X_-=\cL_{r,0}$ when $s=1$.
\end{proof}

As a consequence of the preceding theorem, we can determine the fusion product of the simple quotient in the exact sequence \eqref{eqn:A-as-Lk(sl2)-mod} with itself:
\begin{cor}\label{cor:fusionofd11}
In $\cC_k^{\mathrm{wt}}(\mathfrak{sl}_2)$,
\begin{equation*}
\sigma(\cD_{1,1}^+) \boxtimes \sigma(\cD_{1,1}^+) \cong 
\begin{cases}
\sigma^4(\cL_{1,0}) &\text{if}\,\,\, v = 2\\
\sigma^2(\cD_{1,2}^+)\oplus\sigma^3(\cE_{\lambda_{1,3},\Delta_{1,1}})  &\text{if}\,\,\, v \geq 3\\
\end{cases}.
\end{equation*}
\end{cor}
\begin{proof}
If $v=2$, then $\cD_{1,1}^+=\sigma^2(\cD_{1,1}^-)$ (recall \eqref{eqn:HW-as-SF-of-LW}), so
\begin{equation*}
\sigma(\cD_{1,1}^+)\boxtimes\sigma(\cD_{1,1}^+)\cong\sigma^4(\cD_{1,1}^+\boxtimes\cD_{1,1}^-)\cong\sigma^4(\cL_{1,0})
\end{equation*}
by Proposition \ref{prop:tens-prod-and-spec-flow} and Theorem \ref{thm:nonzero}. If $v\geq 3$, then $\cD_{1,1}^+=\sigma(\cD_{u-1,v-2}^-)$, so 
\begin{align*}
\sigma(\cD_{1,1}^+)\boxtimes\sigma(\cD_{1,1}^+) & \cong\sigma^3(\cD_{1,1}^+\boxtimes\cD_{u-1,v-2}^-)\nonumber\\
& \cong \sigma^3(\cD_{u-1,v-3}^-\oplus\cE_{-\lambda_{u-1,v-3},\Delta_{u-1,v-1}})\cong\sigma^2(\cD_{1,2}^+)\oplus\sigma^3(\cE_{\lambda_{1,3},\Delta_{1,1}}),
\end{align*}
using Proposition \ref{prop:tens-prod-and-spec-flow}, Theorem \ref{thm:nonzero}, \eqref{eqn:HW-as-SF-of-LW}, and \eqref{eqn:lambda-Delta-symmetries}.
\end{proof}

\subsection{The strategy and its illustration in the case \texorpdfstring{$v = 2$}{v=2}}\label{subsec:v=2}

In this subsection and the next, we prove that $\cC^{\mathrm{wt}}_k(\sl_2)$ is rigid for any non-integral admissible level $k=-2+\frac{u}{v}$. For briefer notation, we set $V=L_k(\sl_2)$, $A=\Vir_{c_k}\otimes\Pi(0)$, $\cC=\cC^{\mathrm{wt}}_k(\sl_2)$, and $N= \cF(A)$. The induced module $N$ appears in the following useful lemma, which is a special case of \cite[Lemma 2.8]{CLR}:
\begin{lemma}\label{lemma:CLR}
    For any object $M$ in $\cC_A$, $A\boxtimes\cG(M)\cong\cG(N\boxtimes_A M)$ as objects of $\cC$.
\end{lemma}

For any object $X$ in $\cC$, define $u_X: X\rightarrow A\boxtimes X$ to be the composition $(\iota_A\boxtimes\id_X)\circ l_X^{-1}$, where $\iota_A$ is the inclusion of $V$ into $A$ and $l_X$ is the left unit isomorphism. We will need the following lemma:
\begin{lemma}\label{lem:check-4b}
    If $X$ is an object of $\cC$ and there is a $\cC$-embedding $f: X\hookrightarrow M$ for some $M$ in $\cC_A$, then $u_X: X\rightarrow A\boxtimes X$ is injective.
\end{lemma}
\begin{proof}
First, if $\mu_M: A\boxtimes M\rightarrow M$ denotes the $A$-action on $M$, then the unit property of objects of $\cC_A$ implies $\mu_M\circ u_M=\id_M$, so $u_M$ is injective. 
Then as the diagram
 \begin{equation*}
\begin{tikzcd}[column sep=3pc]
 X \ar[d, "u_X"] \ar[r, "f"] & \ar[d, "u_M"] M \\ %\ar[r] & \ar[d]_{u_Z} Z  \ar[r] & 0\\
 A \boxtimes X \ar[r, "\id_A\boxtimes f"]  & A \boxtimes M    %\ar[r] & A \boxtimes Z \ar[r]\ar[d] & 0  \\
    \end{tikzcd}
\end{equation*} 
commutes by naturality of the left unit isomorphisms, $u_X$ is also injective.
\end{proof}

Set $Q=\sigma(\cD_{1,1}^+)$, so that by \eqref{eqn:A-as-Lk(sl2)-mod}, we have a non-split exact sequence
\begin{equation}\label{eqn:A-exact-seq-sec-4}
    0\longrightarrow V\xrightarrow{\,\iota_A\,} A\longrightarrow Q\longrightarrow 0
\end{equation}
in $\cC$. Then since $\boxtimes$ is right exact, if $X$ in $\cC$ embeds into some object of $\cC_A$, then Lemma \ref{lem:check-4b} implies that
\begin{equation*}
0 \longrightarrow X \xrightarrow{\,u_X\,} A \boxtimes X \longrightarrow Q \boxtimes X \longrightarrow 0
\end{equation*}
is also exact.

For proving $\cC$ is rigid,
it turns out that the case $v=2$ is much simpler than the $v\geq 3$ case, so we start by presenting the proof strategy and immediately illustrating it in the case $v=2$, that is, $k = -2 + \frac{u}{2}$ for $u \in \ZZ_{\geq 2}$ odd. In this case, the exact sequence \eqref{eqn:A-exact-seq-sec-4} becomes  
\[
0 \longrightarrow V \longrightarrow A \longrightarrow \sigma^2(\cL_{u-1, 0}) \longrightarrow 0.
\]

{\bf Step 1:} \textit{Show that $N=\cF(A)$ has a composition series with all composition factors in $\cC_A^{\loc}$.} When $v=2$, the special case $X = M = A$ of Lemma \ref{lem:check-4b} gives an exact sequence
\begin{equation}\label{eqn:A-times-A-exact-seq-v=2}
0 \longrightarrow A \xrightarrow{\,u_A\,} A \boxtimes A \longrightarrow A \boxtimes  \sigma^2(\cL_{u-1, 0}) \longrightarrow 0.
\end{equation}
This sequence splits because the multiplication map $\mu_A: A\boxtimes A\rightarrow A$ is a one-sided inverse of $u_A$. Moreover, by \cite[Corollary~7.6]{CHY}, $\cL_{u-1, 0}$ is a simple current with $\cL_{u-1, 0} \boxtimes \cL_{u-1, 0} \cong V$, and thus by Proposition \ref{prop:tens-prod-and-spec-flow}, $\sigma^2(\cL_{u-1,0})$ is a simple current as well. In particular, by \cite[Proposition 2.5]{CKLR}, tensoring \eqref{eqn:A-exact-seq-sec-4} with $\sigma^2(\cL_{u-1,0})$ yields another non-split exact sequence
\[
0 \longrightarrow \sigma^2(\cL_{u-1, 0})  \longrightarrow A \boxtimes \sigma^2(\cL_{u-1, 0})  \longrightarrow \sigma^4(V) \longrightarrow 0.
\]
This implies $\sigma^{-3}(A\boxtimes\sigma^2(\cL_{u-1,0}))$ is an indecomposable lower-bounded weight $L_k(\sl_2)$-module with socle $\sigma^{-1}(\cL_{u-1,0})\cong\cD_{1,1}^-$ and top $\sigma(V)\cong\cD_{u-1,1}^+$. The classification of indecomposable lower-bounded weight $L_k(\sl_2)$-modules (recall \eqref{eqn:exact-seq-E+E-}) thus forces $A \boxtimes \sigma^2(\cL_{u-1, 0}) \cong \sigma^3(\cE^-_{1, 1})$. Now \eqref{eqn:A-times-A-exact-seq-v=2} implies
\[
A \boxtimes A \cong A \oplus  \sigma^3(\cE^-_{1, 1}) \cong A\oplus\cG(M^k_{1,1}\otimes\Pi_2(t))
\]
as an object of $\cC$.
By Frobenius reciprocity (Lemma \ref{reciprocity}), for any $r$, $\ell$, and $\lambda$,
\begin{equation}
\begin{split}
\Hom_{\cC_A}(N, M^k_{r, 1} \otimes \Pi_\ell(\lambda)) &\cong \Hom_{\cC}(A, \cG(M^k_{r, 1} \otimes \Pi_\ell(\lambda))) 
\end{split}
\end{equation}
and the latter is non-zero if and only if $M^k_{r, 1} \otimes \Pi_\ell(\lambda)$ is either $A$ or $\tau(Q) = M^k_{1, 1} \otimes \Pi_2(t)$ (recall Remark \ref{rem:onetone} and \eqref{eqn:A-exact-seq-sec-4}). 
It follows that
\begin{equation}\label{eqn:N-v=2}
N \cong A \oplus M^k_{1, 1} \otimes \Pi_2(t)
\end{equation}
as an object of $\cC_A$. In particular, all composition factors of $N$ are objects of $\cC_A^\loc$.

\medskip

{\bf Step 2:} \textit{Show that every simple object in $\cC_A$ is local.} Still assuming $v=2$, let $X$ be a simple object in $\cC_A$. Then the socle of $\cG(X)$ is a simple object of $\cC$ and thus by Remark \ref{rem:onetone} is the simple quotient of $\cG(M^k_{r, 1} \otimes \Pi_{\ell}(\lambda))$ for certain $r$, $\ell$, $\lambda$. Hence
\[
\Hom_{\cC_A}(\cF(\cG(M^k_{r, 1} \otimes \Pi_{\ell}(\lambda))), X) \cong  
\Hom_{\cC}(\cG(M^k_{r, 1} \otimes \Pi_{\ell}(\lambda)), \cG(X)) \neq 0,
\]
so it is enough to show that $\cF(\cG(M^k_{r, 1} \otimes \Pi_{\ell}(\lambda)))$ has a composition series consisting of local modules. Using Lemma \ref{lemma:CLR}, \eqref{eqn:N-v=2}, and \eqref{eqn:A-fusion-rules}, we get
\[
A \boxtimes \cG(M^k_{r, 1} \otimes \Pi_{\ell}(\lambda)) \cong \cG(N \boxtimes_A (M^k_{r, 1} \otimes \Pi_{\ell}(\lambda))) \cong \cG(M^k_{r, 1} \otimes \Pi_{\ell}(\lambda)) \oplus \cG(M^k_{r, 1} \otimes \Pi_{\ell+2}(\lambda+t))
\]
as objects of $\cC$.
Thus because
\[
\Hom_{\cC_A}(\cF(\cG(M^k_{r, 1} \otimes \Pi_{\ell}(\lambda))), M^k_{r, 1} \otimes \Pi_{\ell}(\lambda)) \cong  
\Hom_{\cC}(\cG(M^k_{r, 1} \otimes \Pi_{\ell}(\lambda)), \cG(M^k_{r, 1} \otimes \Pi_{\ell}(\lambda))) \cong \mathbb C
\]
by Frobenius reciprocity, there is a short exact sequence
\[
0 \longrightarrow Y \longrightarrow \cF(\cG(M^k_{r, 1} \otimes \Pi_{\ell}(\lambda))) \longrightarrow M^k_{r, 1} \otimes \Pi_{\ell}(\lambda) \longrightarrow 0 
\]
in $\cC_A$ such that
$\cG(Y) \cong \cG(M^k_{r, 1} \otimes \Pi_{\ell+2}(\lambda+t))$.
Similarly, there is a short exact sequence
\[
0 \longrightarrow Z \longrightarrow \cF(\cG(M^k_{r, 1} \otimes \Pi_{\ell+2}(\lambda+t))) \longrightarrow M^k_{r, 1} \otimes \Pi_{\ell+2}(\lambda+t) \longrightarrow 0 
\]
in $\cC_A$ such that
$\cG(Z) \cong \cG(M^k_{r, 1} \otimes \Pi_{\ell+4}(\lambda+2t))=\cG(M^k_{r,1}\otimes\Pi_{\ell+4}(\lambda))$.

We still need to show that 
$Y \cong M^k_{r, 1} \otimes \Pi_{\ell+2}(\lambda+t)$ as objects of $\cC_A$.
Since 
\begin{equation}\nonumber
\begin{split}
&\Hom_{\cC_A}(\cF(\cG(M^k_{r, 1} \otimes \Pi_{\ell+2}(\lambda+t))), \cF(\cG(M^k_{r, 1} \otimes \Pi_{\ell}(\lambda)) \\
&\qquad\qquad\qquad \cong \Hom_{\cC}(\cG(M^k_{r, 1} \otimes \Pi_{\ell+2}(\lambda+t)), \cG(M^k_{r, 1} \otimes \Pi_{\ell}(\lambda)) \oplus \cG(M^k_{r, 1} \otimes \Pi_{\ell+2}(\lambda+t))) \cong \mathbb C,
\end{split}
\end{equation}
 and since the image of a  non-zero map from $\cF(\cG(M^k_{r, 1} \otimes \Pi_{\ell+2}(\lambda+t)))$ to $\cF(\cG(M^k_{r, 1} \otimes \Pi_{\ell}(\lambda)))$ must be $M^k_{r, 1} \otimes \Pi_{\ell+2}(\lambda+t)$, we get $Y \cong M^k_{r, 1} \otimes \Pi_{\ell+2}(\lambda+t)$.
Moreover, from Frobenius reciprocity,
\begin{equation}\nonumber
\begin{split}
\Hom_{\cC_A}(\cF(\cG(M^k_{r, 1} \otimes \Pi_{\ell}(\lambda))), M^k_{r, 1} \otimes \Pi_{\ell+2}(\lambda+t)) &\cong 
\Hom_{\cC}(\cG(M^k_{r, 1} \otimes \Pi_{\ell}(\lambda)), \cG(M^k_{r, 1} \otimes \Pi_{\ell+2}(\lambda+t))) \\ &\ \cong \begin{cases}
     \CC & \text{if}\,\,\lambda \in \{ \nu_{r, 1}, \nu_{u-r, 1}\} \\ \,0 & \text{otherwise} \\ 
\end{cases}
\end{split}
\end{equation}
 So $\cF(\cG(M^k_{r, 1} \otimes \Pi_{\ell}(\lambda)))$  is indecomposable if and only if $\lambda\neq\nu_{r,1},\nu_{u-r,1}$ mod $\ZZ$, that is, 
 \begin{equation}\label{eqn:F-of-typical-v=2}
0 \longrightarrow M^k_{r, 1} \otimes \Pi_{\ell+2}(\lambda+t) \longrightarrow \cF(\cG(M^k_{r, 1} \otimes \Pi_{\ell}(\lambda))) \longrightarrow M^k_{r, 1} \otimes \Pi_{\ell}(\lambda) \longrightarrow 0 
\end{equation}
is a non-split short exact sequence in $\cC_A$ when $\lambda\neq\nu_{r,1},\nu_{u-r,1}$ mod $\ZZ$. In any case, every simple quotient of $\cF(\cG(M^k_{r, 1} \otimes \Pi_{\ell}(\lambda)))$ is local for any $r$, $\ell$, $\lambda$, and thus every simple object of $\cC_A$ is local.

Step 2 and Theorem \ref{thm:CAloc-rigid} now show that condition (1) of Theorem \ref{rigidity-of-C-from-CAloc} holds in the case $\cC=\cC^{\mathrm{wt}}_k(\sl_2)$, $A=\Vir_{c_k}\otimes\Pi(0)$ for $k=-2+\frac{u}{2}$, $u$ odd. It then follows from \cite[Theorem 3.14]{CMSY} that $\cC_A$ is rigid.

\medskip

{\bf Step 3:} \textit{Show that if $X$ in $\cC$ is simple, then any non-zero $\cC_A$-morphism from $\cF(X')$ to $\cF(X)^*$ is an isomorphism.} Continuing to assume $v=2$, first consider the typical case $X=\sigma^{\ell+1}(\cE_{2\lambda-k,\Delta_{r,s}})\cong\cG(M^k_{r,1}\otimes\Pi_\ell(\lambda))$ for $1\leq r\leq u-1$, $\ell\in\ZZ$, and $\lambda \neq \nu_{r, 1},\nu_{u-r,1}$ mod $\ZZ$. The dual of $M^k_{r,1}\otimes\Pi_\ell(\lambda)$ in $\cC_A$ is $M^k_{r,1}\otimes\Pi_{-\ell}(-\lambda)$ by the 
self-duality of $M^k_{r,1}$ in $\cC^\Vir_k$,
the simple current fusion rules \eqref{eqn:P(0)-fusion-rules} for weight $\Pi(0)$-modules, and the
result in Theorem \ref{thm:CAloc-rigid} that $\cC_A^\loc$ is tensor equivalent to $\cC^\Vir_k\boxtimes\cC^{\mathrm{wt}}_{\Pi(0)}$. Thus \eqref{eqn:F-of-typical-v=2} implies that there is a non-split short exact sequence
\[
0 \longrightarrow M_{r,1}^k\otimes \Pi_{-\ell}(-\lambda) \longrightarrow 
\cF(\cG(M_{r,1}^k\otimes \Pi_{\ell}(\lambda)))^* \longrightarrow    M_{r,1}^k\otimes \Pi_{-\ell -2}(-\lambda-t)\longrightarrow 0.
\]
On the other hand, 
$\cG(M_{r,1}^k\otimes \Pi_{\ell}(\lambda))'  \cong \cG(M_{r,1}^k\otimes \Pi_{-\ell-2}(t-\lambda))$ by \eqref{eq:contragedientdual typical},
so \eqref{eqn:F-of-typical-v=2} also yields a non-split short exact sequence
\[
0 \longrightarrow M_{r,1}^k\otimes \Pi_{-\ell}(2t-\lambda) \longrightarrow 
\cF(\cG(M_{r,1}^k\otimes \Pi_{\ell}(\lambda))') \longrightarrow    M_{r,1}^k\otimes \Pi_{-\ell -2}(t-\lambda)\rightarrow 0.
\]
Since $2t=u\in\ZZ$, it follows that $\cF(\cG(M_{r,1}^k\otimes \Pi_{\ell}(\lambda))')$ and $\cF(\cG(M_{r,1}^k\otimes \Pi_{\ell}(\lambda)))^*$ are both indecomposable modules with the same composition factors. Thus if a $\cC_A$-morphism $f: \cF(\cG(M_{r,1}^k\otimes \Pi_{\ell}(\lambda))')\rightarrow \cF(\cG(M_{r,1}^k\otimes \Pi_{\ell}(\lambda)))^*$ is non-zero, then $f$ must be an isomorphism, because otherwise the image of $f$ would be isomorphic to $M^k_{r,1}\otimes\Pi_{-\ell-2}(t-\lambda)$, which is not a submodule of $\cF(\cG(M_{r,1}^k\otimes \Pi_{\ell}(\lambda)))^*$.

It now follows that $\cF(\cG(M_{r,1}^k\otimes \Pi_{\ell}(\lambda))')$ and $\cF(\cG(M_{r,1}^k\otimes \Pi_{\ell}(\lambda)))^*$ are indeed isomorphic in $\cC_A$, because the proof of \cite[Theorem 6.11]{MY2} (see also \cite[Theorem 3.16]{CMSY}) shows that a non-zero morphism between these modules exists.

Now we consider the atypical case $X=\sigma^{\ell}(\cD_{r,1}^+)$ for $1\leq r\leq u-1$ and $\ell\in\ZZ$. Now $\sigma^\ell(\cD_{r,1}^+)$ is a homomorphic image of $\sigma^\ell(\cE_{u-r,1}^-)\cong\cG(M^k_{r,1}\otimes\Pi_{\ell-1}(\nu_{r,1})$,
so because the exact sequence \eqref{eqn:F-of-typical-v=2} splits in this case, we get a $\cC_A$-surjection
\[
\cF(\cG(M_{r,1}^k\otimes \Pi_{\ell-1}(\nu_{r,1}))) \cong M_{r,1}^k\otimes \Pi_{\ell-1}(\nu_{r,1}) \oplus M_{r,1}^k\otimes \Pi_{\ell +1}(\nu_{u-r,1}) \twoheadrightarrow \cF(\sigma^\ell(\cD_{r,1}^+)).
\]
Thus $\cF(\sigma^\ell(\cD_{r,1}^+))$ is semisimple with at most two composition factors. By Frobenius reciprocity,
\begin{equation*}
    \Hom_{\cC_A}(\cF(\sigma^\ell(\cD_{r,1}^+)), M^k_{r',1}\otimes\Pi_{\ell'}(\lambda))\cong\Hom_\cC(\sigma^\ell(\cD_{r,1}^+), \cG(M^k_{r',1}\otimes\Pi_{\ell'}(\lambda))),
\end{equation*}
and the latter space is non-zero only for $r'=r$ (or $u-r$), $\ell'=\ell+1$, and $\lambda=\nu_{u-r,1}$ mod $\ZZ$.
Thus $\cF(\sigma^{\ell}(\cD_{r,1}^+)) \cong M_{r,1}^k\otimes \Pi_{\ell+1}(\nu_{u-r,1})$, and then
\[
\cF(\sigma^\ell(\cD_{r,1}^+))^* \cong M_{r,1}^k\otimes \Pi_{-\ell-1}(\nu_{u-r,1}).
\]
On the other hand, using \eqref{eqn:HW-as-SF-of-LW},
\begin{align*}
    \cF(\sigma^\ell(\cD_{r,1}^+)')  \cong \cF(\sigma^{-\ell}(\cD_{r,1}^-)) \cong \cF(\sigma^{-\ell-2}(\cD_{r,1}^+)) \cong M_{r,1}^k \otimes \Pi_{-\ell - 1}(\nu_{u-r,1})
\end{align*}
as well.
So $\cF(\sigma^\ell(\cD_{r,1}^+)')$ and $\cF(\sigma^\ell(\cD_{r,1}^+))^*$ are isomorphic simple objects of $\cC_A$, and thus any non-zero map between them is an isomorphism.

Step 3 is enough to verify condition (2) of Theorem \ref{rigidity-of-C-from-CAloc} in the case $\cC=\cC^{\mathrm{wt}}_k(\sl_2)$, $A=\Vir_{c_k}\otimes\Pi(0)$ for $k=-2+\frac{u}{2}$, $u$ odd, since the evaluation map $e_X: X'\boxtimes X\rightarrow V$ is non-zero for any simple object $X$ in $\cC$, and thus surjective since $V$ is simple.

\medskip

{\bf Step 4:} \textit{Verify the non-degeneracy condition (3) of Theorem \ref{rigidity-of-C-from-CAloc}.} This has already been done in Corollary \ref{cor:condition-3-for-sl2}.

\medskip

After completing Steps 1 through 4, all conditions of Theorem \ref{rigidity-of-C-from-CAloc} are verified, so we conclude:
\begin{theorem}
The tensor category $\cC^{\mathrm{wt}}_k(\sl_2)$ is rigid for $k=-2+\frac{u}{2}$, $u$ odd.
%  For  $V = L_k(\mathfrak{sl}_2)$, $A = \Vir_{c_k}\otimes \Pi(0)$ for $k = -2 + \frac{u}{2}$, the category $\cC$ is rigid. 
\end{theorem}

\subsection{The general case \texorpdfstring{$v \geq 3$}{v>2}}\label{subsec:v>2}

To prove that $\cC=\cC^{\mathrm{wt}}_k(\sl_2)$ is rigid for $k=-2+\frac{u}{v}$ with $v\geq 3$, we use the same four step strategy as in the previous subsection, where again Step 4 is already covered by Corollary \ref{cor:condition-3-for-sl2}. For Step 1, we show that $N:=\cF(A)$ has a composition series of local modules in the next lemma:

\begin{lemma}\label{lemma:Ncomposition}
As objects of $\cC_A$, $N\cong A \oplus M$, where $M$ fits into a non-split short exact sequence
\[
0 \longrightarrow  M^k_{1, 1} \otimes \Pi_2(-t) \longrightarrow M \longrightarrow  M^k_{1, 2} \otimes \Pi_1(-t/2) \longrightarrow 0.
\]
In particular, the three composition factors  $M^k_{1, 1} \otimes \Pi_0(0)$, $M^k_{1, 2} \otimes \Pi_1(-\frac{t}{2})$, $M^k_{1, 1} \otimes \Pi_2(-t)$ of $N$ are local.
\end{lemma}
\begin{proof}
Tensoring \eqref{eqn:A-exact-seq-sec-4} with $A$ yields the right exact sequence
\begin{equation*}
    A\xrightarrow{\cF(\iota_A)\circ r_A^{-1}} A\boxtimes A\longrightarrow A\boxtimes\sigma(\cD_{1,1}^+)\longrightarrow 0
\end{equation*}
in $\cC_A$, where $\cF(\iota_A)=\id_A\boxtimes\iota_A$ and $r_A$ is the right unit isomorphism in $\cC$. In fact, this sequence is exact and splits because $\mu_A\circ\cF(\iota_A)\circ r_A^{-1}=\id_A$ by the right unit property of commutative algebras in $\cC$, where $\mu_A: A\boxtimes A\rightarrow A$ is the algebra multiplication, which is a morphism in $\cC_A$. Thus it remains to show that $M:=\cF(\sigma(\cD_{1,1}^+))$ has the structure indicated in the lemma statement.

Because $\sigma(\cD_{1,1}^+) \cong \sigma^2(\cD^{-}_{u-1,v-2})$ embeds into $\sigma^2(\cE^{-}_{u-1,v-2})$, Frobenius reciprocity yields
\begin{equation}\nonumber
\begin{split}
\Hom_{\cC_A}(\cF(\sigma(\cD_{1,1}^+)), M^k_{1, 2} \otimes \Pi_1(-t/2)) &\cong \Hom_{\cC}(\sigma(\cD_{1,1}^+), \cG(M^k_{1, 2} \otimes \Pi_1(-t/2))) \\ 
&\cong \Hom_{\cC}(\sigma(\cD_{1,1}^+),\sigma^2(\cE_{u-1,v-2}^-)) \cong \CC.
\end{split}
\end{equation}
Thus $M^k_{1, 2} \otimes \Pi_1(-\frac{t}{2})$ is a simple quotient of $\cF(\sigma(\cD_{1,1}^+))$; we still need to show that this induced module also has $M^k_{1, 1} \otimes \Pi_2(-t)$ as a composition factor. Note that $\cG(M^k_{1, 1} \otimes \Pi_2(-t)) \cong \sigma^3(\cE_{\lambda_{1,3}, \Delta_{1,1}})$.

We tensor the exact sequence \eqref{eqn:A-exact-seq-sec-4} with $\sigma(\cD_{1,1}^+)$ and then apply Lemma \ref{lem:check-4b} and Corollary \ref{cor:fusionofd11} to obtain an exact sequence
\[
0\longrightarrow\sigma(\cD_{1,1}^+) \longrightarrow \cF(\sigma(\cD_{1,1}^+)) \longrightarrow \sigma^2(\cD_{1,2}^+)\oplus\sigma^3(\cE_{\lambda_{1,3}, \Delta_{1,1}}) \longrightarrow 0
\]
in $\cC$.
Since $\sigma(\cD_{1,1}^+)$ and $\sigma^2(\cD_{1,2}^+)$ are the composition factors of $\sigma^2(\cE_{u-1,v-2}^-)\cong\cG(M_{1,2}^k\otimes\Pi_1(-\frac{t}{2}))$, the kernel $K$ of the $\cC_A$-surjection $\cF(\sigma(\cD_{1,1}^+))\twoheadrightarrow M_{1,2}^k\otimes\Pi_1(-\frac{t}{2})$
satisfies
\begin{equation*}
    \cG(K)\cong \sigma^3(\cE_{\lambda_{1,3}, \Delta_{1,1}}) \cong \cG(M^k_{1, 1} \otimes \Pi_2(-t)).
\end{equation*}
Moreover, the exact sequence
\begin{equation*}
 0\longrightarrow K\longrightarrow\cF(\sigma(\cD_{1,1}^+))\longrightarrow M_{1,2}^k\otimes\Pi_1(-t/2)\longrightarrow 0
\end{equation*}
does not split in $\cC_A$ because by Frobenius reciprocity,
\begin{equation*}
    \Hom_{\cC_A}(\cF(\sigma(\cD_{1,1}^+)), K) \cong\Hom_\cC(\sigma(\cD_{1,1}^+),\cG(K))\cong\Hom_\cC(\sigma(\cD_{1,1}^+), \sigma^3(\cE_{\lambda_{1,3}, \Delta_{1,1}})) = 0.
\end{equation*} 
Finally, we show that $K\cong M^k_{1,1}\otimes\Pi_2(-t)$ as objects of $\cC_A$.
Indeed, since $\cG(K)$ is a simple object of $\cC$, its ribbon twist $\theta_{\cG(K)} =e^{2\pi i L_0}$ is a non-zero scalar and hence is a $\cC_A$-endomorphism of $K$. This implies $K$ is local (see for example \cite[Lemma 2.81]{CKM-ext}),
and Theorem \ref{thm:CAloc-rigid} gives a full classification of simple local $A$-modules. This classification implies in particular that
    two simple local $A$-modules that restrict to the same object of $\cC$ are isomorphic as $A$-modules. Thus $K\cong M^k_{1,1}\otimes\Pi_2(-t)$, as required.
\end{proof}

For Step 2, we prove the following theorem:
\begin{theorem}\label{thm:simple-local}
    All simple objects of $\cC_A$ are local.
\end{theorem}
\begin{proof}
    Let $X$ be a simple object in $\cC_A$, and let $Z$ be a simple submodule of $\cG(X)$. We prove that $X$ is local in the two cases that $Z$ is typical or atypical.

    \textbf{Case 1:} If $Z$ is typical, then there exists a simple object $Y$ in $\cC_A^{\loc}$ such that $\cG(Y) \cong Z$. We will prove by contradiction that $X\cong Y$, and thus $X$ is local.
    Since
        \[
        \Hom_{\cC}(Z, \cG(X)) \neq 0,\qquad \Hom_{\cC}(Z, \cG(Y)) \neq 0,
        \]
        by Frobenius reciprocity, there exist surjective $\cC_A$-morphisms
        \[
        \cF(Z) \longrightarrow X,\qquad \cF(Z) \longrightarrow Y.
        \]
Thus if $X \ncong Y$, then there is a surjective $\cC_A$-morphism
        \[
        \cF(Z) \longrightarrow X \oplus Y.
        \]
        Because the restriction functor $\cG$ is exact, there is also a $\cC$-surjection
        \[
        \cG(\cF(Z)) \longrightarrow \cG(X) \oplus \cG(Y),
        \]
        which implies that the multiplicity of $Z$ in $\cG(\cF(Z))$ is at least $2$.

        On the other hand, by Lemma \ref{lemma:CLR}, $\cG(\cF(Z)) \cong A\boxtimes\cG(Y)\cong\cG(N\boxtimes_A Y)$ where $N=\cF(A)$. Then 
        \begin{equation*}
            N\boxtimes_A Y = Y\oplus(M\boxtimes_A Y)
        \end{equation*}
        by Lemma \ref{lemma:Ncomposition}, where $M\boxtimes_A Y$ fits into the short exact sequence
        \begin{equation*}
            0\longrightarrow (M^k_{1,1}\otimes\Pi_2(-t))\boxtimes_A Y\longrightarrow M\boxtimes_A Y\longrightarrow (M^k_{1,2}\otimes\Pi_1(-t/2))\boxtimes_A Y\longrightarrow 0;
        \end{equation*}
        note that $Y$ is flat in $\cC_A$ because it is in the rigid subcategory $\cC_A^\loc$.   
Thus the multiplicity of $Z$ in 
        $\cG(\cF(Z))$ is the sum of the multiplicities of $Z$ in $\cG(Y)$, $\cG((M^k_{1,1}\otimes\Pi_2(-t)) \boxtimes_A Y)$, and $\cG((M^k_{1,2}\otimes\Pi_1(-\frac{t}{2}))\boxtimes_A Y)$. 
        
        We may assume $Z = \sigma^\ell(\cE_{\lambda, \Delta_{r,s}})$ where $\lambda \neq \pm \lambda_{r,s}$ mod $2\ZZ$, so that $Y = M_{r,s}^k\otimes \Pi_{\ell-1}(\frac{\lambda+k}{2})$. Then
        \begin{align*}
        \cG((M^k_{1,1}\otimes\Pi_2(-t)) \boxtimes_A Y) & \cong \sigma^{\ell+2}(\cE_{\lambda-2t, \Delta_{r,s}}),\\
        \cG((M^k_{1,2}\otimes\Pi_1(-t/2))\boxtimes_A Y) & \cong 
        \begin{cases}
        \sigma^{\ell+1}(\cE_{\lambda-t,\Delta_{r,2}}) & \text{if}\;\; s=1\\
        \sigma^{\ell+1}(\cE_{\lambda -t, \Delta_{r,s-1}}) \oplus \sigma^{\ell+1}(\cE_{\lambda -t, \Delta_{r,s+1}}) &\text{if}\;\; 2\leq s \leq v-2\\
        \sigma^{\ell+1}(\cE_{\lambda -t, \Delta_{r,v-2}}) & \text{if}\;\; s = v-1
        \end{cases} .
        \end{align*}
        None of these has $Z$ as a composition factor, so the multiplicity of $Z$ in $\cG(\cF(Z))$ is only $1$, contradicting the assumption $X\ncong Y$. Thus $X \cong Y$ as an $A$-module, and $X$ is local.

        \textbf{Case 2:} If $Z$ is atypical, then by Remark \ref{rem:onetone}, there are non-isomorphic simple modules $Y_1, Y_2$ in $\cC_A^{\loc}$ with non-split exact sequences
        \[
        0 \longrightarrow \text{soc}(\cG(Y_1)) \longrightarrow \cG(Y_1) \longrightarrow Z \longrightarrow 0,
        \]
        \[
        0 \longrightarrow Z \longrightarrow \cG(Y_2) \longrightarrow \text{top}(\cG(Y_2)) \longrightarrow 0.
        \]
We will prove that $X\cong Y_2$ by contradiction, thus showing $X$ is local. Now,
        \[
        \Hom_{\cC}(\cG(Y_1), \cG(Y_1)) \neq 0,\quad \Hom_{\cC}(\cG(Y_1), \cG(Y_2)) \neq 0,\quad \Hom_{\cC}(\cG(Y_1), \cG(X)) \neq 0,
        \]
       % \[
       % \Hom_{\cC}(\cG(Y_1), \cG(Y_2)) \neq 0,
       % \]
       % \[
       % \Hom_{\cC}(\cG(Y_1), \cG(X)) \neq 0,
       % \]
so by Frobenius reciprocity, there are surjective $\cC_A$-morphisms from $\cF(\cG(Y_1))$ to $Y_1$, $Y_2$, and $X$. Thus if $X \ncong Y_2$ (and $X\ncong Y_1$ automatically since $\Hom_\cC(Z,\cG(Y_1))=0$), then there is a surjective $\cC_A$-morphism 
        \[
        \cF(\cG(Y_1)) \longrightarrow X \oplus Y_1 \oplus Y_2,
        \]
        implying that $Z$ has multiplicity at least $3$ in $\cG(\cF(\cG(Y_1)))$. 
        
        On the other hand, $\cG(\cF(\cG(Y_1)))\cong\cG(N\boxtimes_A Y_1)$ using Lemma \ref{lemma:CLR} as before, and by Lemma \ref{lemma:Ncomposition},
        \begin{equation*}
            N\boxtimes_A Y_1 \cong Y_1\oplus (M\boxtimes_A Y_1),
        \end{equation*}
        where
        \[
        0 \longrightarrow (M^k_{1,1}\otimes\Pi_2(-t)) \boxtimes_A Y_1 \longrightarrow M\boxtimes_A Y_1 \longrightarrow (M^k_{1,2}\otimes\Pi_1(-t/2))\boxtimes_A Y_1 \longrightarrow 0
        \]
        is exact because $Y_1$ is rigid in $\cC_A$. Thus we can compute the multiplicity of $Z$ in $\cG(\cF(\cG(Y_1)))$ explicitly.

   We may assume that $Z = \sigma^\ell(\cD_{r,s}^+)$ for some $\ell\in\ZZ$ and $1\leq r\leq u-1$, $1\leq s\leq v-1$. Because $ \cG(M_{r,s}^k \otimes \Pi_{\ell-1}(\nu_{r,s}))\cong\sigma^\ell(\cE_{u-r,v-s}^-)$ surjects onto $\sigma^\ell(\cD_{r,s}^+)$, 
   we take $Y_1 = M_{r,s}^k \otimes \Pi_{\ell-1}(\nu_{r,s})$.
   Then
        \begin{align*}
        \cG((M^k_{1,1}\otimes\Pi_2(-t))\boxtimes_A Y_1) & \cong \begin{cases}
            \sigma^{\ell+2}(\cE_{\lambda_{r,s+2}, \Delta_{r,s}}) & \text{if}\;\;1\leq s\leq v-2\\
            \sigma^{\ell+2}(\cE^-_{r,v-1}) & \text{if}\;\;s=v-1
        \end{cases},\nonumber\\
        \cG((M^k_{1,2}\otimes\Pi_1(-t/2))\boxtimes_A Y_1) & \cong 
        \begin{cases}
        \sigma^{\ell+1}(\cE^-_{u-r,v-2}) & \text{if}\;\; s=1\\
           \sigma^{\ell+1}(\cE_{\lambda_{r,s+1},\Delta_{r,s-1}})\oplus \sigma^{\ell+1}(\cE_{u-r,v-s-1}^-)  & \text{if}\;\; 2\leq s \leq v-2\\
          \sigma^{\ell+1}(\cE_{\lambda_{r,s+1},\Delta_{r,v-2}}) & \text{if}\;\; s = v-1
        \end{cases} .
        \end{align*}
If $1\leq s\leq v-2$, then $Z$ is the socle of $\sigma^{\ell+1}(\cE_{u-r,v-s-1}^-)$, and if $s=v-1$, then $Z$ is the socle of $\sigma^{\ell+2}(\cE_{r,v-1}^-)$.
        Since $Z$ is also the top of $\cG(Y_1)$,
        it follows that $Z$ has multiplicity $2$ in $\cG(\cF(\cG(Y_1)))$, contradicting the assumption that $X\ncong Y_2$. Thus $X \cong Y_2$ and therefore $X$ is local.
\end{proof}

Now that Steps 1 and 2 are complete, condition (1) in Theorem \ref{rigidity-of-C-from-CAloc} holds for $\cC=\cC^{\mathrm{wt}}_k(\sl_2)$ and $A=\Vir_{c_k}\otimes\Pi(0)$, and therefore $\cC_A$ is rigid by \cite[Theorem 3.14]{CMSY}. It remains to check condition (2) in Theorem \ref{rigidity-of-C-from-CAloc}, and since evaluations $e_X: X'\boxtimes X\rightarrow V$ in $\cC$ are always surjective for $X\neq 0$ (since they are non-zero and $V=L_k(\sl_2)$ is simple) we just need to complete Step 3 and show that for any simple object $X$ in $\cC$, any non-zero morphism from $\cF(X')$ to $\cF(X)^*$ in $\cC_A$ has to be an isomorphism. To prove this, we need more details about $\cC_A$ as an abelian category. We start with a consequence of Lemma \ref{lemma:CLR}, Lemma \ref{lemma:Ncomposition}, and the fusion rules \eqref{eqn:A-fusion-rules} in $\cC_A^\loc$:
\begin{cor}\label{cor:fusionwithA}
For $1\leq r \leq u-1$, $1\leq s \leq v-1$, $\ell \in \ZZ$ and $\lambda \in \CC$, 
\[
A \boxtimes \cG(M^k_{r, s} \otimes \Pi_\ell(\lambda)) \cong \cG(M^k_{r, s} \otimes \Pi_\ell(\lambda)) \oplus M_{r, s, \lambda, \ell}
\]
where the object $M_{r, s, \lambda,\ell}$ in $\cC$ fits into a short exact sequence
    \[
    0 \rightarrow \cG(M^k_{r, s} \otimes \Pi_{\ell+2}(\lambda - t)) \rightarrow M_{r, s, \lambda,\ell} \rightarrow 
     \cG(M^k_{r, s - 1} \otimes \Pi_{\ell+1}(\lambda -t/2)) \oplus \cG(M^k_{r, s + 1} \otimes \Pi_{\ell+1}(\lambda-t/2)) \rightarrow 0,
    \]
    with the convention that $M^k_{r, 0} = M^k_{r, v} = 0$.
\end{cor}

We now construct and describe a projective cover for each simple object in $\cC_A$:

\begin{theorem}\label{thm:R-structure}
For any $1\leq r\leq u-1$, $1\leq s\leq v-1$, $\lambda+\ZZ \in \CC/\ZZ$, and $\ell \in \ZZ$,
    there is a projective and injective module $R_{r,s, \lambda, \ell}$ in $\cC_A$ with the following Loewy diagram, 
    
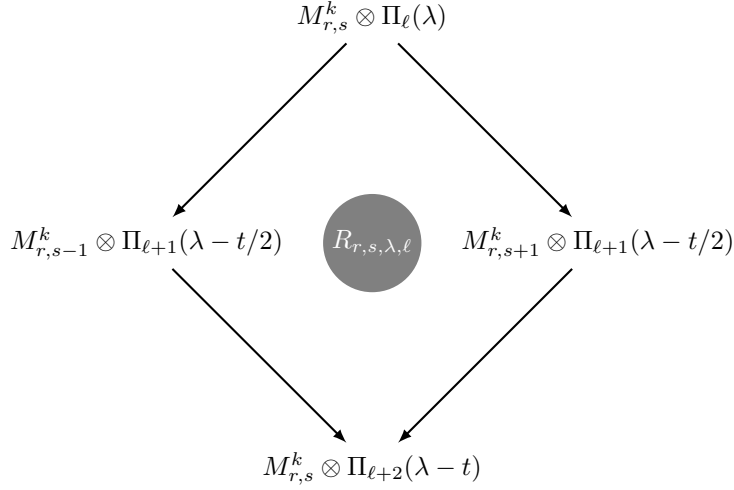
\begin{center}
\begin{tikzpicture}[scale=1]
\node (top) at (0,3) [] {$ M^k_{r, s} \otimes \Pi_\ell(\lambda)$};
\node (left) at (-3,0) [] {$M^k_{r, s - 1} \otimes \Pi_{\ell+1}(\lambda -t/2)$};
\node (right) at (3,0) [] {$M^k_{r, s + 1} \otimes \Pi_{\ell+1}(\lambda -t/2)$};
\node (bottom) at (0,-3) [] {$M^k_{r, s} \otimes \Pi_{\ell+2}(\lambda - t)$};
\draw[->, thick] (top) -- (left);
\draw[->, thick] (top) -- (right);
\draw[->, thick] (left) -- (bottom);
\draw[->, thick] (right) -- (bottom);
\node (label) at (0,0) [circle, inner sep=2pt, color=white, fill=black!50!] {$R_{r, s, \lambda, \ell}$};
\end{tikzpicture}
\end{center}
with the convention that $M^k_{r, 0} = M^k_{r, v} = 0$. 
\end{theorem}

\begin{proof}
We first consider the case $\lambda\notin\QQ$ and set $R_{r, s, \lambda, \ell} = \cF(\cG(M^k_{r, s} \otimes \Pi_\ell(\lambda)))$. Then $\cG(M^k_{r, s} \otimes \Pi_\ell(\lambda))\cong\sigma^{\ell+1}(\cE_{2\lambda-k,\Delta_{r,s}})$ is projective in $\cC$, and since the induction of a projective object is projective by Frobenius reciprocity, $R_{r, s, \lambda, \ell}$ is projective in $\cC_A$. Since
$\cC_A$ is rigid, $R_{r,s,\lambda,\ell}$ is also injective in $\cC_A$ by \cite[Proposition 6.1.3]{EGNO}.
Moreover, since $\lambda\notin\QQ$, all three or four composition factors of $A\boxtimes\cG(M^k_{r,s}\otimes\Pi_\ell(\lambda))$ in Corollary \ref{cor:fusionwithA} are projective in $\cC$, and thus
  \begin{equation}\label{eq:fusiongenericlambda}
    \begin{split}
    \cG(R_{r,s,\lambda,\ell}) & \cong 
    \cG(M^k_{r, s} \otimes \Pi_\ell(\lambda)) \oplus \cG(M^k_{r, s} \otimes \Pi_{\ell+2}(\lambda - t))\\ 
    &\qquad\oplus \cG(M^k_{r, s - 1} \otimes \Pi_{\ell+1}(\lambda -t/2)) \oplus \cG(M^k_{r, s + 1} \otimes \Pi_{\ell+1}(\lambda-t/2)).
     \end{split}
    \end{equation} 
    We use Frobenius reciprocity to determine the top of $R_{r, s, \lambda, \ell}$: since $\cG(M^k_{r, s} \otimes \Pi_\ell(\lambda))$ is simple in $\cC$,
    \[
    \Hom_{\cC_A}(R_{r, s, \lambda, \ell}, M^k_{r', s'} \otimes \Pi_{\ell'}(\lambda')) \cong \Hom_{\cC}(\cG(M^k_{r, s} \otimes \Pi_\ell(\lambda)), \cG(M^k_{r', s'} \otimes \Pi_{\ell'}(\lambda')))
    \]
is non-zero and one-dimensional if and only if $(r, s, \ell, \lambda) = (r',s', \ell', \lambda')$.
Thus $\text{top}(R_{r, s, \lambda, \ell}) \cong M^k_{r, s} \otimes \Pi_\ell(\lambda)$.

For the remaining composition factors of $R_{r,s,\lambda,\ell}$, Frobenius reciprocity and \eqref{eq:fusiongenericlambda} imply that 
\begin{equation*}
    \Hom_{\cC_A}(R_{r', s', \lambda', \ell'}, R_{r, s, \lambda, \ell}) \cong
        \Hom_{\cC}(\cG(M^k_{r', s'} \otimes \Pi_{\ell'}(\lambda')), \cG(R_{r,s\lambda,\ell}))
\end{equation*}
is non-zero and one-dimensional for $\lambda'\notin\QQ$ precisely in the cases that $(r',s',\lambda',\ell')$ is equivalent to
\begin{equation*}
    (r,s,\lambda,\ell),\quad (r,s,\lambda-t,\ell+2),\quad (r,s-1,\lambda-t/2,\ell+1),\quad (r,s+1,\lambda-t/2,\ell+1),
\end{equation*}
where the third case does not occur if $s=1$, and the fourth does not occur if $s=v-1$. Thus $R_{r,s,\lambda,\ell}$ contains the simple quotient $M^k_{r',s'}\otimes\Pi_{\ell'}(\lambda')$ of $R_{r',s',\lambda',\ell'}$ as a subquotient in these three or four cases, and by \eqref{eq:fusiongenericlambda}, this exhausts the composition factors of $R_{r,s,\lambda,\ell}$.

Since $R_{r, s, \lambda-t, \ell+2}$ and $R_{r, s, \lambda, \ell}$ only have $M^k_{r, s} \otimes \Pi_{\ell+2}(\lambda-t)$ as a common composition factor, 
the image of the non-zero $\cC_A$-morphism $R_{r,s,\lambda-t,\ell+2}\rightarrow R_{r,s,\lambda,\ell}$ is a submodule of
$R_{r, s, \lambda, \ell}$ isomorphic to $M^k_{r, s} \otimes \Pi_{\ell+2}(\lambda-t)$. Let $Q_{r, s, \lambda, \ell} = R_{r, s, \lambda, \ell}/M^k_{r, s} \otimes \Pi_{\ell+2}(\lambda-t)$. Since $M^k_{r, s} \otimes \Pi_{\ell+2}(\lambda-t)$ is not in $\mathrm{top}(R_{r, s, \lambda, \ell})$, the tops of $Q_{r, s, \lambda, \ell}$ and $R_{r, s, \lambda, \ell}$ coincide. This is consistent with three possible Loewy diagrams for $Q_{r, s, \lambda, \ell}$:
    \begin{center}
\begin{tikzpicture}[scale=1]
\node (top) at (0,2.5) [] {$ M^k_{r, s} \otimes \Pi_\ell(\lambda)$};
\node (left) at (-2,0.5) [] {$M^k_{r, s - 1} \otimes \Pi_{\ell+1}(\lambda -\frac{t}{2})$};
\node (right) at (2,0.5) [] {$M^k_{r, s + 1} \otimes \Pi_{\ell+1}(\lambda -\frac{t}{2})$};
\draw[->, thick] (top) -- (left);
\draw[->, thick] (top) -- (right);
\node (top2) at (6,3) [] {$M^k_{r, s} \otimes \Pi_\ell(\lambda)$};
\node (middle2) at (6,1.5) [] {$M^k_{r, s - 1} \otimes \Pi_{\ell+1}(\lambda -\frac{t}{2})$};
\node (bottom2) at (6, 0) [] {$M^k_{r, s + 1} \otimes \Pi_{\ell+1}(\lambda -\frac{t}{2})$};
\draw[->, thick] (top2) -- (middle2);
\draw[->, thick] (middle2) -- (bottom2);
\node (top3) at (10,3) [] {$ M^k_{r, s} \otimes \Pi_\ell(\lambda)$};
\node (middle3) at (10,1.5) [] {$M^k_{r, s + 1} \otimes \Pi_{\ell+1}(\lambda -\frac{t}{2})$};
\node (bottom3) at (10, 0) [] {$M^k_{r, s - 1} \otimes \Pi_{\ell+1}(\lambda -\frac{t}{2})$};
\draw[->, thick] (top3) -- (middle3);
\draw[->, thick] (middle3) -- (bottom3);
\end{tikzpicture} 
\end{center}
But since Frobenius reciprocity and \eqref{eq:fusiongenericlambda} again imply that 
\[
\Hom_{\cC_A}(R_{r,s\pm 1,\lambda-t/2,\ell+1}, Q_{r,s,\lambda,\ell})\cong\Hom_\cC(\cG(M^k_{r,s\pm1}\otimes\Pi_{\ell+1}(\lambda-t/2)), \cG(Q_{r,s,\lambda,\ell}))\neq 0
\]
for both sign choices (or for one sign choice when $s=1$ or $v-1$), and since $R_{r,s\pm 1,\lambda-t/2,\ell+1}$ and $Q_{r,s,\lambda,\ell}$ only have $M^k_{r,s\pm 1}\otimes\Pi_{\ell+1}(\lambda-\frac{t}{2})$ as a common factor, $Q_{r,s,\lambda,\ell}$ must contain $M^k_{r,s\pm 1}\otimes\Pi_{\ell+1}(\lambda-\frac{t}{2})$ as a submodule for both sign choices. Hence, the diagram on the left is the Loewy diagram of $Q_{r,s,\lambda,\ell}$.

To complete the determination of the Loewy diagram of $R_{r,s,\lambda,\ell}$, we just need to show that the socle of $R_{r,s,\lambda,\ell}$ contains only $M^k_{r,s}\otimes\Pi_{\ell+2}(\lambda-t)$, equivalently the top of $R_{r,s,\lambda,\ell}^*$ contains only $M^k_{r,s}\otimes\Pi_{\ell+2}(\lambda-t)^*\cong M^k_{r,s}\otimes\Pi_{-\ell-2}(t-\lambda)$. It is enough to prove that
\begin{equation}\label{eqn:R-dual}
    R_{r,s,\lambda,\ell}^*\cong R_{r,s,t-\lambda,-\ell-2},
\end{equation}
since $t-\lambda\notin\QQ$ and thus we have already shown that $\mathrm{top}(R_{r,s,t-\lambda,-\ell-2})\cong M^k_{r,s}\otimes\Pi_{-\ell-2}(t-\lambda)$. Indeed, since $R_{r,s,t-\lambda,-\ell-2}$ and $R_{r,s,\lambda,\ell}^*$ are both projective (the latter by \cite[Proposition 6.1.3]{EGNO}), there are $\cC_A$-morphisms $f: R_{r,s,t-\lambda,-\ell-2}\rightarrow R_{r,s,\lambda,t}^*$ and $g: R_{r,s,\lambda,t}^*\rightarrow R_{r,s,t-\lambda,-\ell-2}$ such that the diagrams
\begin{equation*}
\begin{tikzcd}[column sep=2pc]
& R_{r,s,t-\lambda,-\ell-2} \arrow[d, two heads] \arrow[ld, "f"'] \\
R_{r,s,\lambda,\ell}^* \arrow[r, two heads] & M^k_{r,s}\otimes\Pi_{-\ell-2}(t-\lambda) 
    \end{tikzcd},\qquad
\begin{tikzcd}[column sep=2pc]
& R_{r,s,\lambda,\ell}^* \arrow[d, two heads] \arrow[ld, "g"'] \\
R_{r,s,t-\lambda,-\ell-2} \arrow[r, two heads] & M^k_{r,s}\otimes\Pi_{-\ell-2}(t-\lambda) 
    \end{tikzcd}
\end{equation*}
commute. Moreover, $R_{r,s,\lambda,\ell}$ and $R_{r,s,t-\lambda,-\ell-2}$ are indecomposable since they have simple tops, and thus $R_{r,s,\lambda,t}^*$ is indecomposable as well. Thus $f\circ g$ and $g\circ f$ are either isomorphisms or nilpotent by Fitting's Lemma. The commutative diagrams imply they cannot be nilpotent, so they are isomorphisms. In particular, $f$ is both injective and surjective and thus an isomorphism. This completes the proof of the theorem in the case $\lambda\notin\QQ$.

For $\lambda \in \QQ$, we define $R_{r,s,\lambda,\ell} =(M^k_{1,1}\otimes\Pi_0(\mu))\boxtimes_A R_{r,s,\lambda',\ell}$ for $\mu,\lambda'\notin\QQ$ such that $\mu+\lambda' =\lambda$ mod $\ZZ$. Since $M^k_{1, 1} \otimes \Pi_{0}(\mu)$ is invertible in $\cC_A$, \cite[Proposition 2.5(7)]{CKLR} says that the Loewy diagram of $R_{r,s,\lambda,\ell}$ is obtained by replacing each composition factor $X$ in the Loewy diagram of $R_{r,s,\lambda',\ell}$ with $(M^k_{1,1}\otimes\Pi_0(\mu))\boxtimes_A X$. Thus $R_{r,s,\lambda,\ell}$ has the Loewy diagram indicated in the theorem statement, and $R_{r,s,\lambda,\ell}$ is projective in $\cC_A$ because projective objects in a rigid tensor category form a tensor ideal. Then $R_{r,s,\lambda,\ell}$ is also injective in $\cC_A$ by \cite[Proposition 6.1.3]{EGNO}. Note that $R_{r,s,\lambda,\ell}$ does not depend up to isomorphism on the choice of $\mu$ and $\lambda'$ since any choice yields a projective cover of $M^k_{r,s}\otimes\Pi_{\ell}(\lambda)$ in $\cC_A$, and projective covers are unique up to isomorphism.
 \end{proof}

 Note that the proof of \eqref{eqn:R-dual} in the case $\lambda\notin\QQ$ now works for all $\lambda\in\CC$ since by Theorem \ref{thm:R-structure}, $R_{r,s,\lambda,\ell}^*$ and $R_{r,s,t-\lambda,-\ell-2}$ are indecomposable projective objects of $\cC_A$ with isomorphic simple tops for all $\lambda\in\CC$.
 
\begin{theorem}\label{thm:indsimple1}
If $\lambda\neq\nu_{r,s},\nu_{u-r,v-s}\;\mathrm{mod}\;\ZZ$, then $\cF(\cG(M^k_{r, s} \otimes \Pi_{\ell}(\lambda))) \cong R_{r, s, \lambda, \ell}$ and any non-zero $\cC_A$-morphism from $\cF(\cG(M^k_{r, s} \otimes \Pi_{\ell}(\lambda))')$ to $\cF(\cG(M^k_{r, s} \otimes \Pi_{\ell}(\lambda)))^*$ is an isomorphism.
\end{theorem}
\begin{proof}
The first statement holds if $\lambda\notin\QQ$ by the definition of $R_{r,s,\lambda,\ell}$, but we still need to consider the case that $\lambda\in\QQ$ but not equal to $\nu_{r,s},\nu_{u-r,v-s}$ mod $\ZZ$. As in the proof of Theorem \ref{thm:R-structure}, because $\cG(M^k_{r, s} \otimes \Pi_{\ell}(\lambda))$ is simple in $\cC$, Frobenius reciprocity implies that $\mathrm{top}(\cF(\cG(M^k_{r, s} \otimes \Pi_{\ell}(\lambda)))) \cong M^k_{r, s} \otimes \Pi_{\ell}(\lambda)$. Moreover, Corollary \ref{cor:fusionwithA} forces the multiplicity of $M^k_{r, s} \otimes \Pi_{\ell}(\lambda)$ in $\cF(\cG(M^k_{r, s} \otimes \Pi_{\ell}(\lambda)))$ to be $1$.

Since $R_{r,s,\lambda,\ell}$ is projective in $\cC_A$ and also has top isomorphic to $M^k_{r, s} \otimes \Pi_{\ell}(\lambda)$, there is a non-zero $\cC_A$-morphism $f: R_{r,s,\lambda,\ell}\rightarrow \cF(\cG(M^k_{r, s} \otimes \Pi_{\ell}(\lambda)))$ whose image contains $M^k_{r, s} \otimes \Pi_{\ell}(\lambda)$ as a simple quotient. Thus $f$ must be surjective since otherwise $\cF(\cG(M^k_{r, s} \otimes \Pi_{\ell}(\lambda)))$ would have a simple quotient not isomorphic to $M^k_{r, s} \otimes \Pi_{\ell}(\lambda)$. Moreover, Corollary \ref{cor:fusionwithA} implies that $A\boxtimes\cG(M^k_{r, s} \otimes \Pi_{\ell}(\lambda))$ and $\cG(R_{r,s,\lambda,\ell})$ have the same length in $\cC$, so $f$ is also injective and thus an isomorphism. This proves the first statement.

Now the first statement, \eqref{eq:contragedientdual typical}, and \eqref{eqn:R-dual} imply that 
\begin{align*}
    \cF(\cG(M^k_{r, s} \otimes \Pi_{\ell}(\lambda))') & \cong\cF(\cG(M^k_{r,s}\otimes\Pi_{-\ell-2}(t-\lambda))) \cong R_{r,s,t-\lambda,-\ell-2},\\
    \cF(\cG(M^k_{r, s} \otimes \Pi_{\ell}(\lambda)))^* & \cong R_{r,s\lambda,\ell}^*\cong R_{r,s,t-\lambda,-\ell-2}.
\end{align*}
Then the second statement follows because
\begin{equation*}
    \End_{\cC_A}(R_{r,s,t-\lambda,-\ell-2})\cong\Hom_\cC(\cG(M^k_{r,s}\otimes\Pi_{-\ell-2}(t-\lambda)),\cG(R_{r,s,t-\lambda,-\ell-2}))\cong\CC
\end{equation*}
by Frobenius reciprocity and Corollary \ref{cor:fusionwithA}.
\end{proof}

We turn to the atypical case. Since the module $R_{r,s,\nu_{r,s},\ell}$ for $1\leq r\leq u-1$, $1\leq s\leq v-1$, and $\ell\in\ZZ$ is projective and has $M^k_{r,s}\otimes\Pi_\ell(\nu_{r,s})$ as its unique simple quotient, $\Hom_{\cC_A}(R_{r,s,\nu_{r,s},\ell}, R_{r',s',\lambda',\ell'})$ for any $(r',s',\lambda',\ell')$ has the same dimension as the multiplicity of $M^k_{r,s}\otimes\Pi_\ell(\nu_{r,s})$ in $R_{r',s',\lambda',\ell'}$. In particular, there is a unique (up to scaling) $\cC_A$-morphism
\begin{equation*}
    f_{r,s,\ell}: R_{r,s,\nu_{r,s},\ell}\longrightarrow R_{r,s-1,\nu_{r,s}+t/2,\ell-1} =  R_{r,s-1,\nu_{r,s-1},\ell-1}
\end{equation*}
for $2\leq s\leq v-1$. For such values of $s$ we define $M_{r,s,\ell}\subseteq R_{r,s-1,\nu_{r,s-1},\ell-1}$ to be the image of $f_{r,s,\ell}$; inspection of the Loewy diagrams in Theorem \ref{thm:R-structure} shows that $M_{r,s,\ell}$ has Loewy diagram
\begin{center}
\begin{tikzpicture}[scale=1]
\node (top1) at (0,2) [] {$M^k_{r, s} \otimes \Pi_\ell(\nu_{r, s})$};
\node (bottom1) at (0,0) [] {$M^k_{r, s - 1} \otimes \Pi_{\ell+1}(\nu_{r, s+1})$};
\draw[->, thick] (top1) -- (bottom1);
\node (label) at (-3,1) [circle, inner sep=2pt, color=white, fill=black!50!] {$M_{r, s, \ell}$};
\end{tikzpicture}
\end{center}
  for $2\leq s\leq v-1$, and there is a non-split short exact sequence
    \begin{equation}\label{eqn:Mrsl-exact-seq}
    0 \longrightarrow M_{r, s+1, \ell+1} \longrightarrow R_{r, s, \nu_{r, s}, \ell} \longrightarrow M_{r, s, \ell} \longrightarrow 0
    \end{equation}
for $2\leq s\leq v-2$. We then define $M_{r,1,\ell}=\mathrm{Coker}\,f_{r,2,\ell+1}\cong M^k_{r,1}\otimes\Pi_\ell(\nu_{r,1})$ and $M_{r,v,\ell}=\mathrm{Ker}\,f_{r,v-1,\ell-1}\cong M^k_{r,v-1}\otimes\Pi_{\ell+1}(\nu_{r,v+1})$ for $1\leq r\leq u-1$ and $\ell\in\ZZ$, so that \eqref{eqn:Mrsl-exact-seq} holds for all $1\leq s\leq v-1$. Since $M^k_{r,v-1}=M^k_{u-r,1}$ and $\nu_{r,v+1}=\nu_{u-r,1}$ mod $\ZZ$, we can also write $M_{r,v,\ell}\cong M^k_{u-r,1}\otimes\Pi_{\ell+1}(\nu_{u-r,1})$.

\begin{theorem}\label{thm:ind-simple-atypical}
    For $1\leq r\leq u-1$, $1\leq s\leq v-1$, and $\ell\in\ZZ$, $\cF(\sigma^\ell(\cD_{r,s}^+))\cong M_{r,s+1,\ell}$, and any non-zero $\cC_A$-morphism from $\cF(\sigma^\ell(\cD_{r,s}^+)')$ to $\cF(\sigma^\ell(\cD_{r,s}^+))^*$ is an isomorphism.
\end{theorem}
\begin{proof}
    On the one hand, because the induction functor $\cF$ is right exact, the exact sequence \eqref{eqn:exact-seq-E+E-} in $\cC$ induces a right exact sequence
    \begin{equation*}
        \cF(\sigma^\ell(\cD_{u-r,v-s}^-))\longrightarrow\cF(\cG(M^k_{r,s}\otimes\Pi_{\ell-1}(\nu_{r,s})))\longrightarrow\cF(\sigma^\ell(\cD_{r,s}^+))\longrightarrow 0
    \end{equation*}
    in $\cC_A$. It follows from Corollary \ref{cor:fusionwithA} and the bijection $\tau:\mathrm{Irr}(\cC)\rightarrow\mathrm{Irr}(\cC_A)$ given in Remark \ref{rem:onetone} and Theorem \ref{thm:simple-local} that the composition factors of $\cF(\cG(M^k_{r,s}\otimes\Pi_{\ell-1}(\nu_{r,s})))$ are precisely
    \begin{equation}\label{eqn:ind-of-rsl-1}
M^k_{r,s}\otimes\Pi_{\ell-1}(\nu_{r,s}),\quad M^k_{r,s-1}\otimes\Pi_{\ell}(\nu_{r,s+1}),\quad M^k_{r,s+1}\otimes\Pi_\ell(\nu_{r,s+1}),\quad M^k_{r,s}\otimes\Pi_{\ell+1}(\nu_{r,s+2}).
    \end{equation}
    Thus the composition factors of $\cF(\sigma^\ell(\cD_{r,s}^+))$ form a subset of these three or four. Similarly, since
    \begin{equation*}
        \sigma^\ell(\cD_{u-r,v-s}^-)\cong\begin{cases}
            \sigma^{\ell-2}(\cD_{u-r,v-1}^+) & \text{if}\;s=1\\
            \sigma^{\ell-1}(\cD_{r,s-1}^+) & \text{if}\;2\leq s\leq v-1\\
        \end{cases},
    \end{equation*}
    the composition factors of $\cF(\sigma^\ell(\cD_{u-r,v-s}^-))$ form a subset of
    \begin{equation}\label{eqn:comp-factor-s=1}
        M^k_{u-r,v-1}\otimes\Pi_{\ell-3}(\nu_{u-r,v-1}),\quad M^k_{u-r,v-2}\otimes\Pi_{\ell-2}(\nu_{u-r,v}),\quad M^k_{u-r,v-1}\otimes\Pi_{\ell-1}(\nu_{u-r,v+1})
    \end{equation}
    if $s=1$ and a subset of
    \begin{equation}\label{eqn:comp-factor-s>1}
        M^k_{r,s-1}\otimes\Pi_{\ell-2}(\nu_{r,s-1}),\quad M^k_{r,s-2}\otimes\Pi_{\ell-1}(\nu_{r,s}),\quad M^k_{r,s}\otimes\Pi_{\ell-1}(\nu_{r,s}),\quad M^k_{r,s-1}\otimes\Pi_{\ell}(\nu_{r,s+1})
    \end{equation}
    if $2\leq s\leq v-1$. Comparing \eqref{eqn:ind-of-rsl-1} with \eqref{eqn:comp-factor-s=1} and \eqref{eqn:comp-factor-s>1}, and noting $M^k_{u-r,v-1}\otimes\Pi_{\ell-1}(\nu_{u-r,v+1})= M^k_{r,1}\otimes\Pi_{\ell-1}(\nu_{r,1})$ for the $s=1$ case, we see that $\cF(\sigma^\ell(\cD_{r,s}^+))$ must contain at least $M^k_{r,s+1}\otimes\Pi_\ell(\nu_{r,s+1})$ and $M^k_{r,s}\otimes\Pi_{\ell+1}(\nu_{r,s+2})$ as composition factors, both with multiplicity $1$. Note that these are precisely the composition factors of $M_{r,s+1,\ell}$, including the case $s=v-1$ (where the first of these composition factor vanishes and $M_{r,v,\ell}$ is simple).

    On the other hand, there is a non-zero (and unique up to scaling) $\cC$-morphism $\sigma^\ell(\cD_{r,s})\rightarrow\cG(M^k_{r',s'}\otimes\Pi_{\ell'}(\lambda')$ if and only if $(r',s',\lambda',\ell')$ is equivalent to $(r,s+1,\nu_{r,s+1},\ell)$ for $1\leq s\leq v-2$, or $(u-r,1,\nu_{u-r,1},\ell+1)$ for $s=v-1$. Thus by Frobenius reciprocity, $\mathrm{top}(\cF(\sigma^\ell(\cD_{r,s}^+)))\cong\mathrm{top}(M_{r,s+1,\ell})$. Let $P$ be the projective cover of $\mathrm{top}(M_{r,s+1,\ell})$ in $\cC_A$, that is, $P=R_{r,s+1,\nu_{r,s+1},\ell}$ if $1\leq s\leq v-2$ and $P=R_{u-r,1,\nu_{u-r,1},\ell+1}$ if $s=v-1$. Since $P$ is projective in $\cC_A$, there is a morphism $f:P\rightarrow\cF(\sigma^\ell(\cD_{r,s}^+))$ such that the diagram
    \begin{equation*}
\begin{tikzcd}[column sep=2pc]
& P \arrow[d, two heads] \arrow[ld, "f"'] \\
\cF(\sigma^\ell(\cD_{r,s}^+)) \arrow[r, two heads] & \mathrm{top}(M_{r,s+1,\ell})
    \end{tikzcd}
\end{equation*}
commutes. Moreover, $f$ is surjective as $\mathrm{Im}\,f$ contains the unique composition factor in $\mathrm{top}(\cF(\sigma^\ell(\cD_{r,s}^+)))$ (since all composition factors of $\cF(\sigma^\ell(\cD_{r,s}^+))$ have multiplicity $1$ by \eqref{eqn:ind-of-rsl-1}). Comparing the composition factors of $P$ from Theorem \ref{thm:R-structure} with the possible composition factors of $\cF(\sigma^\ell(\cD_{r,s}^+))$ in \eqref{eqn:ind-of-rsl-1}, we see that $\mathrm{Ker}\,f$ contains $M_{r,s+2,\ell+1}$ when $1\leq s\leq v-2$ and $M_{u-r,2,\ell+2}$ when $s=v-1$. Thus since $f$ is surjective and $\cF(\sigma^\ell(\cD_{r,s}^+))$ contains at least the composition factors of $M_{r,s+1,\ell}$, we conclude that
\begin{align*}
    \cF(\sigma^\ell(\cD_{r,s}^+)) & \cong R_{r,s+1,\nu_{r,s+1},\ell}/M_{r,s+2,\ell+1} \cong M_{r,s+1,\ell}
    \end{align*}
    if $1\leq s\leq v-2$, and
    \begin{align*}
    \cF(\sigma^\ell(\cD_{r,v-1}^+)) & \cong R_{u-r,1,\nu_{u-r,1},\ell+1}/M_{u-r,2,\ell+2}\cong M^k_{u-r,1}\otimes\Pi_{\ell+1}(\nu_{u-r,1})\cong M_{r,u,\ell},
\end{align*}
proving the first statement of the theorem.

For the second statement, we first calculate
\begin{align*}
    \cF(\sigma^{\ell}(\cD_{r,s}^+)') &\cong\cF(\sigma^{-\ell}(\cD_{r,s}^-) \cong\begin{cases}
        \cF(\sigma^{-\ell-1}(\cD_{u-r,v-s-1}^+)) & \text{if}\;1\leq s\leq v-2\\
        \cF(\sigma^{-\ell-2}(\cD_{r,v-1}^+)) & \text{if}\;s=v-1\\
    \end{cases}\nonumber\\
    &\cong \begin{cases}
        M_{u-r,v-s,-\ell-1} & \text{if}\;1\leq s\leq v-2\\
        M_{r,v,-\ell-2} & \text{if}\;s=v-1\\
    \end{cases} .
\end{align*}
Thus for $1\leq s\leq v-2$, using $M^k_{u-r,v-s}=M^k_{r,s}$ and $\nu_{u-r,v-s} =-\nu_{r,s+2}$ mod $\ZZ$, there is a non-split exact sequence
\begin{equation*}
    0\longrightarrow M^k_{r,s+1}\otimes\Pi_{-\ell}(-\nu_{r,s+1})\longrightarrow\cF(\sigma^\ell(\cD_{r,s}^+)')\longrightarrow M^k_{r,s}\otimes\Pi_{-\ell-1}(-\nu_{r,s+2})\longrightarrow 0,
\end{equation*}
where the submodule and quotient are non-isomorphic. On the other hand, by reversing the Loewy diagram of $M_{r,s+1,\ell}$ and taking the dual of each composition factor, we see that $\cF(\sigma^\ell(\cD_{r,s}^+))^*\cong M_{r,s+1,\ell}^*$ fits into a non-split short exact sequence with the same submodule and quotient. Thus any non-zero morphism from $\cF(\sigma^\ell(\cD_{r,s}^+)')$ to $\cF(\sigma^\ell(\cD_{r,s}^+))^*$ must be an isomorphism, since otherwise its image would have to be isomorphic to $M^k_{r,s}\otimes\Pi_{-\ell-1}(-\nu_{r,s+2})$, which is not a submodule of $\cF(\sigma^\ell(\cD_{r,s}^+))^*$. This proves the second statement of the theorem for $1\leq s\leq v-2$, and for $s=v-1$, we have
\begin{align*}
   \cF(\sigma^{\ell}(\cD_{r,v-1}^+)') &\cong  M_{r,v,-\ell-2}\cong M^k_{u-r,1}\otimes\Pi_{-\ell-1}(\nu_{u-r,1}) = M^k_{u-r,1}\otimes\Pi_{-\ell-1}(-\nu_{u-r,1})\nonumber\\
   &\cong (M^k_{u-r,1}\otimes\Pi_{\ell+1}(\nu_{u-r,1}))^*\cong M_{r,v,\ell}^*\cong\cF(\sigma^{\ell}(\cD_{r,v-1}^+))^*.
\end{align*}
Thus $\cF(\sigma^{\ell}(\cD_{r,v-1}^+)')$ and $\cF(\sigma^{\ell}(\cD_{r,v-1}^+))^*$ are isomorphic simple objects of $\cC_A$, hence any non-zero morphism between them is an isomorphism.    
\end{proof}

Theorems \ref{thm:indsimple1} and \ref{thm:ind-simple-atypical} complete Step 3 of the strategy, so condition (2) in Theorem \ref{rigidity-of-C-from-CAloc} holds for $\cC=\cC^{\mathrm{wt}}_k(\sl_2)$, $A=\Vir_{c_k}\otimes\Pi(0)$ when $k=-2+\frac{u}{v}$, $v\geq 3$. Thus we conclude:
\begin{theorem}\label{thm:rigidity-for-v>2}
    The tensor category $\cC^{\mathrm{wt}}_k(\sl_2)$ is rigid when $k=-2+\frac{u}{v}$ for coprime $u\in\ZZ_{\geq 2}$, $v\in\ZZ_{\geq 3}$.
\end{theorem}

\end{document}